\documentclass[11pt]{article}
\usepackage{amsmath}
\usepackage{pifont}
\usepackage[title]{appendix}

\usepackage{amsfonts}
\usepackage{amssymb,amsmath,amsthm, amsfonts}

\def\sqr#1#2{{\vcenter{\vbox{\hrule height.#2pt
				\hbox{\vrule width.#2pt height#1pt \kern#1pt \vrule width.#2pt}
				\hrule height.#2pt}}}}
\def\sqr#1#2{{\vcenter{\vbox{\hrule height.#2pt
              \hbox{\vrule width.#2pt height#1pt \kern#1pt \vrule width.#2pt}
              \hrule height.#2pt}}}}
\def\3n{\negthinspace \negthinspace \negthinspace }
\def\2n{\negthinspace \negthinspace }
\def\1n{\negthinspace }

\def\={\buildrel \triangle \over =}

\def\esssup{\mathop{\rm esssup}}

\def\max{\mathop{\rm max}}

\def\exp{\mathop{\rm exp}}
\def\sup{\mathop{\rm sup}}
\def\inf{\mathop{\rm inf}}

\def\({\Big (}
\def\){\Big )}
\def\[{\Big[}
\def\]{\Big]}
\def\be{\begin{equation}}

\def\ee{\end{equation}}

\def\square#1{\vbox{\hrule\hbox{\vrule height#1%
     \kern#1\vrule}\hrule}}
\def\rectangle#1#2{\vbox{\hrule\hbox{\vrule height#1%
     \kern#2\vrule}\hrule}}

\font\tenbb=msbm10 \font\sevenbb=msbm7 \font\fivebb=msbm5

\newfam\bbfam
\scriptscriptfont\bbfam=\fivebb \textfont\bbfam=\tenbb
\scriptfont\bbfam=\sevenbb

\newtheorem{lemma}{Lemma}[section]
\newtheorem{remark}{Remark}[section]
\newtheorem{example}{Example}[section]
\newtheorem{theorem}{Theorem}[section]
\newtheorem{corollary}{Corollary}[section]

\newtheorem{definition}{Definition}[section]
\newtheorem{proposition}{Proposition}[section]

\makeatletter
   
   \@addtoreset{equation}{section}
\makeatother
\begin{document}
\title{Derivative over Wasserstein spaces along curves of densities\footnotemark[1]}
 \author{Rainer Buckdahn$^{1,2}$,\,\, Juan Li$^{3,**}$,\,\, Hao Liang$^3$ \\
 {$^1$\small Laboratoire de Math\'{e}matiques de Bretagne Atlantique, Univ Brest,}\\
	{\small UMR CNRS 6205, 6 avenue Le Gorgeu, 29200 Brest, France.}\\
	{$^2$\small  School of Mathematics and Statistics, Shandong University,}
	{\small Jinan 250100, P. R. China.}\\
	{$^3$\small  School of Mathematics and Statistics, Shandong University, Weihai,}
	{\small Weihai 264209, P. R. China.}\\
 {\small{\it E-mails: rainer.buckdahn@univ-brest.fr,\,\ juanli@sdu.edu.cn,\,\ haoliang@mail.sdu.edu.cn.}}
 \date{October 04, 2020}}
\renewcommand{\thefootnote}{\fnsymbol{footnote}}
\footnotetext[1]{The work has been supported by the NSF of P.R. China (No. 12031009, 11871037), National Key R and D Program of China (NO. 2018YFA0703900), NSFC-RS (No. 11661130148; NA150344),  and also supported by the ``FMJH Program Gaspard Monge in optimization and operation research", and the ANR (Agence Nationale de la Recherche) project ANR-16-CE40-0015-01. $^{**}$ Corresponding author.}
\maketitle
\begin{abstract}
 In this paper, given any random variable $\xi$ defined over a probability space $(\Omega,\mathcal{F},Q)$, we focus on the study of the derivative of functions of the form $L\mapsto F_Q(L):=f\big((LQ)_{\xi}\big),$ defined over the convex cone of densities $L\in\mathcal{L}^Q:=\{ L\in L^1(\Omega,\mathcal{F},Q;\mathbb{R}_+):\ E^Q[L]=1\}$ in $L^1(\Omega,\mathcal{F},Q).$ Here $f$ is a function over the space $\mathcal{P}(\mathbb{R}^d)$ of probability laws over $\mathbb{R}^d$ endowed with its Borel $\sigma$-field $\mathcal{B}(\mathbb{R}^d)$.
The problem of the differentiability of functions $F_Q$ of the above form has its origin in the study of mean-field control problems for which the controlled dynamics admit only weak solutions. Inspired by P.-L. Lions' results \cite{Lions} we show that, if for given $L\in\mathcal{L}^Q$,  $L'\mapsto F_{LQ}(L'):\mathcal{L}^{LQ}\rightarrow\mathbb{R}$ is differentiable at $L'=1$, the derivative is of the form $g(\xi)$, where $g:\mathbb{R}^d\rightarrow\mathbb{R}$ is a Borel function which depends on $(Q,L,\xi)$ only through the law $(LQ)_\xi$. Denoting this derivative by $\partial_1F((LQ)_\xi,x):=g(x),\, x\in\mathbb{R}^d$, we study its properties, and we relate it to partial derivatives, recently investigated in \cite{BCL}, and, moreover, in the case when $f$ restricted to the 2-Wasserstein space $\mathcal{P}_2(\mathbb{R}^d)$ is differentiable in P.-L. Lions' sense and $(LQ)_{\xi}\in\mathcal{P}_2(\mathbb{R}^d)$, we investigate the relation between the derivative with respect to the density of $F_Q(L)=f\big((LQ)_{\xi}\big)$ and the derivative of $f$ with respect to the probability measure. Our main result here shows that $\partial_x\partial_1F((LQ)_\xi,x)=\partial_\mu f((LQ)_\xi,x),\ x\in \mathbb{R}^d,$ where $\partial_\mu f((LQ)_\xi,x)$ denotes the derivative of $f:\mathcal{P}_2(\mathbb{R}^d)\rightarrow \mathbb{R}$ at $(LQ)_\xi$.
By using techniques of Girsanov transformation in the context of Malliavin calculus (See \cite{B94}), this result will be proved first for smooth Wiener functionals $L$ and $\xi$, and will be then extended by subtle approximation arguments to the general case. Finally, the link with A. Bensoussan's approach in \cite{BFY15} based on functions of densities of random variables instead of their law is studied.
\end{abstract}



\section{Introduction}

Stimulated by the seminal works of Lasry and Lions \cite{LL07}, by Lions \cite{Lions} (Refer also to Cardaliaguet \cite{C12}) but also by Huang, Caines and Malham\'{e} \cite{HCM07} as well as by a large manifold of applications, in recent years the study of mean-field problems has been enjoying a great interest and attracting numerous researchers. Among the different approaches let us mention those which characterise the limit problem of a large number of coupled equations (dynamics of a game associated with different players, systems of backward stochastic differential equations or also systems of Hamilton-Jacobi-Bellman equations with coupling) as mean-field game or as the so-called Master equation.
Such limit problems were studied, for instance, in the pioneering works \cite{Lions} and \cite{HCM07} already cited above, but also in the works by Bensoussan, Frehse and Yam \cite{BFY15}, \cite{BFY17}, by Cardaliaguet \cite{C17}, and also in that by Buckdahn, Djehiche, Li and Peng \cite{BDLP09}. With the same motivation Li and Min \cite{LM16} studied weak solutions for stochastic mean-field equations. In his course at Coll\`{e}ge de France \cite{Lions} (for the written version please refer to \cite{C12}) P.-L. Lions introduced and studied the innovative notion of derivative over Wasserstein spaces. Strongly motivated by these works Buckdahn, Li, Peng and Rainer \cite{BLPR17} investigated this derivative for stochastic mean-field equations and associated with them mean-field PDEs. This notion of derivative over Wasserstein spaces has turned out to be crucial in the study of the stochastic maximum principle (SMP) for stochastic control problems with coefficients depending not only on the controlled state process and the control, but also on their law. There are numerous papers by different authors investigating the stochastic optimal control problems for controlled mean-field stochastic differential equations with strong solution, e.g., \cite{BCP18}, \cite{BDL11aD}, \cite{BLM16aD}, \cite{Li12} and \cite{PW18}, and also particular cases of mean-field stochastic control systems with only weak solution, e.g., \cite{ABC19} and \cite{BLM17}.
Bayraktar, Cosso and Pham \cite{BCP18} focused on control problems driven by McKean-Vlasov dynamics where the diffusion coefficient may be degenerate and they obtained a probabilistic representation of its value function. They provided with so-called randomization method to formulate a new weak control problem over a set of equivalent probability measures which can be shown to have the same value function as the original one, and derived the dynamic programming principle without assuming controls of feedback type. Relying on the notion of derivative over Wasserstein spaces mentioned above, Pham and Wei \cite{PW18} and Acciaio, Backhoff-Veraguas and Carmona \cite{ABC19} both considered the case of dynamics depending on the joint law of the state and the control. Pham and Wei \cite{PW18} studied the dynamic programming Bellman equation in the framework of feedback controls.  Acciaio, Backhoff-Veraguas and Carmona \cite{ABC19} proved suitable versions of the Pontryagin stochastic maximum principle.

Our original starting point for the present work was the study of the SMP for a rather general mean-field control system with weak solution, but this investigation leads to another fundamental problem, that of the study of the derivative of functions $F_Q:\mathcal{L}^Q\rightarrow \mathbb{R}$ defined over the space $
\mathcal{L}^Q:=\{L\in L^1(\Omega,\mathcal{F},Q;\mathbb{R}_+): E^Q[L]=1\}$ of densities on a probability space $(\Omega,\mathcal{F},Q)$, so that the present manuscript focuses on the study of the derivative of such functions of densities, while the study of a general mean-field control problem involving such derivatives is the object of a forthcoming work.

Let us explain this link between stochastic mean-field control problems admitting only a weak solution and the derivative of functions over $\mathcal{L}^Q$ with the help of a simple example.

\smallskip

Let $b:\mathbb{R}^2\rightarrow \mathbb{R}$ be a bounded Borel function. Given a measurable, non anticipating functional $u:[0,T]\times C([0,T];\mathbb{R})\rightarrow \mathbb{R}$, we consider over a suitable probability space $(\Omega,\mathcal{F},Q)$ the stochastic differential equation (SDE)
\begin{equation}\label{eq1a}
dX_t^{u}=b(X_t^{u},u_t)dt+dB_t,\ t\in[0,T],\ X_0^{u}=0,
\end{equation}
where $B=(B_t)_{t\in[0,T]}$ is a $Q$-Brownian motion and $u_t=u(t,X_{\cdot\wedge t}^{u}),\ t\in [0,T]$. Of course, such an SDE admits, in general, only a weak solution, which is obtained by putting $X^{u}:=B$, $L^{u}:=\exp\{\int_0^Tb(X^{u}_t,u_t)dB_t- \frac12\int_0^Tb(X^{u}_t,u_t)^2dt\}$ and $B_t^{u}=B_t-\int_0^tb(X_s^{u},u_s)ds,\, t\in[0,T].$ Indeed, due to the Girsanov Theorem, $B^{u}=(B^{u}_t)_{t\in[0,T]}$ is an $Q^{u}:=L^{u}Q$-Brownian motion and, for $\mathbb{F}$ denoting the filtration generated by $B$, $(\Omega,\mathcal{F},\mathbb{F},Q^{u},B^{u},X^{u})$ is a weak solution of (\ref{eq1a}),
\begin{equation}\label{eq1b}
	dX_t^{u}=b(X_t^{u},u_t)dt+dB_t^{u},\, t\in[0,T],\, X_0^{u}=0.
\end{equation}
\noindent We consider as set of admissible controls $u_t=u(t,B_{\cdot\wedge t}),\ t\in[0,T],$ the space $\mathcal{U}_{ad}:=L^\infty_{\mathbb{F}}([0,T]\times\Omega;dsdQ)$, and we endow the controlled equation (\ref{eq1b}) with a cost functional $J:\mathcal{U}_{ad}\rightarrow\mathbb{R}$. Let $u^*\in\mathcal{U}_{ad}$ be an optimal control, $J(u^*)\le J(u),\ u\in \mathcal{U}_{ad}$. Studying the stochastic maximum principle means to provide necessary and possibly sufficient optimality conditions for the control $u^*$. A way to do this consists in considering the linear perturbation $u^\varepsilon:=u^*+\varepsilon u, \varepsilon\in \mathbb{R},\ u\in\mathcal{U}_{ad},$ and to use, supposing differentiability, that $\partial_\varepsilon J(u^\varepsilon)\big|_{\varepsilon=0}=0$. In the classical case, when $J$ is a terminal cost functional of the form
$J(u):=\displaystyle E^{Q^u}[\varphi(X_T^{u})]=E^Q[L^{u}\varphi(B_T)],\ u\in\mathcal{U}_{ad},$ we have

\smallskip

\centerline{$\partial_\varepsilon J(u^\varepsilon)\big|_{\varepsilon=0}=\displaystyle E^Q[\varphi(B_T)\partial_\varepsilon L^{u^\varepsilon} \big|_{\varepsilon=0}],$}

\smallskip

\noindent and the study of the SMP is classical for this case. However, recently numerous authors have studied the SMP for stochastic mean-field control problems in which the coefficients do not depend only on the control state process and its control, but also on the law of the controlled state process. In our case this leads to a cost functional of the form $J(u)=f(Q^{u}_{B_T}), u\in\mathcal{U}_{ad},$ and studying the SMP

\smallskip

\centerline{$\partial_\varepsilon J(u^\varepsilon)\big|_{\varepsilon=0}=\partial_\varepsilon f(Q^{u^\varepsilon}_{B_T})\big|_{\varepsilon=0}=\partial_\varepsilon f((L^{u^\varepsilon}Q)_{B_T})\big|_{\varepsilon=0},$}

\smallskip

\noindent  is tightly related with the investigation of the derivative w.r.t. the density $L\in \mathcal{L}^Q$ of functions $F_Q:\mathcal{L}^Q\rightarrow\mathbb{R}$ of the form $F_Q(L):=f((LQ)_\xi),$ where $f$ is a real-valued function defined over the space $\mathcal{P}(\mathbb{R}^d)$ of probability laws over $(\mathbb{R}^d,\mathcal{B}(\mathbb{R}^d))$ and $\xi:(\Omega,\mathcal{F},Q)\rightarrow (\mathbb{R}^d,\mathcal{B}(\mathbb{R}^d))$ is an arbitrarily fixed random variable.

The objective of our manuscript is the study of the derivative of such functions  $L\rightarrow F_Q(L):=f((LQ)_\xi)$ w.r.t. the density. In order to understand better the setting in which these studies shall be done, let us fix for simplicity any deterministic function $h=(h_t)_{t\in[0,T]}\in L^2([0,T],dt)$ and consider the regular $\mathcal{L}^Q$-valued curve $\varepsilon\rightarrow L^\varepsilon=\exp\{\int_0^T\varepsilon h_sdB_s-\frac12\int_0^T(\varepsilon h_s)^2ds\},\, \varepsilon\in \mathbb{R}$. Let $\xi=\varphi(B_T)$, where $\varphi\in C^1_b(\mathbb{R})$. Then, due to the Girsanov Theorem, $(L^\varepsilon Q)_\xi=Q_{\varphi(B_T+\varepsilon\int_0^Th_tdt)},\ \varepsilon\in\mathbb{R}$, and for any function $f:\mathcal{P}_2(\mathbb{R})\rightarrow \mathbb{R}$ differentiable in the sense of P-L. Lions \cite{Lions} (See also \cite{BLPR17}) over the space $\mathcal{P}_2(\mathbb{R})$ of probability laws with finite second moment, with derivative $\partial_\mu f:\mathcal{P}_2(\mathbb{R})\times\mathbb{R}\rightarrow \mathbb{R}$ (See Subsection 3.3), we have
\begin{equation}\label{eq1d}
\begin{array}{lll}
& &\partial_\varepsilon F_Q(L^\varepsilon)\big|_{\varepsilon=0}=\displaystyle
\partial_\varepsilon f((L^\varepsilon Q)_\xi)\big|_{\varepsilon=0}
=\partial_\varepsilon f\big(Q_{\varphi(B_T+\varepsilon\int_0^Th_tdt)}\big) \big|_{\varepsilon=0}\\
&=&\displaystyle E^Q\big[\partial_\mu f(Q_{\varphi(B_T)},\varphi(B_T))\partial_x\varphi(B_T)\int_0^Th_tdt\big]
=\displaystyle E^Q\big[\int_0^{\varphi(B_T)}\partial_\mu f(Q_{\varphi(B_T)},y)dy\cdot \int_0^Th_tdB_t\big]\\
&=&\displaystyle E^Q\big[\int_0^{\xi}\partial_\mu f(Q_{\xi},y)dy\cdot \partial_\varepsilon L^\varepsilon\big|_{\varepsilon=0}\big],
\end{array}
\end{equation}
where the last equality but one is a consequence of the integration by parts formula in the Malliavin calculus. The above relation (\ref{eq1d}) extended to general curves in $\mathcal{L}^Q$ will allow to conclude (See Subsection 3.3) that the derivative $DF_Q(1)$ of $F_Q$ w.r.t. the measure at $L=1$ is given by

\smallskip

\centerline{$\displaystyle DF_Q(1)=\int_0^{\xi}\partial_\mu f(Q_{\xi},y)dy-E^Q\big[\int_0^{\xi}\partial_\mu f(Q_{\xi},y)dy\big],$}

\smallskip

\noindent  and that there is a measurable function $g:\mathbb{R}\rightarrow \mathbb{R}$ only depending on $\xi$ through $Q_\xi$, such that $DF_Q(1)=g(\xi),\ Q$-a.s. Denoting $g$ by $\partial_1F(Q_\xi,\cdot)$, we have
\begin{equation}\label{eq1aa}
\partial_1F(Q_\xi,x)=\int_0^{x}\partial_\mu f(Q_{\xi},y)dy-E^Q\big[\int_0^{\xi}\partial_\mu f(Q_{\xi},y)dy\big],\, x\in\mathbb{R},
\end{equation}
and in particular, we get that
\begin{equation}\label{eq1e}
\partial_x\big(\partial_1F(Q_\xi,x)\big)=\partial_\mu f(Q_{\xi},x),\, x\in \mathbb{R}.
\end{equation}
We will see that formula (\ref{eq1e}) holds true in a rather general setting, also for multi-dimensional random variables $\xi$. The interest of this formula stems from the fact that it allows to regard the derivative w.r.t. the density as a kind of first order derivative, and that w.r.t. the measure as a second order one.

\smallskip

The above example proves the interest of investigating the derivative of functions $L\rightarrow F_Q(L):=f((LQ)_\xi),\, L\in\mathcal{L}^Q,$ for arbitrarily fixed $\xi\in L^0(\Omega,\mathcal{F},Q;\mathbb{R}^d),$ and its tight relation with the derivative of $f:\mathcal{P}_2(\mathbb{R}^d)\rightarrow\mathbb{R}$ studied by Lions \cite{Lions}.

Let us combine the explanation of the organisation of the paper with that of our main results. After the preliminaries in Section 2 in which we recall some basic notions, and in particular that of the derivative of a function $f:\mathcal{P}_2(\mathbb{R}^d)\rightarrow\mathbb{R}$ w.r.t. the probability measure, Section 3 is devoted to the study of the derivative of functions $F_Q:\mathcal{L}^Q\rightarrow\mathbb{R}$ of the form $F_Q(L):=f((LQ)_\xi),\, L\in\mathcal{L}^Q,$ for arbitrarily fixed $\xi\in L^0(\Omega,\mathcal{F},Q;\mathbb{R}^d),$ and for functions $f$ defined over the space $\mathcal{P}(\mathbb{R}^d)$ of the probability measures over $(\mathbb{R}^d,\mathcal{B}(\mathbb{R}^d))$. If $F_{Q}$ is differentiable in $L$ (See Definition \ref{Def3.1}), its unique derivative $DF_Q(L)\in L(L^1_0(\Omega,\mathcal{F},Q),\mathbb{R})$(\footnote[1]{$L^p_0(\Omega, \mathcal{F},Q)$ is the sub-space of random variables $\eta$ in $L^p(\Omega,\mathcal{F},Q)$ with $E^Q[\eta]=0$ $(1\le p\le +\infty).$}) is identified with $DF_Q(L)\in  L^\infty_0(\Omega,\mathcal{F},Q)$. The choice of the special form of the function $F_Q(L):=f((LQ)_\xi)$ implies that $F_{Q}(L'L)=F_{LQ}(L'),\ L\in\mathcal{L}^Q,\ L'\in\mathcal{L}^{LQ}$, which allows to reduce our study of the derivative essentially to that of $F_{LQ}:\mathcal{L}^{LQ}\rightarrow \mathbb{R}$ at $L'=1$. In analogy to P.-L. Lions' result for the derivatives over $\mathcal{P}_2(\mathbb{R}^d)$, Theorem \ref{P3} states that there exists a Borel function $g:\mathbb{R}^d\rightarrow\mathbb{R}$ such that $DF_{LQ}(1)=g(\xi),\ Q$-a.s., and $g$ depends on $(Q,L,\xi)$ only through $(LQ)_\xi$, which motivates the notation $\partial_1F((LQ)_\xi,x):=g(x),\ x\in\mathbb{R}^d.$
Further properties of $\partial_1F$ are studied in Subsection 3.1 and an illustrating example is given. Subsection 3.2 is devoted to the investigation of the link between the derivative of $L\rightarrow F_{Q}(L)$ and the partial derivative of $L\rightarrow G(Q_{(L,\xi)}):=f((LQ)_\xi)$. Given any measurable space $(E,\mathcal{E})$, partial derivatives w.r.t. the conditional law $Q_{L|\xi}$ at $Q_{(L,\xi)}$ of functions
$Q_{(L,\xi)}\rightarrow H(Q_{(L,\xi)}):\mathcal{P}_{2,0}(\mathbb{R}^d\times E)\rightarrow \mathbb{R}$ (\footnote[2]{ $\mathcal{P}_{2,0}(\mathbb{R}^d\times E)=\{\mu\in\mathcal{P}(\mathbb{R}^d\times E):\ \int_{\mathbb{R}^d}|x|^2\mu(dxde)<+\infty\}$.}) were recently studied in \cite{BCL}, and the relation $(LQ)_\xi=\int_{E}eQ_{(L,\xi)}(dedx)$ allows to re-interpret the results of Subsection 3.1 in terms of the notion of partial derivative. Also here an illustrating example is given. Finally, Subsection 3.3 is devoted to the study of the derivative of $F_Q(L)=f((LQ)_\xi)$, when $\xi\in L^2(\Omega,\mathcal{F},Q;\mathbb{R}^d)$ and $f:\mathcal{P}_2(\mathbb{R}^d)\rightarrow\mathbb{R}$ is differentiable in P.-L. Lions' sense. In Theorem 3.2, generalising (\ref{eq1aa}), we give for this general case the explicit form of the derivative $\partial_1f((LQ)_\xi,x)$, when $\xi$ is a real-valued random variable, and for $\xi$ an $\mathbb{R}^d$-valued random variable ($d\ge 1$) we show that we have $\partial_x\partial_1F((LQ)_\xi,x)=\partial_\mu f((LQ)_\xi,x),\, x\in\mathbb{R}^d$. Using techniques of Girsanov transformation in the frame of the Malliavin calculus (See \cite{B94}), we prove the results first in the case, where $\xi$ and $L$ are regular in the sense of the Malliavin derivative, and then, by a suitable, rather subtle approximation (See Proposition \ref{lem2}) we extend these results to the general case. The proof of Proposition \ref{lem2} is given in the Appendix (Section 4). Finally, at the end of Subsection 3.3 we relate the mean-field approach based on densities of random variables by Bensoussan (See, e.g., \cite{BFY15}) to the results of our manuscript. Considering a sufficiently regular function $\Phi$ defined over the space of probability densities in $L^1(\mathbb{R},\mathcal{B}(\mathbb{R}),dx)$ and defining $F_Q(L):=f((LQ)_\xi):=\Phi(f_\xi^{LQ}),$ for a random variable $\xi\in L^2(\Omega,\mathcal{F},Q)$ and  $L\in\mathcal{L}^Q$ such that $(LQ)_\xi$ has a square integrable density denoted by $f_\xi^{LQ}$, we show that, under suitable assumptions, $(LQ)_\xi(dx)$-a.s.,

\smallskip

\centerline{$\begin{array}{lll}
\partial_1F((LQ)_\xi,x)&=&D\Phi(f_\xi^{LQ},x)-E^{LQ} [D\Phi(f_\xi^{LQ},\xi)],\\
\partial_\mu f((LQ)_\xi,x)&=&\partial_x\partial_1F((LQ)_\xi,x)=\partial_x D\Phi(f_\xi^{LQ},x).	
\end{array}$}

\noindent Here $D\Phi(f_\xi^{LQ})(\cdot)\in L(L^2_0(\mathbb{R},\mathcal{B}(\mathbb{R}),dx),\mathbb{R})$ denotes the $L^2(dx)$-derivative of $\Phi$ and is identified with $D\Phi(f_\xi^{LQ},\cdot)\in L^2_0(\mathbb{R},\mathcal{B}(\mathbb{R}),dx)$.

\section{Preliminaries}

Let us assume that $(E,d)$ is a separable complete metric space; $\mathcal{B}(E)$ denotes the Borel $\sigma$-field over $(E,d)$, and $\mathcal{P}(E)$ is the space of all probability measures endowed with the topology of weak convergence. For $p\geq 1$, we consider the space of probability measures on $(E,\mathcal{B}(E))$ with finite $p$-th moment, denoted by $\mathcal{P}_p(E)$,
\begin{equation}
\mathcal{P}_p(E):=\Big\{\mu\in\mathcal{P}(E)\Big|\int_E d(z,z_0)^p\mu(dz)<+\infty, \mbox{ for some } z_0\in E \Big\},
\end{equation}
or equivalent,
\begin{equation}
\mathcal{P}_p(E):=\Big\{\mu\in\mathcal{P}(E)\Big|\int_E d(z,z_0)^p\mu(dz)<+\infty, \mbox{ for all } z_0\in E \Big\}.
\end{equation}
The space $\mathcal{P}_p(E)$ is endowed with the $p$-Wasserstein metric
\begin{equation*}
W_{p}(\mu,\nu):=\inf\Big\{\left(\int_{E\times E}(d(z,z'))^p\rho(dzdz')\right)^{\frac{1}{p}}\Big|\ \rho\in\mathcal{P}_p(E\times E)\mbox{ with }
                      \rho(\cdot\times E)=\mu,\ \rho(E\times\cdot)=\nu\Big\},
\end{equation*}
where $\mu,\nu\in\mathcal{P}_p(E)$.

Let $(\Omega,\mathcal{F},P)$ be a complete probability space which is ``rich enough'' in the sense that, for all $k\ge 1$, $\mathcal{P}(E^k)=\{P_\vartheta,\  \vartheta\in L^0(\Omega,\mathcal{F},P;E^k)\}$, where $L^0(\Omega,\mathcal{F},P;E^k)$ (resp., $L^p(\Omega,\mathcal{F},P;E^k)$) denotes the space of all $E^k$-valued random variables over $(\Omega,\mathcal{F},P)$ (resp., that of all $E^k$-valued random variables over $(\Omega,\mathcal{F},P)$ with finite $p$-th moment). Then, obviously, for $\mu,\nu\in\mathcal{P}_p(E^k)$, the $p$-Wasserstein metric can be rewritten as follows:
\begin{equation*}
W_{p}(\mu,\nu):=\inf\big\{\big(E[|\xi-\eta|^p]\big)^{\frac{1}{p}}\big|\ \xi,\ \eta\in L^p(\Omega,\mathcal{F},P;E^k) \mbox{ with } P_\xi=\mu,\ P_\eta=\nu\big\}.
\end{equation*}

Let now $p=2$ and $E=\mathbb{R}^d\, (d\ge 1)$;
with the identification of the random variables which coincide $P$-a.s.  $L^{2}(\Omega,\mathcal{F},P;\mathbb{R}^d)$ is a real Hilbert space with inner product $(\xi,\eta)_{L^{2}(P)}=E[\xi\cdot\eta],\ \xi,\eta\in L^{2}(\Omega,\mathcal{F},P;\mathbb{R}^d)$, and norm $|\xi|_{L^{2}(P)}=(\xi,\xi)_{L^{2}(P)}^{\frac{1}{2}}$. We say that a function $f:\mathcal{P}_{2}(\mathbb{R}^d)\rightarrow\mathbb{R}$ is differentiable at $\mu\in\mathcal{P}_{2}(\mathbb{R}^d)$ if, for the lifted function $\widetilde{f}(\vartheta):=f(P_{\vartheta}),\ \vartheta\in L^{2}(\Omega,\mathcal{F},P;\mathbb{R}^d)$,
there is some $\vartheta_{0}\in L^{2}(\Omega,\mathcal{F},P;\mathbb{R}^d)$ with $P_{\vartheta_{0}}=\mu$, such that the function $\widetilde{f}:L^{2}(\Omega,\mathcal{F},P;\mathbb{R}^d)\rightarrow\mathbb{R}$ is Fr\'{e}chet differentiable in $\vartheta_{0}$,
i.e., there exists a continuous linear mapping $D\widetilde{f}(\vartheta_{0}):L^{2}(\Omega,\mathcal{F},P;\mathbb{R}^d) \rightarrow \mathbb{R}$ (i.e., $D\widetilde{f}(\vartheta_{0})\in L(L^{2}(\Omega,\mathcal{F},P;\mathbb{R}^d);\mathbb{R})$) such that
\begin{equation}
\widetilde{f}(\vartheta_{0}+\eta)-\widetilde{f}(\vartheta_{0})=D\widetilde{f}(\vartheta_{0})(\eta)+o(|\eta|_{L^{2}(P)}),
\end{equation}
with $|\eta|_{L^{2}(P)}\rightarrow0$, for $\eta\in L^{2}(\Omega,\mathcal{F},P;\mathbb{R}^d)$. Since $D\widetilde{f}(\vartheta_{0})\in L(L^{2}(\Omega,\mathcal{F},P;\mathbb{R}^d);\mathbb{R})$,
Riesz' Representation Theorem yields the existence of a ($P$-a.s.) unique random variable $\theta_{0}\in L^{2}(\Omega,\mathcal{F},P;\mathbb{R}^d)$ such that $D\widetilde{f}(\vartheta_{0})(\eta)=(\theta_{0},\eta)_{L^{2}(P)}=E[\theta_{0}\eta]$,
for all $\eta\in L^{2}(\Omega,\mathcal{F},P;\mathbb{R}^d)$. In \cite{C12} it was proven that there is a Borel function $h_{0}:\mathbb{R}^d\rightarrow\mathbb{R}^d$ such that $\theta_{0}=h_{0}(\vartheta_{0})$, $P$-a.s. The function $h_0$ was shown to depend on $\vartheta_0$ only through its law $P_{\vartheta_0}$.
Taking into account the definition of $\widetilde{f}$, this allows to write
\begin{equation}\label{equ2.4}
f(P_{\vartheta})-f(P_{\vartheta_{0}})=E[h_{0}(\vartheta_{0})\cdot(\vartheta-\vartheta_{0})]+o(|\vartheta-\vartheta_{0}|_{L^{2}(P)}),
\end{equation}
for $\vartheta\in L^{2}(\Omega,\mathcal{F},P;\mathbb{R}^d)$ with $|\vartheta-\vartheta_{0}|_{L^{2}(P)}\rightarrow 0 $.

We call $\partial_{\mu}f(P_{\vartheta_{0}},y):=h_{0}(y),\ y\in\mathbb{R}^d$, the derivative of $f:\mathcal{P}_{2}(\mathbb{R}^d)\rightarrow\mathbb{R}$ at $P_{\vartheta_{0}}$.

For simplicity in our approach we will consider the function $\widetilde{f}:L^{2}(\Omega,\mathcal{F},P;\mathbb{R}^d)\rightarrow\mathbb{R}$ as Fr\'{e}chet differentiable over the whole space $L^{2}(\Omega,\mathcal{F},P;\mathbb{R}^d)$. In this case we have the derivative $\partial_{\mu}f(P_{\vartheta},y)$ defined $P_{\vartheta}(dy)$-a.e., for all $\vartheta\in L^{2}(\Omega,\mathcal{F},P;\mathbb{R}^d)$
through the relation $D\widetilde{f}(\vartheta)(\eta)=E[\partial_{\mu}f(P_{\vartheta},\vartheta)\eta]$, for all $\eta\in L^{2}(\Omega,\mathcal{F},P;\mathbb{R}^d)$.

As we have seen, in the above approach the definition of the derivative of a function $f:\mathcal{P}_2(\mathbb{R}^d)\rightarrow\mathbb{R}$ defined through the Fr\'{e}chet derivative of its lifted function $\widetilde{f}:L^2(\Omega,\mathcal{F},P;\mathbb{R}^d)\rightarrow \mathbb{R}$ is, although independent of the special choice of $P$, heavily depending on the fact that the same probability $P$ is used while working with the derivative of $f$. However, as we will see, working with weak solutions of a stochastic controlled problem will lead us to derivatives of functions of the form $\theta\rightarrow f(\mu^\theta)$, where $\mu^\theta=\big(L^{u^\theta}P\big)_{X_t^{u^\theta}}$. This makes necessary to study in addition to the Fr\'echet derivative of $L^2(\Omega,\mathcal{F},LP;\mathbb{R}^d)\ni\xi\rightarrow f\big((LP)_\xi\big)$ also that w.r.t. the density $L\rightarrow f\big((LP)_\xi\big)$.

\section{Derivative with respect to the density}
\subsection{Derivative w.r.t. the density and properties}

Let $(\Omega,\mathcal{B}(\Omega))$ be a Radon space, and $Q$ be a probability measure on $(\Omega,\mathcal{B}(\Omega))$. Putting $\mathcal{F}=\mathcal{B}(\Omega)\vee\mathcal{N}_Q$, where $\mathcal{N}_Q$ denotes the collection of all $Q$-null sets, $(\Omega,\mathcal{F},Q)$ is a complete probability space. For $p\in[1,\infty]$ we define the spaces $L^p_0(\Omega,\mathcal{F},Q):=\{\eta\in L^p(\Omega,\mathcal{F},Q)\,|\, E^Q[\eta]=0\}$ and
$\mathcal{L}^Q:=\{L\in L^1(\Omega,\mathcal{F},Q)\,|\, L>0,\, E^Q[L]=1\}$. We also introduce the notation $\mathcal{L}^Q-\mathcal{L}^Q:= \{L-L'\,|\,L,L'\in\mathcal{L}^Q\}$. Then we have the following lemma.

\begin{lemma}\label{Lemma1}
{\rm (i)} $L_0^1(\Omega,\mathcal{F},Q)=\{\alpha\eta\, |\, \alpha>0,\,\eta\in\mathcal{L}^Q-\mathcal{L}^Q\}$;\\
{\rm (ii)} For all $\varphi\in L(L_0^1(\Omega,\mathcal{F},Q);\mathbb{R})$, there exists some ($Q$-a.s.) unique $\vartheta_\varphi\in L^\infty_0(\Omega,\mathcal{F},Q)$ such that
$$ \varphi(\eta)=E^Q[\vartheta_\varphi \eta],\, \,  \eta\in L_0^1(\Omega,\mathcal{F},Q). $$
\end{lemma}
\begin{proof}
(i) As the inclusion ``$\supset$'' is obvious, we only need to prove  ``$\subset$''. Decomposing $\theta\in L_0^1(\Omega,\mathcal{F},Q)$ in its positive and its negative parts, we have $\theta=\theta^+-\theta^-$. Then, $E^Q[\theta^+]-E^Q[\theta^-]=E^Q[\theta]=0$. Thus, putting $L:=\frac{1+\theta^+}{1+E^Q[\theta^+]}$ and $L':=\frac{1+\theta^-}{1+E^Q[\theta^-]}$, we see that $L,\ L'\in\mathcal{L}^Q$, and for $\alpha=1+E^Q[\theta^+]>0$, we have $\theta=\alpha(L-L')$.\\
(ii) It is well-known that, for all $\psi\in L(L^1(\Omega,\mathcal{F},Q);\mathbb{R})$, there exists a ($Q$-a.s.) unique $\theta_\psi\in L^\infty(\Omega,\mathcal{F},Q)$ such that $\psi(\eta)=E^Q[\theta_\psi \eta],$ for all $\eta\in L^1(\Omega,\mathcal{F},Q)$. Let now $\varphi\in L(L_0^1(\Omega,\mathcal{F},Q);\mathbb{R})$. We extend $\varphi$ to $L^1(\Omega,\mathcal{F},Q)$ by putting
$ \widetilde{\varphi}(\eta):=\varphi(\eta-E^Q[\eta]),\ \eta\in L^1(\Omega,\mathcal{F},Q)$. Obviously, $\widetilde{\varphi}\in L(L^1(\Omega,\mathcal{F},Q);\mathbb{R})$. Thus, there exists $\theta_{\widetilde{\varphi}}\in L^\infty(\Omega,\mathcal{F},Q)$ such that $\widetilde{\varphi}(\eta)=E^Q[\theta_{\widetilde{\varphi}}\eta],$ for all $\eta\in L^1(\Omega,\mathcal{F},Q)$. Hence, for all $\eta\in L^1_0(\Omega,\mathcal{F},Q)$, it holds that
$$ \varphi(\eta)=\widetilde{\varphi}(\eta)=E^Q[\theta_{\widetilde{\varphi}}\eta]=E^Q[(\theta_{\widetilde{\varphi}}-E^Q[\theta_{\widetilde{\varphi}}])\eta] .$$
Obviously, $\vartheta_\varphi:=\theta_{\widetilde{\varphi}}-E^Q[\theta_{\widetilde{\varphi}}]\in L^\infty_0(\Omega,\mathcal{F},Q)$ and $\widetilde{\varphi}(\eta)=E^Q[\vartheta_{\varphi}\eta],$ for all $\eta\in L^1_0(\Omega,\mathcal{F},Q)$. The ($Q$-a.s.) uniqueness of $\vartheta_\varphi$ follows from a standard argument, similar to that of $\theta_{\widetilde{\varphi}}$.
\end{proof}

Let now $\xi\in L^0(\Omega,\mathcal{F},Q;\mathbb{R}^d)$ be arbitrarily given. For $L\in\mathcal{L}^Q$, $(LQ)_\xi\in\mathcal{P}(\mathbb{R}^d)$ is defined by
$$ \int_{\mathbb{R}^d} \varphi d(LQ)_\xi:=E^Q[\varphi(\xi)L],\ \varphi\in b\mathcal{B}(\mathbb{R}^d), $$
where $b\mathcal{B}(\mathbb{R}^d):=\{\varphi:\mathbb{R}^d\rightarrow\mathbb{R}\,|\, \varphi \mbox{ bounded Borel function}\}$. Let us fix now an arbitrary function $f:\mathcal{P}(\mathbb{R}^d)\rightarrow\mathbb{R}$, and put
\begin{equation}\label{FQL}
F_Q(L):=f\big((LQ)_\xi\big),\ L\in\mathcal{L}^Q.
\end{equation}
\begin{remark}
If $\xi\in L^\infty(\Omega,\mathcal{F},Q;\mathbb{R}^d)$, the function $F_Q(L)=f\big((LQ)_\xi\big),\ L\in\mathcal{L}^Q,$ is well-defined for $f:\mathcal{P}_1(\mathbb{R}^d)\rightarrow\mathbb{R}$. Indeed, for $L\in\mathcal{L}^Q,$ $\displaystyle \int_{\mathbb{R}^d}|x|d(LQ)_\xi=E^Q[L\cdot|\xi|]\leq |\xi|_{\infty}<+\infty, $ where $|\xi|_{\infty}:=\esssup\limits_{\Omega}|\xi(\omega)|$.
\end{remark}

In the sense of the Fr\'{e}chet derivative over Banach spaces, but now for $\mathcal{L}^Q\subset L^1(\Omega,\mathcal{F},Q)$, we give the following definition.

\begin{definition}\label{Def3.1}
Given $L\in\mathcal{L}^Q$, we say that $F_Q:\mathcal{L}^Q\rightarrow\mathbb{R}$ defined in (\ref{FQL}) is differentiable at $L$, if there is some $(DF_Q)(L)\in L(L_0^1(\Omega,\mathcal{F},Q);\mathbb{R})$ such that
\begin{equation}\label{DFQL}
F_Q(L')-F_Q(L)=(DF_Q)(L)(L'-L)+o(|L'-L|_{L^1(Q)}),
\end{equation}
for all $L'\in\mathcal{L}^Q$ with $|L'-L|_{L^1(Q)}\rightarrow0$.
\end{definition}

We observe that this definition is well-stated. Indeed, we have the following statement:

\begin{lemma}\label{Lemma3.2}
For any given $L\in\mathcal{L}^Q$, we suppose that the function $F_Q:\mathcal{L}^Q\rightarrow\mathbb{R}$ defined in (\ref{FQL}) is differentiable at $L$ in the sense of the above Definition \ref{Def3.1}. Then the continuous linear functional $(DF_Q)(L)\in L(L_0^1(\Omega,\mathcal{F},Q);\mathbb{R})$ satisfying (\ref{DFQL}) is unique.
\end{lemma}
\begin{proof}
The proof of the lemma is strongly related with Lemma \ref{Lemma1}. For $i=1,2,$ let $\varphi_i\in L(L_0^1(\Omega,\mathcal{F},Q); \mathbb{R})$ be such that (\ref{DFQL}) is satisfied with  $(DF_Q)(L)$ replaced by $\varphi_i$. Due to Lemma \ref{Lemma1} there is a unique $\vartheta_\varphi\in L_0^\infty(\Omega,{\cal F},Q)$ such that, for $\varphi:=\varphi_1-\varphi_2\in L(L_0^1(\Omega,\mathcal{F},Q);\mathbb{R})$, it holds $E^Q[\vartheta_\varphi\eta]=\varphi(\eta)$, for all $\eta\in L^1_0(\Omega,\mathcal{F},Q)$. Then, from the difference of the both relations (\ref{DFQL}) satisfied by $\varphi_i,\, i=1,2$, we have
\begin{equation}\label{equ3.3}
E^Q[\vartheta_\varphi(L'-L)]=\varphi(L'-L)=o(|L'-L|_{L^1(Q)}),
\end{equation}
for all $L'\in\mathcal{L}^Q$ with $|L'-L|_{L^1(Q)}\rightarrow 0$.

Let now $\theta$ be an arbitrary element of $L^1_0(\Omega,\mathcal{F},Q)$, which we decompose in its positive and its negative parts $\theta=\theta^+-\theta^-$ (Observe that $E^Q[\theta^+]=E^Q[\theta^-]$), and letting $\alpha>0$, we define the following densities in $\mathcal{L}^Q$:

\centerline{$\displaystyle L'_1=\frac{L+\alpha\theta^+}{1+\alpha E^Q[\theta^+]},\quad L'_2=\frac{L+\alpha\theta^-}{1+\alpha E^Q[\theta^-]}.$}

\noindent Obviously,

\centerline{$\displaystyle L_1'-L=\alpha\frac{\theta^+-LE^Q[\theta^+]}{1+\alpha E^Q[\theta^+]},\, L_2'-L=\alpha\frac{\theta^--LE^Q[\theta^-]}{1+\alpha E^Q[\theta^-]},\, \mbox{ and } |L_i'-L|_{L^1(Q)}\le \frac{2\alpha E^Q[\theta^+]}{1+\alpha E^Q[\theta^+]},$}

\smallskip

\noindent for $i=1,2,$ and from (\ref{equ3.3})

\centerline{$\displaystyle\frac{\alpha}{1+\alpha E^Q[\theta^+]}E^Q[\vartheta_\varphi \theta]=E^Q[\vartheta_\varphi (L_1'-L_2')]=\displaystyle o\big(\frac{2\alpha E^Q[\theta^+]}{1+\alpha E^Q[\theta^+]}\big),\, \mbox{ as } \alpha\rightarrow 0.$}

\smallskip

\noindent But this implies that $E^Q[\vartheta_\varphi \theta]=0$, for all $\theta\in L^1_0(\Omega,\mathcal{F},Q)$. Consequently, choosing $\theta=\vartheta_\varphi$ we see that $\vartheta_\varphi=0$, $Q$-a.s. Therefore, $\varphi_1-\varphi_2=\varphi=0.$
\end{proof}

\medskip

After having seen that the definition of the derivative of the function $F_Q:\mathcal{L}^Q\rightarrow\mathbb{R}$ introduced in (\ref{FQL}) is well formulated, let us study now some properties of this derivative. For this we use the following notation: For $L\in\mathcal{L}^Q$ we put $Q_L:=LQ.$ Note that $Q_L$ is a probability on $(\Omega,\mathcal{F})$ and $\mathcal{L}^{Q_L}=\{L'\in L^1(\Omega,\mathcal{F},Q_L; \mathbb{R}_{+}):\ E^{Q_L}[L'](=E^Q[L'L])=1\}$.

\begin{lemma}\label{Lemma3.3}
Let $L\in\mathcal{L}^Q.$ Then the function $F_Q:\mathcal{L}^Q\rightarrow\mathbb{R}$ is differentiable at $L$ if and only if $F_{Q_L}:\mathcal{L}^{Q_L}\rightarrow\mathbb{R}$ is differentiable at $L_0=1$. Moreover, if $F_Q:\mathcal{L}^Q\rightarrow\mathbb{R}$ is differentiable at $L$ (and, thus, equivalently,  $F_{Q_L}:\mathcal{L}^{Q_L}\rightarrow\mathbb{R}$ is differentiable at $L_0=1$), then we have
\begin{equation}\label{equivDFQL}
\begin{split}
&DF_{Q_L}(1)=DF_Q(L)-E^{Q_L}[DF_Q(L)]\in L_0^\infty(\Omega,\mathcal{F},Q_L),\ Q_L\mbox{-a.s.}\, (\sim Q\mbox{-a.s.}),\\
&DF_Q(L)=DF_{Q_L}(1)-E^Q[DF_{Q_L}(1)]\in L_0^\infty(\Omega,\mathcal{F},Q),\ Q\mbox{-a.s.}
\end{split}
\end{equation}
\end{lemma}
\begin{proof}
Let $F_Q:\mathcal{L}^Q\rightarrow\mathbb{R}$ be differentiable at $L\in\mathcal{L}^Q$. Then, due to the preceding Lemma \ref{Lemma3.2}, abusing notation, there is a ($Q$-a.s.) unique $DF_Q(L)\in L_0^\infty(\Omega,\mathcal{F},Q)$ such that,
$$ (DF_Q)(L)(\eta)=E^Q[DF_Q(L)\eta],\ \eta\in L_0^1(\Omega,\mathcal{F},Q), $$
and, thus,
$$ F_Q(L')-F_Q(L)=E^Q[DF_Q(L)(L'-L)]+o(|L'-L|_{L^1(Q)}), $$
for $L'\in\mathcal{L}^Q$, $|L'-L|_{L^1(Q)}\rightarrow0$. On the other hand, consider $L''\in \mathcal{L}^{Q_L}$ with $|L''-1|_{L^1(Q_L)}\rightarrow0$. Then $L''\cdot L\in\mathcal{L}^Q$, $|L''L-L|_{L^1(Q)}=|L''-1|_{L^1(Q_L)}$, and
\begin{equation*}
\begin{split}
&F_{Q_L}(L'')-F_{Q_L}(1)=f\big((L'' Q_L)_\xi\big)-f\big((Q_L)_\xi\big)=f\big(((L''L) Q)_\xi\big)-f\big((LQ)_\xi\big)\\
=&F_Q(L''L)-F_Q(L)=E^Q[DF_Q(L)(L''L-L)]+o(|L''L-L|_{L^1(Q)})\\
=& E^{Q_L}[DF_Q(L)(L''-1)]+o(|L''-1|_{L^1(Q_L)}).
\end{split}
\end{equation*}
As $E^{Q_L}[L''-1]=0$, this yields
$$ F_{Q_L}(L'')-F_{Q_L}(1)=E^{Q_L}\big[\big(DF_Q(L)-E^{Q_L}[DF_Q(L)]\big)(L''-1)\big]+o(|L''-1|_{L^1(Q_L)}). $$
Note that $DF_Q(L)-E^{Q_L}[DF_Q(L)]\in L^\infty_0(\Omega,\mathcal{F},Q_L)$.
Consequently, $F_{Q_L}:\mathcal{L}^{Q_L}\rightarrow\mathbb{R}$ is differentiable at $L_0=1$, and
\begin{equation}\label{*1}
DF_{Q_L}(1)=DF_Q(L)-E^{Q_L}[DF_Q(L)],\ Q_L\mbox{-a.s.}
\end{equation}
Conversely, given any $L\in\mathcal{L}^Q$ we suppose now that $F_{Q_L}:\mathcal{L}^{Q_L}\rightarrow\mathbb{R}$ is differentiable at $L_0=1$. Then, for $L'\in\mathcal{L}^Q$ with $|L'-L|_{L^1(Q)}\rightarrow0$, noting that $L'/L\in\mathcal{L}^{Q_L}$ (Remark that $Q_L\{L=0\}=E^Q[L\mathbf{1}_{\{L=0\}}]=0$), we have
\begin{equation*}
\begin{split}
&F_{Q}(L')-F_{Q}(L)=f\big((L' Q)_\xi\big)-f\big((L Q)_\xi\big)=f\big((\frac{L'}{L} Q_L)_\xi\big)-f\big((Q_L)_\xi\big)\\
=&F_{Q_L}(\frac{L'}{L})-F_{Q_L}(1)=E^{Q_L}[DF_{Q_L}(1)(\frac{L'}{L}-1)]+o(|\frac{L'}{L}-1|_{L^1(Q_L)})\\
=& E^Q[DF_{Q_L}(1)(L'-L)]+o(|L'-L|_{L^1(Q)})\\
=& E^Q\big[\big(DF_{Q_L}(1)-E^Q[DF_{Q_L}(1)]\big)(L'-L)\big]+o(|L'-L|_{L^1(Q)}).
\end{split}
\end{equation*}
As $DF_{Q_L}(1)\in L^\infty_0(\Omega,\mathcal{F},Q_L)$, it holds that
$DF_{Q_L}(1)-E^Q[DF_{Q_L}(1)]\in L^\infty_0(\Omega,\mathcal{F},Q)$, and, consequently, $F_Q:\mathcal{L}^Q\rightarrow\mathbb{R}$ is differentiable at $L$ and $DF_Q(L)=DF_{Q_L}(1)-E^Q[DF_{Q_L}(1)],\ Q$-a.s.
\end{proof}

\medskip

The next statement investigates the form of the derivative with respect to the density. The result is comparable with the characterisation of the derivative of a function $f:\mathcal{P}_2(\mathbb{R}^d)\rightarrow \mathbb{R}$ at a probability law $P_{\vartheta_0}$ by a Borel function $h_0:\mathbb{R}^d\rightarrow\mathbb{R}$, $D\widetilde{f}(\vartheta_0)=h_0(\vartheta_0),$  where $h_0$ depends on the random variable $\vartheta_0$ only through its law $P_{\vartheta_0}$, see (\ref{equ2.4}).

\smallskip

\begin{theorem}\label{P3}
Let $F_{Q_L}:\mathcal{L}^{Q_L}\rightarrow\mathbb{R}$ be differentiable at $L_0=1$. Then there exists a bounded Borel function $g:\mathbb{R}^d\rightarrow\mathbb{R}$ such that $DF_{Q_L}(1)=g(\xi),\ Q$-a.s. Moreover, $g$ depends on $(Q,L,\xi)$ only through the law $(Q_L)_\xi$.
\end{theorem}
\begin{proof}
Let $L''\in\mathcal{L}^{Q_L}$. As, for all $\varphi\in b\mathcal{B}(\mathbb{R}^d)$,
$$ \int_{\mathbb{R}^d}\varphi d(L''Q_L)_\xi=E^{Q_L}[L''\varphi(\xi)]=E^{Q_L}\big[E^{Q_L}[L''\,|\,\xi]\varphi(\xi)\big], $$
and as the conditional expectation $E^{Q_L}[L''\,|\,\xi]$ of $L''$ under $Q_L$, knowing $\xi$,  belongs to $\mathcal{L}^{Q_L}$, it follows that $(L''Q_L)_\xi=\big(E^{Q_L}[L''\,|\,\xi]Q_L\big)_\xi$. Therefore, $F_{Q_L}(L'')=F_{Q_L}\big(E^{Q_L}[L''\,|\,\xi]\big)$.
Consequently,
\begin{equation}\label{81}
\begin{split}
&E^{Q_L}[DF_{Q_L}(1)(L''-1)]=F_{Q_L}(L'')-F_{Q_L}(1)+o(|L''-1|_{L^1(Q_L)})\\
=&F_{Q_L}\big(E^{Q_L}[L''\,|\,\xi]\big)-F_{Q_L}(1)+o(|L''-1|_{L^1(Q_L)})\\
=&E^{Q_L}\big[(DF_{Q_L})(1)\big(E^{Q_L}[L''\,|\,\xi]-1\big)\big]+o(|L''-1|_{L^1(Q_L)}),
\end{split}
\end{equation}
for $L''\in\mathcal{L}^{Q_L}$ with $|L''-1|_{L^1(Q_L)}\rightarrow0$. Here we have used for the last equality that
$$\big|E^{Q_L}[L''\,|\,\xi]-1\big|_{L^1(Q_L)}\le |L''-1|_{L^1(Q_L)}, \mbox{ and so } o(|E^{Q_L}[L''\,|\,\xi]-1|_{L^1(Q_L)})=o(|L''-1|_{L^1(Q_L)}).$$

Let now $\eta\in L^1_0(\Omega,\mathcal{F},Q_L)$.  Arguing similarly as in the proof of Lemma \ref{Lemma3.2}, we put  $\eta_1:=\eta^+$, $\eta_2:=\eta^-$, and for $i=1,2$ and $\alpha>0$ we define
$$ L_i:=\frac{1+\alpha\eta_i}{1+\alpha E^{Q_L}[\eta_i]}\in\mathcal{L}^{Q_L},\, i=1,2.$$
Note that $\displaystyle L_i-1=\alpha\frac{\eta_i-E^{Q_L}[\eta_i]}{1+\alpha E^{Q_L}[\eta_i]}$ and, thus, $|L_i-1|_{L^1(Q_L)}\leq 2\alpha E^{Q_L}[\eta_i]\rightarrow0,\ \alpha\rightarrow0$ and
$$ \frac{1}{\alpha}(L_i-1)=\frac{\eta_i-E^{Q_L}[\eta_i]}{1+\alpha E^{Q_L}[\eta_i]}\longrightarrow\eta_i-E^{Q_L}[\eta_i],\, \mbox{ in } L^1(Q_L),\, \mbox{ as }\alpha\rightarrow0. $$
Hence, substituting $L''=L_i$ in \eqref{81}, dividing by $\alpha>0$ and taking the limit as $\alpha\searrow 0$ yields, for $i=1,2,$
$$ E^{Q_L}\big[DF_{Q_L}(1)\big(\eta_i-E^{Q_L}[\eta_i]\big)\big]=E^{Q_L}\big[DF_{Q_L}(1)\big(E^{Q_L}[\eta_i\,|\,\xi]-E^{Q_L}[\eta_i]\big)\big]. $$
But $\eta=\eta_1-\eta_2$ and $E^{Q_L}[\eta_1]=E^{Q_L}[\eta_2]$. Hence, for all $\eta\in L^1_0(\Omega,\mathcal{F},Q_L)$, we have
$$ E^{Q_L}[DF_{Q_L}(1)\eta]=E^{Q_L}\big[DF_{Q_L}(1)E^{Q_L}[\eta\, |\, \xi]\big],$$
and, thus,
\begin{equation}\label{*3}
E^{Q_L}\big[\big(DF_{Q_L}(1)-E^{Q_L}[DF_{Q_L}(1)\, |\, \xi]\big)\eta\big]=0,\ \mbox{ for all } \eta\in L^1_0(\Omega,\mathcal{F},Q_L).
\end{equation}
Choosing $\eta:=DF_{Q_L}(1)-E^{Q_L}[DF_{Q_L}(1)\,|\,\xi]$ in \eqref{*3} shows that
$$ DF_{Q_L}(1)=E^{Q_L}[DF_{Q_L}(1)\, |\, \xi],\  Q_L\mbox{-a.s.}(\sim Q\mbox{-a.s.}).$$
Consequently, there exists a Borel function $g:\mathbb{R}^d\rightarrow\mathbb{R}$ such that
\begin{equation}\label{new3A}
DF_{Q_L}(1)=g(\xi),\  Q\mbox{-a.s.},
\end{equation}
and from \eqref{equivDFQL}, $DF_Q(L)=DF_{Q_L}(1)-E^Q[DF_{Q_L}(1)]=\bar{g}(\xi),\ Q$-a.s., for $\bar{g}(y):=g(y)-E^Q[DF_{Q_L}(1)]$, $y\in\mathbb{R}^d$.

It remains to show that $g$ depends on $(Q,L,\xi)$ only through the law $(Q_L)_\xi$. Let $Q'$ be a probability measure on $(\Omega,\mathcal{F})$, $L'\in\mathcal{L}^{Q'}$ and $\xi'\in L^0(\Omega,\mathcal{F},Q';\mathbb{R}^d)$ be such that $(L'Q')_{\xi'}=(LQ)_\xi$, or, equivalently, $(Q'_{L'})_{\xi'}=(Q_L)_\xi$.

We still use the notation $F_{Q'_{L'}}(\widehat{L}'):=f\big((\widehat{L}'Q'_{L'})_{\xi'}\big),\ \widehat{L}'\in\mathcal{L}^{Q'_{L'}}$, and we are going to show that the differentiability of $F_{Q_L}:\mathcal{L}^{Q_L}\rightarrow \mathbb{R}$ at $L_0=1$ implies that of $F_{Q'_{L'}}:\mathcal{L}^{Q'_{L'}} \rightarrow\mathbb{R}$ at $L'_0=1$. Let $\widehat{L}'\in\mathcal{L}^{Q'_{L'}}$ be such that $|\widehat{L}'-1|_{L^1(Q'_{L'})}$ tends to $0$, and let $h:\mathbb{R}^d\rightarrow\mathbb{R}$ be a Borel function such that $E^{Q'_{L'}}[\widehat{L}'\, |\, \xi']=h(\xi'),\ Q'_{L'}$-a.s. Notice that, for all $\varphi\in b\mathcal{B}(\mathbb{R}^d)$,
\begin{equation}\label{new3}
\begin{split}
&\int_{\mathbb{R}^d}\varphi d(\widehat{L}'Q'_{L'})_{\xi'}=E^{Q'_{L'}}[\widehat{L}'\varphi(\xi')]=E^{Q'_{L'}}\big[E^{Q'_{L'}}[\widehat{L}'\,|\, \xi']\varphi(\xi')\big]\\
&=E^{Q'_{L'}}[h(\xi')\varphi(\xi')]=\int_{\mathbb{R}^d}h\cdot\varphi d(Q'_{L'})_{\xi'}=\int_{\mathbb{R}^d}h\cdot\varphi d(Q_{L})_{\xi}=E^{Q_L}[h(\xi)\varphi(\xi)].
\end{split}
\end{equation}
It is easy to check that $h(\xi)\in\mathcal{L}^{Q_L}$. Then it follows from \eqref{new3} that
$$ \int_{\mathbb{R}^d}\varphi d(\widehat{L}'Q'_{L'})_{\xi'}=E^{Q_L}[h(\xi)\varphi(\xi)]=\int_{\mathbb{R}^d}\varphi d\big(h(\xi)Q_{L}\big)_{\xi},\ \mbox{ for all } \varphi\in b\mathcal{B}(\mathbb{R}^d), $$
which implies that $(\widehat{L}'Q'_{L'})_{\xi'}=\big(h(\xi)Q_{L}\big)_{\xi}=\big(h(\xi')Q'_{L'}\big)_{\xi'}$. Then, recalling the definition of $F_{Q'_{L'}}$, we obtain
$$ F_{Q'_{L'}}(\widehat{L}')=f\big((\widehat{L}'Q'_{L'})_{\xi'}\big)=f\big((h(\xi)Q_{L})_\xi\big)=F_{Q_L}\big(h(\xi)\big). $$
Consequently, using the fact that $DF_{Q_L}(1)$ exists and coincides $Q$-a.s. (and thus also $Q_L$-a.s.) with $g(\xi)$ (see (\ref{new3A})), we have
\begin{equation*}
\begin{split}
&F_{Q'_{L'}}(\widehat{L}')-F_{Q'_{L'}}(1)=F_{Q_{L}}\big(h(\xi)\big)-F_{Q_{L}}(1)=E^{Q_L}\big[g(\xi)\big(h(\xi)-1\big)\big]+o(|h(\xi)-1|_{L^1(Q_L)})\\
=&E^{Q'_{L'}}\big[g(\xi')\big(h(\xi')-1\big)\big]+o(|h(\xi')-1|_{L^1(Q'_{L'})})=E^{Q'_{L'}}\big[g(\xi')\big(\widehat{L}'-1\big)\big]+o(|\widehat{L}'-1|_{L^1(Q'_{L'})}),
\end{split}
\end{equation*}
i.e., also $F_{Q'_{L'}}:\mathcal{L}^{Q'_{L'}}\rightarrow\mathbb{R}$ is differentiable at $L_0=1$, and
$$ DF_{Q'_{L'}}(1)=g(\xi'),\ Q'_{L'}\mbox{-a.s.}$$
This shows that the function $g$ introduced in (\ref{new3A}) depends on $(Q,L,\xi)$ only through the law $(Q_L)_\xi.$
Now the proof is complete.
\end{proof}

The above theorem justifies the following notation. Using the function $g$ introduced in (\ref{new3A}), we define:
$$\partial_1 F\big((Q_L)_\xi,x\big):=g(x),\ x\in\mathbb{R}^d.$$
We observe that the function is $(Q_L)_{\xi}(dx)$-a.s. well defined, and
\begin{equation}\label{***}
\partial_1 F\big((Q_L)_\xi,\xi\big)=g(\xi)=DF_{Q_L}(1),\ Q_L\mbox{-a.s.}
\end{equation}
\begin{remark}\label{Re3.2}
We have seen in Lemma \ref{Lemma3.3} that the differentiability of $F_Q:\mathcal{L}^Q\rightarrow\mathbb{R}$ at $L\in\mathcal{L}^Q$ is equivalent with that of $F_{Q_L}:\mathcal{L}^{Q_L}\rightarrow\mathbb{R}$ at $L_0=1$.  Moreover, $Q$-a.s.,
\begin{equation*}
\begin{split}
&DF_{Q_L}(1)=DF_Q(L)-E^{Q_L}[DF_Q(L)],\ Q_L\mbox{-a.s.}\, (\sim Q\mbox{-a.s.}),\\
&DF_Q(L)=DF_{Q_L}(1)-E^Q[DF_{Q_L}(1)],\ Q\mbox{-a.s.}
\end{split}
\end{equation*}

\noindent While $DF_{Q_L}(1)=g(\xi)$, $Q$-a.s., where $g(x)=\partial_1 F\big((Q_L)_\xi,x\big),\, x\in\mathbb{R}^d,$ depends on $(Q, L,\xi)$ only through $(Q_L)_\xi$, we see that
\begin{equation*}
\begin{split}
DF_Q(L)&=DF_{Q_L}(1)-E^Q[DF_{Q_L}(1)]\\
&=\partial_1 F\big((Q_L)_\xi,\xi\big)-E^Q\big[\partial_1 F\big((Q_L)_\xi,\xi\big)\big]=\bar{g}(\xi),\, Q\mbox{-a.s.},
\end{split}
\end{equation*}
\noindent where
$$\bar{g}(x):=\partial_L F\big(Q_\xi,(LQ)_\xi,x\big):=\partial_1 F\big((Q_L)_\xi,x\big)-\int_{\mathbb{R}^d}\partial_1 F\big((Q_L)_\xi,y\big)Q_\xi(dy), \, x\in\mathbb{R}^d,$$
depends not only on $(Q_L)_\xi$ but also on $Q_\xi$.
\end{remark}

Let us consider the following example.
\begin{example}\label{Example3.1}
Let $h\in C^1(\mathbb{R})$, $\varphi\in b\mathcal{B}(\mathbb{R})$, and $f(\mu):=h(\int_{\mathbb{R}}\varphi d\mu),\ \mu\in\mathcal{P}(\mathbb{R})$. For $\xi\in L^0(\Omega,\mathcal{F},Q)$, let $F_Q(L):=f\big((LQ)_\xi\big)=h\big(E^Q[L\varphi(\xi)]\big),\ L\in\mathcal{L}^Q$. It can be easily checked that the Fr\'echet derivative $DF_Q(L)\in L(L^1_0(\Omega,\mathcal{F},Q);\mathbb{R})$ has the form
$$ DF_Q(L)(\eta)=h'\big(E^Q[L\varphi(\xi)]\big)E^Q[\varphi(\xi)\eta],\ \eta\in L_0^1(\Omega,\mathcal{F},Q), $$
while its identification with an element of $L^\infty_0(\Omega,\mathcal{F},Q)$ is just
$$ DF_Q(L)=h'\big(E^Q[L\varphi(\xi)]\big)\big(\varphi(\xi)-E^Q[\varphi(\xi)]\big),\ Q\mbox{-a.s.},\, L\in\mathcal{L}^Q.$$
Consequently,
$$ \partial_L F\big(Q_\xi,(LQ)_\xi,x\big)=h'\Big(\int_{\mathbb{R}}\varphi d(Q_L)_\xi\Big)\Big(\varphi(x)-\int_{\mathbb{R}}\varphi dQ_\xi\Big),\ Q_\xi(dx)\mbox{-a.s.}, $$
and choosing instead of $(Q,L)$ the couple $(Q_L,1)$, we obtain
$$ \partial_1 F\big((Q_L)_\xi,x\big)=h'\big(E^{Q_L}[\varphi(\xi)]\big)\big(\varphi(x)-E^{Q_L}[\varphi(\xi)]\big),\ (Q_L)_\xi(dx)\mbox{-a.s.} $$

\end{example}

\subsection{Relation between the derivative w.r.t. the density and partial derivatives}

The objective of this subsection is to investigate the relation between the derivative with respect to the density introduced in the preceding subsection and partial derivatives of functions over a suitable space of probability measures studied by \cite{BCL}. For this we keep the notations introduced in the preceding subsection.

Let $Q$ be a probability measure over the Radon space $(\Omega,\mathcal{F}=\mathcal{B}(\Omega)\vee\mathcal{N}_Q)$. We note that, for $L\in\mathcal{L}^Q$ and $\xi\in L^0(\Omega,\mathcal{F},Q;\mathbb{R}^d)$,
$$ (LQ)_\xi(dx)=\int_{\mathbb{R}}z Q_{(L,\xi)}(dzdx). $$
Indeed, due to the transfer theorem we have, for all $\varphi\in b\mathcal{B}(\mathbb{R}^d)$,
\begin{equation}\label{821}
\int_{\mathbb{R}^d}\varphi d(LQ)_\xi=E^Q[L\varphi(\xi)]=\int_{\mathbb{R}^d}\varphi(x)\Big(\int_{\mathbb{R}}z Q_{(L,\xi)}(dzdx)\Big).
\end{equation}
In order to be able to work in the framework chosen by \cite{BCL}, we extend  \eqref{821} to signed measures and restrict to $(L,\xi)\in L^2(\Omega,\mathcal{F},Q)\times L^0(\Omega,\mathcal{F},Q;\mathbb{R}^d)$. Thus, given
$(L,\xi)\in L^2(\Omega,\mathcal{F},Q)\times L^0(\Omega,\mathcal{F},Q;\mathbb{R}^d)$, we put
\begin{equation}\label{822}
(LQ)_\xi(dx):=\int_{\mathbb{R}}z Q_{(L,\xi)}(dzdx)\, \big(=(L^+Q)_\xi(dx)-(L^-Q)_\xi(dx)\big).
\end{equation}
Then,
$$ Q_{(L,\xi)}\in\mathcal{P}_{2,0}(\mathbb{R}\times\mathbb{R}^d):=\big\{\mu\in\mathcal{P}(\mathbb{R}\times\mathbb{R}^d):\ \mu(\cdot\times\mathbb{R}^d)\in\mathcal{P}_2(\mathbb{R})\big\}, $$
and
$$ (LQ)_\xi(dx)=\int_{\mathbb{R}}z Q_{(L,\xi)}(dzdx)\in\mathcal{M}(\mathbb{R}^d):=\big\{\alpha\mu_1-\beta\mu_2\,\big|\,\mu_1,\mu_2\in\mathcal{P}(\mathbb{R}^d),\ \alpha,\ \beta\in \mathbb{R}_+\big\}. $$
We observe that $Q_{(L,\xi)}$ has the disintegration
$$ Q_{(L,\xi)}(dzdx)=Q\{L\in dz\,|\, \xi=x\}Q_\xi(dx)=Q_{L|\xi}(dz,x)Q_\xi(dx),$$
where $Q_{L|\xi}(dz,x)$ denotes the regular conditional probability associated with $Q\{L\in dz\,|\, \xi=x\}$; its existence is guaranteed by the fact that $(\Omega,\mathcal{B}(\Omega))$ is a Radon space.

Let $f:\mathcal{M}(\mathbb{R}^d)\rightarrow\mathbb{R}$, and define $G:\mathcal{P}_{2,0}(\mathbb{R}\times\mathbb{R}^d)\rightarrow\mathbb{R}$ through the relation
$$ G(Q_{(L,\xi)}):=f\big((LQ)_\xi\big)=f\Big(\int_{\mathbb{R}}z Q_{(L,\xi)}(dzdx)\Big),\ (L,\xi)\in L^2(\Omega,\mathcal{F},Q)\times L^0(\Omega,\mathcal{F},Q;\mathbb{R}^d). $$
As we suppose that $\mathcal{F}$ is ``rich enough'', we have  $\mathcal{P}(\mathbb{R}\times\mathbb{R}^d)=\big\{Q_{(L,\xi)},\, (L,\xi)\in L^0(\Omega,\mathcal{F},Q;\mathbb{R}\times\mathbb{R}^d)\big\}$ and $\mathcal{P}_{2,0}(\mathbb{R}\times\mathbb{R}^d)=\big\{Q_{(L,\xi)},\, (L,\xi)\in L^2(\Omega,\mathcal{F},Q;\mathbb{R})\times L^0(\Omega,\mathcal{F},Q;\mathbb{R}^d)\big\}$, and the mapping
$G$ is well defined on $\mathcal{P}_{2,0}(\mathbb{R}\times\mathbb{R}^d)$.

Inspired by the approach by P.L. Lions (see \cite{C12}) for the study of the differentiability of a function over $\mathcal{P}_2(\mathbb{R})$, \cite{BCL} introduced the partial derivative of $G$ through the Fr\'echet derivative of its lifted function  $\widetilde{G}_{Q,\xi}:L^2(\Omega,\mathcal{F},Q)\rightarrow\mathbb{R}$ defined by $\widetilde{G}_{Q,\xi}(L):=G(Q_{(L,\xi)}),\ L\in L^2(\Omega,\mathcal{F},Q)$. Observe that, in particular, for $L\in\mathcal{L}^Q\cap L^2(\Omega,\mathcal{F},Q)$,
\begin{equation}\label{823}
\widetilde{G}_{Q,\xi}(L)=G(Q_{(L,\xi)})=f\big((LQ)_\xi\big)=F_Q(L),
\end{equation}
where $F_Q:\mathcal{L}^Q\rightarrow\mathbb{R}$ has been introduced in (\ref{FQL}). So \eqref{823} relates derivatives w.r.t. the density with partial derivatives. Adapting Definition 3.1 in \cite{BCL} to our more special setting here, we make the following definition:
\begin{definition}
The mapping $G:\mathcal{P}_{2,0}(\mathbb{R}\times\mathbb{R}^d) \rightarrow\mathbb{R}$ is said to be (partially) differentiable with respect to $Q_{L|\xi}$ at $Q_{(L,\xi)}$, if $\widetilde{G}_{Q,\xi}:L^2(\Omega,\mathcal{F},Q)\rightarrow\mathbb{R}$ is Fr\'{e}chet differentiable at $L$.
\end{definition}
\begin{remark}\label{Remark3.3}
In fact, a direct translation of Definition 3.1 in \cite{BCL} would lead to the formulation that $G:\mathcal{P}_{2,0}(\mathbb{R}\times\mathbb{R}^d) \rightarrow\mathbb{R}$ is said to be partially differentiable w.r.t. $Q_{L|\xi}$ at $Q_{(L,\xi)}$, if there is some random variable $(L',\xi')\in  L^2(\Omega,\mathcal{F},Q;\mathbb{R})\times L^0(\Omega,\mathcal{F},Q;\mathbb{R}^d)$ with $Q_{(L',\xi')}=Q_{(L,\xi)}$ such that $\widetilde{G}_{Q,\xi'}:L^2(\Omega,\mathcal{F},Q)\rightarrow\mathbb{R}$ is Fr\'{e}chet differentiable at $L'$. However, as a consequence of Theorem 3.1 and Proposition 3.1 in \cite{BCL}, Corollary 3.1 (also in \cite{BCL}) states that the both definitions coincide; for more details we refer the interested readers to \cite{BCL}.
\end{remark}
Let us suppose that $G:\mathcal{P}_{2,0}(\mathbb{R}\times\mathbb{R}^d)\rightarrow\mathbb{R}$ is differentiable w.r.t. $Q_{L|\xi}$ at $Q_{(L,\xi)}$. Then, as $\widetilde{G}_{Q,\xi}:L^2(\Omega,\mathcal{F},Q)\rightarrow\mathbb{R}$ is Fr\'{e}chet differentiable at $L$, due to the Riesz' Representation Theorem
its derivative $D\widetilde{G}_{Q,\xi}(L)\in L(L^2(\Omega,\mathcal{F},Q);\mathbb{R})$ can be identified with an element of $L^2(\Omega,\mathcal{F},Q)$, also denoted by $D\widetilde{G}_{Q,\xi}(L)$, and Theorem 3.1 in \cite{BCL} states that there exists a $Q_{(L,\xi)}$-a.s. unique Borel function $g:\mathbb{R}\times\mathbb{R}^d\rightarrow\mathbb{R}$ such that
$$ D\widetilde{G}_{Q,\xi}(L)=g(L,\xi),\ Q\mbox{-a.s.} $$
Moreover, again due to Theorem 3.1 in \cite{BCL}, $g$ depends on $(Q, L,\xi)$ only through the law $Q_{(L,\xi)}$. This motivates the notation
$$ (\partial_\mu G)_1\big(Q_{(L,\xi)},z,x\big):=g(z,x),\ (z,x)\in\mathbb{R}\times\mathbb{R}^d. $$
Note that, for $L\in L^2(\Omega,\mathcal{F},Q)$, and $L'\in L^2(\Omega,\mathcal{F},Q_L(=LQ))$ such that $L'L\in L^2(\Omega,\mathcal{F},Q)$, we have
\begin{equation}\label{824}
\begin{split}
&\widetilde{G}_{Q,\xi}(L'L)=G\big(Q_{(L'L,\xi)}\big)=f\big((L'LQ)_\xi\big)=f\big((L'Q_L)_\xi\big)\\
&=G\big((Q_L)_{(L',\xi)}\big)=\widetilde{G}_{Q_L,\xi}(L').
\end{split}
\end{equation}
To simplify our argument, let us consider $L\in L^\infty(\Omega,\mathcal{F},Q)$. Then \eqref{824} holds for all $L'\in L^2(\Omega,\mathcal{F},Q_L)$. Recall that we have supposed that $G:\mathcal{P}_{2,0}(\mathbb{R}\times\mathbb{R}^d)\rightarrow\mathbb{R}$ is differentiable w.r.t. $Q_{L|\xi}$ at $Q_{(L,\xi)}$. Then, taking into account \eqref{824}, we have for $L'\in L^2(\Omega,\mathcal{F},Q_L)$ with $|L'-1|_{L^2(Q_L)}\rightarrow0$,
\begin{equation}\label{*4}
\begin{split}
&\widetilde{G}_{Q_L,\xi}(L')-\widetilde{G}_{Q_L,\xi}(1)=\widetilde{G}_{Q,\xi}(L'L)-\widetilde{G}_{Q,\xi}(L)\\
=&(D\widetilde{G}_{Q,\xi})(L)(L'L-L)+o(|L'L-L|_{L^2(Q)})\\
=&E^Q\big[(\partial_\mu G)_1\big(Q_{(L,\xi)},(L,\xi)\big)(L'L-L)\big]+o(|L'L-L|_{L^2(Q)})\\
=&E^{Q_L}\big[(\partial_\mu G)_1\big(Q_{(L,\xi)},(L,\xi)\big)(L'-1)\big]+o(|L'L-L|_{L^2(Q)}).
\end{split}
\end{equation}
Here $(\partial_\mu G)_1\big(Q_{(L,\xi)},(L,\xi)\big)=(D\widetilde{G}_{Q,\xi})(L)\in L^2(\Omega,\mathcal{F},Q)$.
Thus, observing that, for $L$ bounded by $C_L\in \mathbb{R}_+$,
$$ |L'L-L|^2_{L^2(Q)}=E^Q[L^2|L'-1|^2]\leq C_LE^{Q_L}[|L'-1|^2]=C_L|L'-1|^2_{L^2(Q_L)}, $$
it follows from \eqref{*4} that
$$ \widetilde{G}_{Q_L,\xi}(L')-\widetilde{G}_{Q_L,\xi}(1)=E^{Q_L}\big[(\partial_\mu G)_1\big(Q_{(L,\xi)},(L,\xi)\big)(L'-1)\big]+o(|L'-1|_{L^2(Q_L)}), $$
as $|L'-1|_{L^2(Q_L)}\rightarrow0$. This proves that $G:\mathcal{P}_{2,0}(\mathbb{R}\times\mathbb{R}^d)\rightarrow\mathbb{R}$ is partially differentiable w.r.t. $(Q_L)_{L'|\xi}$ at $(Q_L)_{(L',\xi)}=(Q_L)_{(1,\xi)}$, and
\begin{equation}\label{825}
D\widetilde{G}_{Q_L,\xi}(1)=(\partial_\mu G)_1\big(Q_{(L,\xi)},(L,\xi)\big)\in L^2(\Omega,\mathcal{F},Q_L).
\end{equation}
Hence, $Q_L$-a.s.,
\begin{equation}\label{826}
(\partial_\mu G)_1\big((Q_L)_{(1,\xi)},(1,\xi)\big)=D\widetilde{G}_{Q_L,\xi}(1)
=(\partial_\mu G)_1\big(Q_{(L,\xi)},(L,\xi)\big).
\end{equation}
Note that $(Q_L)_{(1,\xi)}=\delta_1\otimes(Q_L)_\xi$, where $\delta_1\in\mathcal{P}_2(\mathbb{R})$ is the Dirac measure with mass at $1$. Indeed, for all $\varphi\in b\mathcal{B}(\mathbb{R}\times\mathbb{R}^d)$,
$$ \int_{\mathbb{R}\times\mathbb{R}^d}\varphi d(Q_L)_{(1,\xi)}=E^{Q_L}[\varphi(1,\xi)]=\int_{\mathbb{R}^d}\varphi(1,x)(Q_L)_\xi(dx)=\int_{\mathbb{R}\times\mathbb{R}^d}\varphi d\big(\delta_1\otimes(Q_L)_\xi\big). $$
Consequently, the partial differentiability of $G:\mathcal{P}_{2,0}(\mathbb{R}\times\mathbb{R}^d)\rightarrow\mathbb{R}$ at $Q_{(L,\xi)}$ w.r.t. $Q_{L|\xi}$ implies that of $G$ at $(Q_L)_{(1,\xi)}$ w.r.t. $(Q_L)_{L'|\xi}$, and from \eqref{826},
$$ (\partial_\mu G)_1\big(Q_{(L,\xi)},(L,\xi)\big)=(\partial_\mu G)_1\big(\delta_1\otimes(Q_L)_\xi,(1,\xi)\big),\ Q_L\mbox{-a.s.} $$
This relation motivates us to introduce (with abusing notation) the function
\begin{equation}\label{827}
(\partial_\mu G)_1(\gamma,x):=(\partial_\mu G)_1\big(\delta_1\otimes\gamma,(1,x)\big),\ x\in\mathbb{R}^d,\ \gamma(dx)\mbox{-a.e.},\ \gamma\in\mathcal{P}(\mathbb{R}^d).
\end{equation}
Obviously, $(\partial_\mu G)_1(\gamma,\cdot):\mathbb{R}^d\rightarrow\mathbb{R}$ is Borel measurable, and
\begin{equation}\label{828}
(\partial_\mu G)_1\big(Q_{(L,\xi)},(L,\xi)\big)=(\partial_\mu G)_1\big((Q_L)_\xi,\xi\big),\ Q_L\mbox{-a.s.}
\end{equation}
We come now back to \eqref{823}. Let now $L\in\mathcal{L}^Q\cap L^\infty(\Omega,\mathcal{F},Q)$ and $L'\in\mathcal{L}^{Q_L}\cap L^2(\Omega,\mathcal{F},Q_L)$. Then
$$ F_{Q_L}(L')=f\big((L'Q_L)_\xi\big)=G\big((Q_L)_{(L',\xi)}\big). $$
We assume that $G:\mathcal{P}_{2,0}(\mathbb{R}\times\mathbb{R}^d)\rightarrow\mathbb{R}$ is partial differentiable at $(Q_L)_{(1,\xi)}$ w.r.t. $(Q_L)_{L'|\xi}$. Then, for $L'\in\mathcal{L}^{Q_L}\cap L^2(\Omega,\mathcal{F},Q_L)$ with $|L'-1|_{L^2(Q_L)}\rightarrow0$, using \eqref{827}, we see that
\begin{equation*}
\begin{split}
&F_{Q_L}(L')-F_{Q_L}(1)=G\big((Q_L)_{(L',\xi)}\big)-G\big((Q_L)_{(1,\xi)}\big)\\
=&E^{Q_L}\big[(\partial_\mu G)_1\big((Q_L)_{(1,\xi)},(1,\xi)\big)(L'-1)\big]+o(|L'-1|_{L^2(Q_L)})\\
=&E^{Q_L}\big[\big\{(\partial_\mu G)_1\big((Q_L)_{\xi},\xi\big)-E^{Q_L}\big[(\partial_\mu G)_1\big((Q_L)_{\xi},\xi\big)\big]\big\}(L'-1)\big]+o(|L'-1|_{L^2(Q_L)}).
\end{split}
\end{equation*}
Hence, the restriction of $F_{Q_L}$ from $\mathcal{L}^{Q_L}$ to $\mathcal{L}^{Q_L}\cap L^2(\Omega,\mathcal{F},Q_L)(\subset L^2(\Omega,\mathcal{F},Q_L))$ is Fr\'{e}chet differentiable (in $L^2(Q_L)$-sense) at $L_0=1$. Let us define
$$ \Phi_{Q_L}(1):=(\partial_\mu G)_1\big((Q_L)_\xi,\xi\big)-E^{Q_L}\big[(\partial_\mu G)_1\big((Q_L)_\xi,\xi\big)\big],\ Q_L\mbox{-a.s.}$$
As $(\partial_\mu G)_1\big((Q_L)_\xi,\xi\big)\in L^2(\Omega,\mathcal{F},Q_L)$, the random variable $\Phi_{Q_L}(1)$ belongs to $L_0^2(\Omega,\mathcal{F},Q_L)$. Similar to the proof of Lemma \ref{Lemma3.2} we obtain that $\Phi_{Q_L}(1)$ is the unique element in $L_0^2(\Omega,\mathcal{F},Q_L)$ such that
\begin{equation}\label{PhiQL}
\begin{split}
&F_{Q_L}(L')-F_{Q_L}(1)=E^{Q_L}\big[\Phi_{Q_L}(1)(L'-1)\big]+o(|L'-1|_{L^2(Q_L)}),
\end{split}
\end{equation}
for $L'\in\mathcal{L}^{Q_L}\cap L^2(\Omega,\mathcal{F},Q_L)$ with $|L'-1|_{L^2(Q_L)}\rightarrow0$.

Let us suppose now that $F_{Q_L}:\mathcal{L}^{Q_L}\rightarrow\mathbb{R}$ is differentiable at $L_0=1$. Then from \eqref{***},
$$DF_{Q_L}(1)=\partial_1 F\big((Q_L)_\xi,\xi\big)(\in L^\infty_0(\Omega,\mathcal{F},Q_L)),\, Q_L\mbox{-a.s.},$$
is such that
$$F_{Q_L}(L')-F_{Q_L}(1)=E^{Q_L}\big[DF_{Q_L}(1)(L'-1)\big]+o(|L'-1|_{L^1(Q_L)}),$$
for $L'\in\mathcal{L}^{Q_L}$ with $|L'-1|_{L^1(Q_L)}\rightarrow0$. Let us restrict to $L'\in\mathcal{L}^{Q_L}\cap L^2(\Omega,\mathcal{F},Q_L)$ with\newline $|L'-1|_{L^2(Q_L)}\rightarrow0$. Then, putting
$$ \alpha(L'):=F_{Q_L}(L')-F_{Q_L}(1)-E^{Q_L}\big[DF_{Q_L}(1)(L'-1)\big] \big(=o(|L'-1|_{L^1(Q_L)})\big), $$
we have
$$ \frac{\alpha(L')}{|L'-1|_{L^2(Q_L)}}\leq \frac{\alpha(L')}{|L'-1|_{L^1(Q_L)}}\rightarrow0,\ \mbox{ as } |L'-1|_{L^1(Q_L)}\leq |L'-1|_{L^2(Q_L)}\rightarrow0.  $$
Consequently,
$$F_{Q_L}(L')-F_{Q_L}(1)=E^{Q_L}\big[DF_{Q_L}(1)(L'-1)\big]+o(|L'-1|_{L^2(Q_L)}),$$
and from the uniqueness of $\Phi_{Q_L}(1)\in L_0^2(\Omega,\mathcal{F},Q_L)$ with (\ref{PhiQL}) we conclude that
$$\partial_1 F\big((Q_L)_\xi,\xi\big)=DF_{Q_L}(1)=\Phi_{Q_L}(1)=(\partial_\mu G)_1\big((Q_L)_\xi,\xi\big)-E^{Q_L}\big[(\partial_\mu G)_1\big((Q_L)_\xi,\xi\big)\big],\ Q_L\mbox{-a.s.}$$

The following statement is now immediate.

\begin{lemma}\label{Lemma3.4}
Let $f:\mathcal{M}(\mathbb{R}^d)\rightarrow \mathbb{R}$ be given such that the function   $G:\mathcal{P}_{2,0}(\mathbb{R}\times\mathbb{R}^d)\rightarrow\mathbb{R}$ defined for all probability $Q$
by $ G(Q_{(L,\xi)}):=f\big((LQ)_\xi\big),\ (L,\xi)\in L^2(\Omega,\mathcal{F},Q)\times L^0(\Omega,\mathcal{F},Q; \mathbb{R}^d),$
is partially differentiable w.r.t. $(Q_L)_{L'|\xi}$ at $(Q_L)_{(1,\xi)}$, and $F_{Q_L}:\mathcal{L}^{Q_L}\rightarrow\mathbb{R}$ introduced in (\ref{FQL}) is differentiable at $L_0=1$. Then,
$$ \partial_1 F\big((Q_L)_\xi,x\big)=(\partial_\mu G)_1\big((Q_L)_\xi,x\big)-E^{Q_L}\big[(\partial_\mu G)_1\big((Q_L)_\xi,\xi\big)\big],\ x\in\mathbb{R}^d,\ (Q_L)_\xi(dx)\mbox{-a.s.} $$
\end{lemma}
We finish this subsection with the following example.
\begin{example}
Given any $g,h\in C^1_b(\mathbb{R})$ and $\psi\in b\mathcal{B}(\mathbb{R})$, we consider the function $\widetilde{f}:\mathcal{L}^Q\times L^0(\Omega,\mathcal{F},Q;\mathbb{R}^2)\rightarrow\mathbb{R}$ defined as follows:
$$\widetilde{f}(L,\xi):=g\big(E^{Q_L}\big[h\big(E^{Q_L}[\psi(\xi_1)\,|\,\xi_2]\big)\big]\big),\ (L,\xi=(\xi_1,\xi_2))\in\mathcal{L}^{Q}\times L^0(\Omega,\mathcal{F},Q;\mathbb{R}^2).$$
Let us begin with noting that $\widetilde{f}(L,\xi)$ depends on $(L,\xi)$ only through the law $(Q_L)_\xi.$ Indeed, let $(L',\xi'=(\xi'_1,\xi'_2))\in \mathcal{L}^{Q}\times L^0(\Omega,\mathcal{F},Q;\mathbb{R}^2)$ be such that $(Q_{L'})_{\xi'}=(Q_L)_\xi$, and let $\rho\in\mathcal{B}(\mathbb{R})$ be such that $E^{Q_{L'}}[\psi(\xi_1')\,|\,\xi_2']=\rho(\xi_2'),\, Q_{L'}$-a.s. Then, as for all $\varphi\in b\mathcal{B}(\mathbb{R})$,
\begin{equation*}
\begin{split}
&E^{Q_L}\big[E^{Q_L}[\psi(\xi_1)\,|\,\xi_2]\varphi(\xi_2)\big]=E^{Q_L}\big[
\psi(\xi_1)\varphi(\xi_2)\big]=\int_{\mathbb{R}^2} \psi(x_1)\varphi(x_2)(Q_L)_\xi(dx)\\
=&\int_{\mathbb{R}^2} \psi(x_1)\varphi(x_2)(Q_{L'})_{\xi'}(dx)=E^{Q_{L'}}\big[E^{Q_{L'}}[\psi(\xi'_1)\,|\,\xi'_2]\varphi(\xi'_2)\big]\\
=&E^{Q_{L'}}\big[\rho(\xi'_2)\varphi(\xi'_2)\big]=E^{Q_{L}}\big[\rho(\xi_2)\varphi(\xi_2)\big],
\end{split}
\end{equation*}
it follows that $E^{Q_L}[\psi(\xi_1)\,|\,\xi_2]=\rho(\xi_2)$, $Q_L$-a.s., and
\begin{equation*}
\begin{split}
\widetilde{f}(L',\xi')&=g\big(E^{Q_{L'}}\big[h\big(E^{Q_{L'}}[\psi(\xi'_1)\,|\,\xi'_2]\big)\big]\big)=g\big(E^{Q_{L'}}\big[h\big(\rho(\xi_2')\big)\big]\big)=g\big(E^{Q_{L}}\big[h\big(\rho(\xi_2)\big)\big]\big)\\
&=g\big(E^{Q_{L}}\big[h\big(E^{Q_{L}}[\psi(\xi_1)\,|\,\xi_2]\big)\big]\big)
=\widetilde{f}(L,\xi).
\end{split}
\end{equation*}
This shows that $\widetilde{f}:\mathcal{L}^Q\times L^0(\Omega,\mathcal{F},Q;\mathbb{R}^2)\rightarrow\mathbb{R}$ defines functions
$$F_{Q}(L)=f\big((LQ)_\xi\big):=\widetilde{f}(L,\xi),\, L\in\mathcal{L}^{Q}\ (\mbox{for fixed } \xi= L^0(\Omega,\mathcal{F},Q;\mathbb{R}^2)),$$ and
$$G(Q_{(L,\xi)}):=\widetilde{f}(L,\xi),\, (L,\xi=(\xi_1,\xi_2))\in\mathcal{L}^{Q}\times L^0(\Omega,\mathcal{F},Q;\mathbb{R}^2).$$
Let us now suppose that $L\in\mathcal{L}^Q\cap L^2(\Omega,\mathcal{F},Q)$. Then, for all $L'\in \mathcal{L}^{Q_L}\cap L^2(\Omega,\mathcal{F},Q_L)$ and for $\xi\in L^0(\Omega,\mathcal{F},Q;\mathbb{R}^2)$,
\begin{equation*}
\begin{split}
\widetilde{G}_{Q_L,\xi}(L'):&=G\big((Q_L)_{(L',\xi)}\big)=\widetilde{f}(L'L,\xi)=g\big(E^{L'Q_L}\big[h\big(E^{L'Q_L}[\psi(\xi_1)\,|\,\xi_2]\big)\big]\big)\\
&=\displaystyle g\big(E^{Q_L}\big[L'h\big(\frac{E^{Q_L}[L'\psi(\xi_1)\,|\,\xi_2]}{E^{Q_L}[L'\,|\,\xi_2]}\big)\big]\big),
\end{split}
\end{equation*}
and for the Fr\'{e}chet derivative in $L^2$-sense at $L'=1$ of
$$ L^2(\Omega,\mathcal{F},Q_L)\ni L'\rightarrow \widetilde{G}_{Q_L,\xi}(L')=g\big(E^{Q_L}\big[L'h\big(\frac{E^{Q_L}[L'\psi(\xi_1)\,|\,\xi_2]}{E^{Q_L}[L'\,|\,\xi_2]}\big)\big]\big)$$
we have, for all $\eta\in L^2(\Omega,\mathcal{F},Q_L)$,
\begin{equation*}
\begin{split}
D\widetilde{G}_{Q_L,\xi}(1)(\eta)&=g'\big(E^{Q_L}\big[h\big(E^{Q_L}[\psi(\xi_1)\,|\,\xi_2]\big)\big]\big)\cdot E^{Q_L}\Big\{\eta h\big(E^{Q_L}[\psi(\xi_1)\,|\,\xi_2]\big)\\
& \qquad + h'\big(E^{Q_L}[\psi(\xi_1)\,|\,\xi_2]\big)\cdot\big( E^{Q_L}[\eta\psi(\xi_1)\,|\,\xi_2]-E^{Q_L}[\psi(\xi_1)\,|\,\xi_2] E^{Q_L}[\eta\,|\,\xi_2]\big)\Big\}\\
&=E^{Q_L}\Big[g'\big(E^{Q_L}\big[h\big(E^{Q_L}[\psi(\xi_1)\,|\,\xi_2]\big)\big]\big)\\
&  \qquad \times\Big\{h\big(E^{Q_L}[\psi(\xi_1)\,|\,\xi_2]\big)+h'\big(E^{Q_L}[\psi(\xi_1)\,|\,\xi_2]\big)\big(\psi(\xi_1)-E^{Q_L}[\psi(\xi_1)\,|\,\xi_2] \big)\Big\}\eta\Big],
\end{split}
\end{equation*}
i.e.,
\begin{equation*}
\begin{split}
D\widetilde{G}_{Q_L,\xi}(1)
&=g'\big(E^{Q_L}\big[h\big(E^{Q_L}[\psi(\xi_1)\,|\,\xi_2]\big)\big]\big)\\
&  \qquad \times\Big\{h\big(E^{Q_L}[\psi(\xi_1)\,|\,\xi_2]\big)+h'\big(E^{Q_L}[\psi(\xi_1)\,|\,\xi_2]\big)\big(\psi(\xi_1)-E^{Q_L}[\psi(\xi_1)\,|\,\xi_2] \big)
\Big\},
\end{split}
\end{equation*}
and, for $x=(x_1,x_2)\in\mathbb{R}^2$, $(Q_L)_\xi(dx)$-a.s.,
\begin{equation*}
\begin{split}
(\partial_\mu G)_1\big((Q_L)_\xi,x\big)
&=g'\big(E^{Q_L}\big[h\big(E^{Q_L}[\psi(\xi_1)\,|\,\xi_2]\big)\big]\big) \cdot\Big\{h\big(E^{Q_L}[\psi(\xi_1)\,|\,\xi_2=x_2]\big)\\
& \qquad +h'\big(E^{Q_L}[\psi(\xi_1)\,|\,\xi_2=x_2]\big)\big(\psi(x_1)-E^{Q_L}[\psi(\xi_1)\,|\,\xi_2=x_2] \big)
\Big\}.
\end{split}
\end{equation*}
On the other hand, a direct computation also shows that $F_{Q_L}:\mathcal{L}^{Q_L}\rightarrow \mathbb{R}$ is differentiable, and due to Lemma \ref{Lemma3.4} we know that
$$ \partial_1 F\big((Q_L)_\xi,x\big)=(\partial_\mu G)_1\big((Q_L)_\xi,x\big)-E^{Q_L}\big[(\partial_\mu G)_1\big((Q_L)_\xi,\xi\big)\big],\ x\in\mathbb{R}^d,\ (Q_L)_\xi(dx)\mbox{-a.s.} $$
\end{example}

\subsection{Relation between derivative w.r.t. the density and derivative over $\mathcal{P}_2(\mathbb{R})$}
Let $(\Omega,\mathcal{F},Q)$ be the classical Wiener space: $\Omega=C_0([0,T])$ (the space of the continuous functions $\omega:[0,T]\rightarrow\mathbb{R}$ with $\omega(0)=0$), $Q$ the Wiener measure on $(\Omega,\mathcal{B}(\Omega))$,  $\mathcal{F}=\mathcal{B}(\Omega)\vee \mathcal{N}_Q$, and $B=(B_t)_{t\in[0,T]}$ the coordinate process, $B_t(\omega)=\omega(t),\, \omega\in\Omega,\, t\in[0,T]$. Remark that $B=(B_t)_{t\in[0,T]}$ is a Brownian motion under $Q$.
Moreover, by $\mathbb{F}=\mathbb{F}^B\vee \mathcal{N}_Q$ we denote the filtration generated by $B$ and completed by all $Q$-null sets.

Let $d\ge 1.$ We fix arbitrarily $\xi\in L^4(\Omega,\mathcal{F},Q;\mathbb{R}^d)$. Then, for all $L\in\mathcal{L}^Q\cap L^2(\Omega,\mathcal{F},Q)$, the probability law $(Q_L)_\xi$ belongs to $\mathcal{P}_2(\mathbb{R}^d)$. Indeed,
$$ \int_{\mathbb{R}^d}|x|^2(Q_L)_\xi(dx)=E^Q[L|\xi|^2]\leq \big(E^Q[L^2]\big)^{\frac{1}{2}}\big(E^Q[|\xi|^4]\big)^{\frac{1}{2}}<+\infty. $$

Let $f:\mathcal{P}_2(\mathbb{R}^d)\rightarrow\mathbb{R}$ be a continuously differentiable function, $\Lambda\subset\mathbb{R}^m$ a connected subset, and $\Lambda\ni\lambda\rightarrow L^\lambda\in \mathcal{L}^Q\cap L^2(\Omega,\mathcal{F},Q)$ continuously $L^2(Q)$-differentiable. We put $Q^\lambda=L^\lambda Q$. The question which raises is what one can say about the differentiability of $\Lambda\ni\lambda\rightarrow f(Q^\lambda_\xi)$ over $\Lambda$, and what about the derivative $\partial_\lambda f(Q^\lambda_\xi)$. We will see that, although not supposing a priori the differentiability of $F_Q(L)=f\big((LQ)_\xi\big)$ w.r.t. the density $L\in\mathcal{L}^Q$, this differentiability w.r.t. the density will be strongly related, but unlike that studied in Subsection 3.1, it will be the $L^2(Q)$-derivative of $F_Q:\mathcal{L}^Q\cap L^2(\Omega,\mathcal{F},Q)\rightarrow\mathbb{R}$.

The approach we have made in Subsection 3.1 for the differentiation of functions $\mathcal{L}^Q\ni L\rightarrow F_Q(L)=f\big((LQ)_\xi\big)$ in $L^1(Q)$-sense can be translated in a direct way to the differentiation in $L^2(Q)$-sense of functions  $\mathcal{L}^Q\cap L^2(\Omega,\mathcal{F},Q)\ni L\rightarrow F_Q(L)=f\big((LQ)_\xi\big)$. For this reason we present the corresponding results without proofs:

Let $\xi\in L^4(\Omega,\mathcal{F},Q;\mathbb{R}^d)$ ($d\geq1$),  $f:\mathcal{P}_2(\mathbb{R}^d)\rightarrow \mathbb{R}$ and $F_Q(L):=f\big((LQ)_\xi\big),\, L\in\mathcal{L}^Q\cap L^2(\Omega,\mathcal{F},Q)$.
\begin{definition}\label{def3.3}
We say that the function $F_Q:\mathcal{L}^Q\cap L^2(\Omega,\mathcal{F},Q)\rightarrow \mathbb{R}$ is $L^2(Q)$-differentiable in $L\in\mathcal{L}^Q\cap L^2(\Omega,\mathcal{F},Q)$, if there exists some $(DF_Q)(L)\in L(L^2_0(\Omega,\mathcal{F},Q);\mathbb{R})$ such that
\begin{equation}\label{def3a}
F_{Q}(L')-F_Q(L)=(DF_Q)(L)(L'-L)+o(|L'-L|_{L^2(Q)}),
\end{equation}
for $L'\in\mathcal{L}^Q\cap L^2(\Omega,\mathcal{F},Q)\mbox{ with }|L'-L|_{L^2(Q)}\rightarrow 0$. Recall from the beginning of Subsection 3.1 that $L^2_0(\Omega,\mathcal{F}, Q)$ is the subspace of all zero-mean random variables in $L^2(\Omega,\mathcal{F},Q)$.
\end{definition}
With similar arguments as those in Subsection 3.1 it can be shown that, if $F_Q$ is $L^2(Q)$-differentiable, $(DF_Q)(L)\in L(L^2_0(\Omega,\mathcal{F},Q);\mathbb{R})$ is unique. Due to the Riesz Representation Theorem $(DF_Q)(L)\in L(L^2_0(\Omega,\mathcal{F},Q);\mathbb{R})$ can be identified with a unique $DF_Q(L)\in L^2_0(\Omega,\mathcal{F},Q)$. Moreover, there is a Borel function $g:\mathbb{R}^d\rightarrow\mathbb{R}$ such that, $Q$-a.s., $DF_Q(L)=g(\xi)$, and $g$ depends on $(Q,L,\xi)$ only through the laws $(LQ)_\xi$ and $Q_\xi$. In particular, when we use instead of $Q$ the probability $Q_L:=LQ$ and compute the derivative $DF_{Q_L}(1)$, we have, in analogy to Subsection 3.1, but under the condition that $L$ is bounded from above by a constant and bounded from below by a constant strictly greater than zero,
\begin{equation}\label{**4}
\begin{split}
DF_{Q_L}(1)&=DF_Q(L)-E^{Q_L}[DF_Q(L)],\ Q_L\mbox{-a.s.}\, (\sim Q\mbox{-a.s.}),\\
DF_Q(L)&=DF_{Q_L}(1)-E^Q[DF_{Q_L}(1)],\ Q\mbox{-a.s.},
\end{split}
\end{equation}
(Indeed, remark that, if $0<c\le L\le C\in\mathbb{R}_+$, then $L^2(\Omega,\mathcal{F},Q_L)=L^2(\Omega,\mathcal{F},Q)$ and the $L^2$-norms of the both spaces are equivalent) and the Borel function $g:\mathbb{R}^d\rightarrow\mathbb{R}$ with $DF_{Q_L}(1)=g(\xi)$, $Q$-a.s., depends on $(Q, L,\xi)$ only through $(Q_L)_\xi$. This allows to use also here the notation
\begin{equation}\label{*5}
\partial_1 F\big((Q_L)_\xi,x\big):=g(x),\, x\in\mathbb{R}^d.
\end{equation}

\begin{lemma}\label{Lemjing}
Let $\xi\in L^4(\Omega,\mathcal{F},Q;\mathbb{R}^d)$ and $f:\mathcal{P}(\mathbb{R}^d)\rightarrow\mathbb{R}$ be such that
$$ F_Q(L')=f\big((L'Q)_{\xi}\big),\, L'\in\mathcal{L}^Q, $$
is differentiable at $L\in\mathcal{L}^Q\cap L^2(\Omega,\mathcal{F},Q)$. Then the restriction $F_Q:\mathcal{L}^Q\cap L^2(\Omega,\mathcal{F},Q)\rightarrow\mathbb{R}$ is $L^2(Q)$-differentiable in $L$, and the derivative and the $L^2(Q)$-derivative of $F_Q$ at $L$ coincide.
\end{lemma}
\begin{proof}
Let $F:\mathcal{L}^Q\rightarrow\mathbb{R}$ be differentiable in $L\in\mathcal{L}^Q\cap L^2(\Omega,\mathcal{F},Q)$. From Remark \ref{Re3.2} to Theorem \ref{P3} we have for the derivative $DF_Q(L)=\partial_L F\big(Q_\xi,(LQ)_{\xi},\xi\big)\in L_0^\infty(\Omega,\mathcal{F},Q)$. Consider $L'\in\mathcal{L}^Q\cap L^2(\Omega,\mathcal{F},Q)$ with $|L'-L|_{L^2(Q)}\rightarrow0$, then, as $|L'-L|_{L^1(Q)}\leq |L'-L|_{L^2(Q)}\rightarrow0$,
$$ F_Q(L')-F_Q(L)=E^Q\big[(\partial_L F)\big(Q_\xi,(LQ)_{\xi},\xi\big)(L'-L)\big]+R(L',L), $$
where $R(L',L)=o(|L'-L|_{L^1(Q)})$.

Note that
$$ \frac{|R(L',L)|}{|L'-L|_{L^2(Q)}}\leq \frac{|R(L',L)|}{|L'-L|_{L^1(Q)}}\rightarrow0,\ \mbox{ as }(|L'-L|_{L^1(Q)}\leq)|L'-L|_{L^2(Q)}\rightarrow0.  $$
Hence,
$$ F_Q(L')-F_Q(L)=E^Q\big[(\partial_L F)\big(Q_\xi,(LQ)_{\xi},\xi\big)(L'-L)\big]+o(|L'-L|_{L^2(Q)}), $$
i.e., $F_Q:\mathcal{L}^Q\cap L^2(\Omega,\mathcal{F},Q)\rightarrow\mathbb{R}$ is $L^2(Q)$-differentiable and has the derivative $(\partial_L F)\big(Q_\xi,(LQ)_{\xi},\xi\big)\in L_0^\infty(\Omega,\mathcal{F},Q)\subset L^2(\Omega,\mathcal{F},Q)$.
\end{proof}

In order to get a first idea about the answer to the question raised above concerning the differentiability of $\Lambda\ni\lambda\rightarrow f(Q^\lambda_\xi)\big(=f\big((L^\lambda Q)_\xi\big)\big)$ when $f:\mathcal{P}_2(\mathbb{R}^d)\rightarrow \mathbb{R}$ is differentiable, let us consider the following example.

\begin{example}\label{Example3.3}
Let $h\in C^1(\mathbb{R})$ and $\varphi\in C_b^1(\mathbb{R})$, and consider the function $f(\mu):=h\big(\int_{\mathbb{R}}\varphi d\mu\big),\ \mu\in\mathcal{P}_2(\mathbb{R})$. Obviously, $f:\mathcal{P}_2(\mathbb{R})\rightarrow\mathbb{R}$ is differentiable in P.L. Lions' sense (See the recall in Section 2), and
$$ \partial_\mu f(\mu,x)=h'\Big(\int_{\mathbb{R}}\varphi d\mu\Big)\varphi'(x),\ (\mu,x)\in\mathcal{P}_2(\mathbb{R})\times\mathbb{R}. $$
Let $\xi\in L^4(\Omega,\mathcal{F},Q)$ and $\Lambda\ni\lambda\rightarrow L^\lambda\in \mathcal{L}^Q\cap L^2(\Omega,\mathcal{F},Q)$ be continuously $L^2(Q)$-differentiable. We have that $f(Q^\lambda_\xi)=h\big(E^Q[L^\lambda \varphi(\xi)]\big)$, and
$$ \partial_\lambda f(Q^\lambda_\xi)=h'\big(E^Q\big[L^\lambda \varphi(\xi)\big]\big)E^Q\big[\varphi(\xi)\partial_\lambda L^\lambda\big]=E^Q\Big[h'\Big(\int_{\mathbb{R}}\varphi dQ^\lambda_\xi\Big)\varphi(\xi)\partial_\lambda L^\lambda\Big]. $$
Since $\lambda\rightarrow L^\lambda$ is $L^2(Q)$-differentiable, $E^Q[\partial_\lambda L^\lambda]=\partial_\lambda E^Q[L^\lambda]=0$. It follows that
\begin{equation*}
\begin{split}
\partial_\lambda f(Q^\lambda_\xi)&=E^Q\Big[h'\Big(\int_{\mathbb{R}}\varphi dQ^\lambda_\xi\Big)\big(\varphi(\xi)-\varphi(0)\big)\partial_\lambda L^\lambda\Big]=E^Q\Big[\Big(\int_0^\xi h'\big(\int_{\mathbb{R}}\varphi dQ^\lambda_\xi\big)\varphi'(y)dy\Big)\partial_\lambda L^\lambda\Big]\\
&=E^Q\Big[\Big(\int_0^\xi \partial_\mu f(Q^\lambda_\xi,y) dy\Big)\partial_\lambda L^\lambda\Big]=E^{Q^\lambda}\Big[\Big(\int_0^\xi \partial_\mu f(Q^\lambda_\xi,y) dy\Big)\partial_\lambda [\ln L^\lambda]\Big],
\end{split}
\end{equation*}
where we have used in the latter line that $Q^\lambda\{L^\lambda=0\}=E^Q[L^\lambda \mathbf{1}_{\{L^\lambda=0\}}]=0.$

On the other hand, using the notations introduced in Subsection 3.1 and recalling Example \ref{Example3.1} with $L^\lambda=L,\, \lambda\in \Lambda$, we see that  $F_{Q_L}(L')=f\big((L'Q_L)_\xi\big)=h\big(E^{Q_L}\big[L'\varphi(\xi)\big]\big)$ is differentiable w.r.t. the density at $L'=1$, and

\begin{equation}\label{equ3.16}
\begin{split}
\partial_1 F\big((Q_L)_\xi,x\big)&=h'\Big(\int_{\mathbb{R}}\varphi d(Q_L)_\xi\Big)\Big(\varphi(x)-\int_\mathbb{R}\varphi d(Q_L)_\xi\Big)\\
&=h'\Big(\int_{\mathbb{R}}\varphi d(Q_L)_\xi\Big)\Big(\big(\varphi(x)-\varphi(0)\big)-\int_\mathbb{R}\big(\varphi-\varphi(0)\big) d(Q_L)_\xi\Big)\\
&=\int_0^x \partial_\mu f\big((Q_L)_\xi,y\big)dy-E^{Q_L}\Big[\int_0^\xi \partial_\mu f\big((Q_L)_\xi,y\big)dy\Big],\ x\in\mathbb{R}.
\end{split}
\end{equation}
Thus, differentiating \eqref{equ3.16} w.r.t. $x\in\mathbb{R}$ yields
\begin{equation}\label{equ3.17}
\partial_x(\partial_1 F)\big((Q_L)_\xi,x\big)=\partial_\mu f\big((Q_L)_\xi,x\big),\, x\in\mathbb{R},
\end{equation}
i.e., compared with the derivative w.r.t. the density, the derivative w.r.t. the probability measure can be regarded as a second order derivative. Also note that for $x\in\mathbb{R}^d$ ($d\ge 2$) relation (\ref{equ3.17}) can be obtained by taking the gradient w.r.t. $x$ in the first line of equation (\ref{equ3.16}).
\end{example}
In the following computations which we consider first the case of 1-dimensional random variables $\xi$, and we will see that relation (\ref{equ3.17}) holds true also in the general case. We begin our discussion for the case of dimension $d=1$ for a random variable $\xi\in L^4(\Omega,\mathcal{F},Q)$.

\medskip

Let us make the following standard assumptions.

\noindent\textbf{Assumption 1.}\\
\indent 1) $\xi\in L^4(\Omega,\mathcal{F},Q)$;\\
\indent 2) $\Lambda\ni\lambda\rightarrow L^\lambda\in \mathcal{L}^Q\cap L^2(\Omega,\mathcal{F},Q)$ continuously $L^2(Q)$-differentiable;\\
\indent 3) $f:\mathcal{P}_2(\mathbb{R})\rightarrow\mathbb{R}$ is continuously differentiable, and there is some continuity modulus $\rho:\mathbb{R}_+\rightarrow$ \\
\indent \ \ \ \ $\mathbb{R}_+$ continuous, increasing, concave with $\rho(0)=0$ such that\\
\indent \qquad\qquad i) $|f(\mu)-f(\mu')|\leq \rho\big(W_1(\mu,\mu')\big),\ \mu,\ \mu'\in\mathcal{P}_2(\mathbb{R})$;\\
\indent \qquad\qquad ii) $ |\partial_\mu f(\mu,y)-\partial_\mu f(\mu',y)|\leq \rho\big(W_1(\mu,\mu')\big),\ \mu,\ \mu'\in\mathcal{P}_2(\mathbb{R}),\ y\in\mathbb{R}. $\\
\indent \ \ \ \ Moreover, there is some constant $C\in\mathbb{R}_+$ such that \\
\indent \qquad\qquad iii) $|\partial_\mu f(\mu,y)|\le C(1+W_1(\mu,\delta_0)+|y|),\ \mu\in\mathcal{P}_2(\mathbb{R}),\ y\in\mathbb{R}.$

However, to simplify the computations, we will use instead of Assumption 1-3) the following:\\
\textbf{Assumption 2.} Let the Assumption 1-1) and 1-2) be satisfied. Moreover, let $f:\mathcal{P}_2(\mathbb{R})\rightarrow\mathbb{R}$ be continuously differentiable, such that for some constant  $C\in\mathbb{R}_+$:\\
\indent \qquad\qquad i) $|f(\mu)-f(\mu')|\leq CW_1(\mu,\mu'),\ \mu,\ \mu'\in\mathcal{P}_2(\mathbb{R})$;\\
\indent \qquad\qquad ii) $ |\partial_\mu f(\mu,y)-\partial_\mu f(\mu',y)|\leq CW_1(\mu,\mu'),\ \mu,\ \mu'\in\mathcal{P}_2(\mathbb{R}),\ y\in\mathbb{R}$;\\
\indent \qquad\qquad iii) $|\partial_\mu f(\mu,y)|\leq C,\ \mu\in\mathcal{P}_2(\mathbb{R}),\ y\in\mathbb{R}$.

\medskip

We have the following result suggested by Example \ref{Example3.3}.

\begin{theorem}\label{Theorem3.1}
Under Assumption 1, with the notation $Q^\lambda:=L^\lambda Q,\ \lambda\in\Lambda$, we have that the function $\Lambda\ni\lambda\rightarrow f(Q^\lambda_\xi)$ is differentiable, and
\begin{equation*}
\begin{split}
\partial_\lambda f(Q^\lambda_\xi)&=E^Q\Big[\Big(\int_0^\xi\partial_\mu f(Q^\lambda_\xi,y)dy\Big)\partial_\lambda L^\lambda\Big]=E^{Q^\lambda}\Big[\Big(\int_0^\xi\partial_\mu f(Q^\lambda_\xi,y)dy\Big)\partial_\lambda[\ln L^\lambda]\Big].
\end{split}
\end{equation*}
\end{theorem}
As a direct consequence we obtain
\begin{theorem}\label{Theorem3.2}
Under Assumption 1, the function $F_Q(L):=f\big((LQ)_\xi\big),\, L\in\mathcal{L}^Q\cap L^2(\Omega,\mathcal{F},Q)$ is continuously $L^2(Q)$-differentiable,
$$DF_Q(L)=\int_0^\xi\partial_\mu f\big((LQ)_\xi,y\big)dy-E^Q\Big[\int_0^\xi\partial_\mu f\big((LQ)_\xi,y\big)dy\Big],\, Q\mbox{-a.s.},\, L\in\mathcal{L}^Q\cap L^2(\Omega,\mathcal{F},Q),$$
and, for the derivative at $L'=1$ of the function $L'\rightarrow F_{Q_L}(L')=f\big((L'Q_L)_\xi\big)\big(=f\big((L'LQ)_\xi\big)\big),\, L'\in \mathcal{L}^{Q_L}\cap L^2(\Omega,\mathcal{F},Q_L)$,
\begin{equation*}
\begin{split}
DF_{Q_L}(1)&=\displaystyle \int_0^\xi\partial_\mu f\big((Q_L)_\xi,y\big)dy-E^{Q_L}\big[\int_0^\xi\partial_\mu f\big((Q_L)_\xi,y\big)dy\big],\, Q_L\mbox{-a.s., i.e.},\\
\partial_1 F\big((Q_L)_\xi,x\big)&=\displaystyle \int_0^x\partial_\mu f\big((Q_L)_\xi,y\big)dy-E^{Q_L}\big[\int_0^\xi\partial_\mu f\big((Q_L)_\xi,y\big)dy\big],\, (Q_L)_\xi(dx)\mbox{-a.s.},\ x\in\mathbb{R}.
\end{split}
\end{equation*}
Moreover, $\partial_1 F\big((Q_L)_\xi,\cdot\big):\mathbb{R}\rightarrow \mathbb{R}$ is continuously differentiable, and
\begin{equation}\label{3.19bis}
\partial_x(\partial_1 F)\big((Q_L)_\xi,x\big)=\partial_\mu f\big((Q_L)_\xi,x\big),\, x\in\mathbb{R}.
\end{equation}
\end{theorem}
\begin{proof} (of Theorem \ref{Theorem3.2}). For simplicity of the argument we use Assumption 2. The generalisation from Assumption 2 to Assumption 1 is a bit more technical, but uses only standard arguments.
	
Let us fix any $L\in\mathcal{L}^Q\cap L^2(\Omega,\mathcal{F},Q)$. Given any other $L'\in\mathcal{L}^Q\cap L^2(\Omega,\mathcal{F},Q)$ we put $\lambda(s):=s,\, s\in[0,1]$, and $L^s(=L^{\lambda(s)}):=(1-s)L+sL'\in\mathcal{L}^Q\cap L^2(\Omega,\mathcal{F},Q),\,\, s\in[0,1]$, and we remark that the mapping $s\rightarrow(L^s,\partial_sL^s)=\big(L+s(L'-L),L'-L\big)$ is $L^2(Q)$-continuous. Then, as Assumption 2 is satisfied, it follows from Theorem \ref{Theorem3.1} that
\begin{equation}\label{a1}
\begin{split}
&f\big((L'Q)_\xi\big)-f\big((LQ)_\xi\big)=\int_0^1\partial_s f\big((L^sQ)_\xi\big)ds\\
&=E^Q\Big[\int_0^1\Big(\int_0^\xi\partial_\mu f\big((L^sQ)_\xi,y\big)dy \Big)\partial_sL^sds\Big]
=E^Q\Big[\int_0^1\Big(\int_0^\xi\partial_\mu f\big((L^sQ)_\xi,y\big)dy \Big)ds(L'-L)\Big]\\
&=E^Q\Big[\Big(\int_0^\xi\partial_\mu f\big((LQ)_\xi,y\big)dy \Big)(L'-L)\Big]+R_{L,L'},
\end{split}
\end{equation}
where $$R_{L,L'}:=E^Q\Big[\Big(\int_0^1\int_0^\xi\big[\partial_\mu f\big((L^sQ)_\xi,y\big)-\partial_\mu f\big((LQ)_\xi,y\big)\big]dyds \Big)(L'-L)\Big].$$
But from Assumption 2 we have
$$\big|\partial_\mu f\big((L^sQ)_\xi,y\big)-\partial_\mu f\big((LQ)_\xi,y\big) \big|\le CW_1\big((L^sQ)_\xi,(LQ)_\xi\big)\le CE^Q[|\xi||L^s-L|].$$
This latter estimate is an immediate consequence of the Kantorovich-Rubinstein duality. Indeed, for all $h\in \mbox{Lip}_1(\mathbb{R})$ (the space of real Lipschitz continuous functions defined over $\mathbb{R}$ with Lipschitz constant $1$) with $h(0)=0$ (Observe that this implies $|h(x)|\le |x|,\, x\in\mathbb{R}$), we have
$$\displaystyle\int_{\mathbb{R}}hd(L^sQ)_\xi-\int_{\mathbb{R}}hd(LQ)_\xi=E^Q[h(\xi)(L^s-L)]\le E^Q[|\xi||L^s-L|].$$
Consequently,
\begin{equation}\label{a13}
\begin{split}
|R_{L,L'}|&\le\displaystyle E^Q\Big[\Big|\int_0^1\int_0^\xi C E^Q[|\xi||L^s-L|]dyds \Big||L'-L|\Big]\le C\big(E^Q[|\xi||L'-L|]\big)^2\\
&\le CE^Q[|\xi|^2]E^Q[|L'-L|^2].
\end{split}
\end{equation}
As $E^Q[L'-L]=0$, (\ref{a1}) yields
\begin{equation}\begin{array}{ll}\label{a11}
&F_Q(L')-F_Q(L)=f\big((L'Q)_\xi\big)-f\big((LQ)_\xi\big)\\
&=\displaystyle E^Q\Big[\Big(\int_0^\xi\partial_\mu f\big((LQ)_\xi,y\big)dy \Big)(L'-L)\Big]+R_{L,L'}\\
&=\displaystyle E^Q\Big[\Big(\int_0^\xi\partial_\mu f\big((LQ)_\xi,y\big)dy-E^Q\Big[\int_0^\xi\partial_\mu f\big((LQ)_\xi,y\big)dy\Big] \Big)(L'-L)\Big]+R_{L,L'},
\end{array}\end{equation}
for all $L'\in\mathcal{L}^Q\cap L^2(\Omega,\mathcal{F},Q)$. Notice that
$$\int_0^\xi\partial_\mu f\big((LQ)_\xi,y\big)dy-E^Q\big[\int_0^\xi\partial_\mu f\big((LQ)_\xi,y\big)dy\big]\in L^2_0(\Omega,\mathcal{F},Q).$$
Thus, recalling (\ref{a13}) we conclude that
$$ DF_Q(L)=\int_0^\xi\partial_\mu f\big((LQ)_\xi,y\big)dy-E^Q\Big[\int_0^\xi\partial_\mu f\big((LQ)_\xi,y\big)dy\Big],\,  Q\mbox{-a.s.},\, L\in\mathcal{L}^Q\cap L^2(\Omega,\mathcal{F},Q). $$
Now, as $W_1\big((L'Q)_\xi,(LQ)_\xi\big)\leq E^Q[|\xi||L'-L|]$, it follows from above that
$$ E^Q\big[|DF_Q(L')-DF_Q(L)|^2\big]\leq C\big(E^Q[|\xi|^2]\big)^2E^Q[|L'-L|^2],\, L,L'\in \mathcal{L}^Q\cap L^2(\Omega,\mathcal{F},Q).$$
Furthermore, from \eqref{**4} and \eqref{*5} the remaining parts of Theorem \ref{Theorem3.2} follow now easily.
\end{proof}

Let us prove Theorem \ref{Theorem3.1} first for a special case. For simplicity we fix $T=1$. We suppose:\\
\textbf{Assumption 3.} Let $n\geq1$, $0=t_0<t_1<...<t_n=1$, $\Delta_i:=(t_{i-1},t_i]$, $B(\Delta_i):=B_{t_i}-B_{t_{i-1}}$.\\
\indent i) $\xi$ is a smooth Wiener functional of the form:
$$\xi=\varphi\big(B(\Delta_1),...,B(\Delta_n)\big),\ \varphi\in C_b^\infty(\mathbb{R}^n);$$
\indent ii) $\gamma^\lambda$ is a smooth Wiener step process of the form:
$$\gamma_t^\lambda=\sum_{i=1}^n\varphi_{i}^{\lambda}\big(B(\Delta_1),...,B(\Delta_{i-1})\big)I_{\Delta_i}(t),$$
\indent \ \ \ \ where $\varphi_i:\Lambda\times\mathbb{R}^{i-1}\rightarrow\mathbb{R}$ is a bounded Borel function, such that:\\
\indent iia) $\varphi_{i}^{\lambda}:\mathbb{R}^{i-1}\rightarrow\mathbb{R}$ is of class $C^\infty$ and all derivatives of all order are bounded over $\Lambda\times\mathbb{R}^{i-1}$,\\
\indent \ \ \ \ \  $1\leq i\leq n$, and\\
\indent iib) $\Lambda\ni\lambda\rightarrow\gamma^\lambda\in L^2_{\mathbb{F}}([0,1]\times\Omega,dsdQ)$ is continuously $L^2(dsdQ)$-differentiable.

\medskip

For $t\in[0,1]$ and $\lambda\in\Lambda$ we introduce the Dol\'{e}an-Dade exponential
$$\mathcal{E}_t^\lambda:=\exp\Big\{\int_0^t\gamma^{\lambda}_sdB_s-\frac{1}{2}\int_0^t|\gamma^{\lambda}_s|^2ds\Big\},$$
and we note that $\mathcal{E}_t^\lambda\in \mathcal{L}^Q\cap L^{\infty,-}(Q)$, where $L^{\infty,-}(Q):=\cap_{1<p<+\infty}L^p(Q)$. Moreover, we use the following notations: $Q_t^\lambda:=\mathcal{E}_t^\lambda Q$ and $(Q_t^\lambda)_{\xi}=(\mathcal{E}_t^\lambda Q)_{\xi}\in\mathcal{P}_2(\mathbb{R})$.

\begin{proposition}\label{lem1}
Under the Assumptions 2 and 3 the result stated in Theorem \ref{Theorem3.1} holds true, i.e., $\Lambda\ni\lambda\rightarrow f\big((Q_t^\lambda)_\xi\big)$ is differentiable and
$$ \partial_\lambda f\big((Q_t^\lambda)_\xi\big)=E^Q\Big[\Big(\int_0^\xi\partial_\mu f\big((Q_t^\lambda)_\xi,y\big)dy\Big)\partial_\lambda[\mathcal{E}_t^\lambda]\Big].$$
\end{proposition}
The proof of Proposition \ref{lem1} is heavily based on Girsanov transformation and Malliavin calculus. Let us begin with some preliminaries. The interested reader is referred to, e.g., \cite{B94}.

Recall that $\Omega=C_0([0,1])$, and $(\Omega,\mathcal{F},Q)$
is the classical Wiener space with coordinate process $B=(B_t)_{t\in[0,1]}.$ Over $\Omega$ we define a flow of
Girsanov transformations as follows:
\begin{equation}\label{841}
\begin{split}
T^\lambda_s\omega&=\omega+\int_0^{s\wedge\cdot}\gamma^\lambda_r(T^\lambda_r\omega)dr,\ s\in[0,1],\ \omega\in\Omega,\\
A^\lambda_s\omega&=\omega-\int_0^{s\wedge\cdot}\gamma^\lambda_r(\omega)dr,\ s\in[0,1],\ \omega\in\Omega,\, \lambda\in\Lambda.
\end{split}
\end{equation}
Remark that, thanks to our Assumption 3, $T_s^\lambda,A_s^\lambda,\, s\in[0,1],$ are well defined (Indeed, $\omega\rightarrow \gamma_r^\lambda(\omega)$ is bounded and Lipschitz w.r.t. the supremum norm on $\Omega$, uniformly w.r.t. $(r,\lambda)$), $T_s^\lambda,A_s^\lambda:\Omega\rightarrow\Omega$ are bijective, and standard arguments using the fact that $\gamma^\lambda$ is adapted to the filtration generated by $B$ show that $(T_s^\lambda)^{-1}=A_s^\lambda,\ s\in[0,1]$. Moreover, for $\lambda\in\Lambda$ and $t\in[0,1]$ arbitrarily fixed we get from the Girsanov Theorem that, on one hand,  $Q_t^\lambda:=\mathcal{E}_t^{\lambda}Q=Q\circ[T_t^\lambda]^{-1}$, and on the other hand, $B^\lambda:=B(A_t^\lambda)$ is an $(\mathbb{F},Q_t^\lambda)$-Brownian motion. Furthermore, from \eqref{841} we have $\displaystyle B_s^\lambda=B_s(A_t^\lambda)=B_s-\int_0^{s\wedge t}\gamma^\lambda_r dr,\ s\in[0,1]$, and thus also
$$ B_s=B^\lambda_s+\int_0^{s\wedge t}\gamma^\lambda_r(B_{\cdot\wedge r})dr,\ s\in[0,1], $$
where we can substitute in $\gamma^\lambda$ that, due to \eqref{841}, $B=T_t^\lambda(B^\lambda)$. So we can conclude that $\mathbb{F}^{B^\lambda}=\mathbb{F}^B=\mathbb{F}$, and $\mathcal{F}_1^{B^\lambda}=\mathcal{F}_1^{B}=\mathcal{F}$ (completed under $Q$). Here $\mathbb{F}^{B^\lambda}$ denotes the filtration generated by $B^\lambda$ and completed by all $Q$-null sets. This allows to consider $B^\lambda$ as our canonical Brownian motion on $(\Omega,\mathcal{F},Q_t^\lambda)$.

\smallskip

For any random variable $\vartheta\in L^0(\Omega,\mathcal{F},Q_t^\lambda)(=L^0(\Omega,\mathcal{F},Q))$, we define
\begin{equation}\label{*6}
\bar{\vartheta}(\omega):=\vartheta(T_t^\lambda\omega),\ \omega\in\Omega,\ \mbox{ i.e., } \ \vartheta(\omega)=\bar{\vartheta}(A_t^\lambda\omega),\ \omega\in\Omega.
\end{equation}
Then, $\bar{\vartheta}(B^\lambda)=\vartheta(B)=\vartheta$. For $\lambda'\in\Lambda$ we put $\mathcal{E}_s^{\lambda,\lambda'}:=\mathcal{E}_s^{\lambda'}/\mathcal{E}_s^{\lambda},\ s\in[0,1].$ Remark that
$$ Q_t^{\lambda'}=\mathcal{E}_t^{\lambda'}Q=\mathcal{E}_t^{\lambda,\lambda'}(\mathcal{E}_t^\lambda Q)=\mathcal{E}_t^{\lambda,\lambda'}Q_t^\lambda. $$
Moreover, for all $s\in[0,t]$,
\begin{equation}\label{equ3.19}
\begin{split}
\mathcal{E}_s^{\lambda,\lambda'}&=\frac{\mathcal{E}_s^{\lambda'}}{\mathcal{E}_s^{\lambda}}=\exp\Big\{\int_0^s(\gamma^{\lambda'}_r-\gamma^{\lambda}_r)dB_r-\frac{1}{2}\int_0^s\big((\gamma^{\lambda'}_r)^2-(\gamma^{\lambda}_r)^2\big)dr\Big\}\\
&=\exp\Big\{\int_0^s(\gamma^{\lambda'}_r-\gamma^{\lambda}_r)dB^\lambda_r-\frac{1}{2}\int_0^s(\gamma^{\lambda'}_r-\gamma^{\lambda}_r)^2 dr\Big\}\\
&=\exp\Big\{\int_0^s(\bar{\gamma}^{\lambda'}_r-\bar{\gamma}^{\lambda}_r)(B^\lambda)dB^\lambda_r-\frac{1}{2}\int_0^s(\bar{\gamma}^{\lambda'}_r-\bar{\gamma}^{\lambda}_r)^2(B^\lambda) dr\Big\},
\end{split}
\end{equation}
where due to definition \eqref{*6}, $\bar{\gamma}^{\lambda'}_r=\gamma^{\lambda'}_r(T_t^\lambda),\, \lambda'\in\Lambda$. In order to characterise $\mathcal{E}_s^{\lambda,\lambda'}$ defined in (\ref{equ3.19}) as a Girsanov density w.r.t. $Q_t^\lambda$, we proceed similarly to \eqref{841} and introduce the transformation $T_s^{\lambda,\lambda'}:\Omega\rightarrow\Omega$ and its inverse $A_s^{\lambda,\lambda'}=(T^{\lambda,\lambda'}_s)^{-1}:\Omega\rightarrow\Omega$ as follows:
\begin{equation}\label{equ3.20bis}
\begin{split}
T^{\lambda,\lambda'}_s\omega&=\omega+\int_0^{s\wedge\cdot}(\bar{\gamma}^{\lambda'}_r-\bar{\gamma}^\lambda_r)(T^{\lambda,\lambda'}_r\omega)dr,\ s\in[0,t],\ \omega\in\Omega,\\
A^{\lambda,\lambda'}_s\omega&=\omega-\int_0^{s\wedge\cdot}(\bar{\gamma}^{\lambda'}_r-\bar{\gamma}^\lambda_r)(\omega)dr,\ s\in[0,t],\ \omega\in\Omega.
\end{split}
\end{equation}
We remark that, putting
$$\bar{\mathcal{E}}_s^{\lambda,\lambda'}:=\mathcal{E}_s^{\lambda,\lambda'}(T_t^\lambda)=\exp\Big\{\int_0^s(\bar{\gamma}^{\lambda'}_r-\bar{\gamma}^{\lambda}_r)(B)dB_r-\frac{1}{2}\int_0^s(\bar{\gamma}^{\lambda'}_r-\bar{\gamma}^{\lambda}_r)^2(B) dr\Big\},\, 0\leq s\leq t\leq 1,$$
as a consequence of the Girsanov Theorem we also have $Q\circ[T^{\lambda,\lambda'}_s]^{-1}=\bar{\mathcal{E}}_s^{\lambda,\lambda'}Q.$
Furthermore, recalling that $Q^\lambda_t=\mathcal{E}_t^\lambda Q=Q\circ[T_t^\lambda]^{-1}$, and \eqref{*6}, we see that, for all $\Phi\in b\mathcal{B}(\Omega)$ and $s\in[0,t]$,
\begin{equation*}
\begin{split}
&E^{Q_t^\lambda}[\mathcal{E}_s^{\lambda,\lambda'}\Phi]=E^{Q}[(\mathcal{E}_s^{\lambda,\lambda'}\Phi)(T_t^\lambda)]=E^Q[\bar{\mathcal{E}}_s^{\lambda,\lambda'}\bar{\Phi}]=E^Q[\bar{\Phi}(T^{\lambda,\lambda'}_s)]\\
&=E^Q[\mathcal{E}_t^{\lambda}\bar{\Phi}(T^{\lambda,\lambda'}_s A_t^\lambda)]=E^{Q_t^\lambda}[\bar{\Phi}(T^{\lambda,\lambda'}_s A_t^\lambda)]=E^{Q_t^\lambda}[\Phi(T_t^\lambda T^{\lambda,\lambda'}_s A_t^\lambda)].
\end{split}
\end{equation*}
This shows that
\begin{equation}\label{equ3.20}
Q_t^\lambda \circ[T_t^\lambda T^{\lambda,\lambda'}_s A_t^\lambda]^{-1}=\mathcal{E}_s^{\lambda,\lambda'}Q_t^\lambda,\ 0\leq s\leq t.
\end{equation}
This formula (\ref{equ3.20}) is very crucial in our approach, because it allows to translate the derivative w.r.t. the probability density $\mathcal{E}^{\lambda'}_t$ to that w.r.t. a probability law.

Let us begin with observing that, for all $\vartheta\in L^0(\Omega,\mathcal{F},Q)$, $0\leq s\leq t\leq 1$,
\begin{equation}\label{equ3.21}
(\mathcal{E}_s^{\lambda,\lambda'}Q_t^\lambda)_{\vartheta}=(Q_t^\lambda)_{\bar{\vartheta}\big(T_s^{\lambda,\lambda'}(B^\lambda)\big)},
\end{equation}
and for $s=t$,
\begin{equation}\label{*7}
(Q_t^{\lambda'})_{\vartheta}=(\mathcal{E}_t^{\lambda'} Q)_{\vartheta}=(\mathcal{E}_t^{\lambda,\lambda'}Q_t^\lambda)_{\vartheta}=(Q_t^\lambda)_{\bar{\vartheta}\big(T^{\lambda,\lambda'}_t(B^\lambda)\big)}.
\end{equation}
Indeed, from (\ref{equ3.20}) we get that, for all $\varphi\in b\mathcal{B}(\mathbb{R})$, and for $0\leq s\leq t\leq 1$,
\begin{equation*}
\begin{split}
E^{Q_t^\lambda}\big[\varphi\big(\bar{\vartheta}(T^{\lambda,\lambda'}_s(B^\lambda))\big)\big]&=E^{Q_t^\lambda}\big[\varphi\big(\vartheta(T_t^\lambda T^{\lambda,\lambda'}_sA_t^\lambda)\big)\big]=E^{Q_t^\lambda}[\mathcal{E}_s^{\lambda,\lambda'} \varphi(\vartheta)].
\end{split}
\end{equation*}
Thus, if $s=t$, we have $(Q_t^{\lambda'})_{\vartheta}=(\mathcal{E}_t^{\lambda'} Q)_{\vartheta}=(\mathcal{E}_t^{\lambda,\lambda'}Q_t^\lambda)_{\vartheta}=(Q_t^\lambda)_{\bar{\vartheta}\big(T^{\lambda,\lambda'}_t(B^\lambda)\big)}$.
Formula \eqref{*7} proves the following statement:
\begin{proposition}\label{Proposition3.2}
With the above notations, given any function $f:\mathcal{P}_2(\mathbb{R})\rightarrow\mathbb{R}$, we have
\begin{equation}\label{equ3.22}
\begin{split}
f\big((\mathcal{E}_t^{\lambda'}Q)_{\xi}\big)-f\big((\mathcal{E}_t^{\lambda}Q)_{\xi}\big)
&=f\big((\mathcal{E}_t^{\lambda,\lambda'}Q_t^\lambda)_{\xi}\big)-f\big((Q_t^\lambda)_{\xi}\big)\\
&=f\big((Q_t^\lambda)_{\bar{\xi}\big(T^{\lambda,\lambda'}_t(B^\lambda)\big)}\big)-f\big((Q_t^\lambda)_{\bar{\xi}(B^\lambda)}\big),\, \lambda'\in\Lambda.
\end{split}
\end{equation}
\end{proposition}
The importance of this proposition stems from the fact that it allows to translate the derivative w.r.t. the density of $\lambda'\rightarrow f\big((\mathcal{E}_t^{\lambda'}Q)_{\xi}\big)$ at $\lambda'=\lambda$ to a derivative w.r.t. the law of $\lambda'\rightarrow f\big((Q_t^\lambda)_{\bar{\xi}\big(T^{\lambda,\lambda'}_t(B^\lambda)\big)}\big)$.

\medskip

Before we can give the proof of Proposition \ref{lem1}, we still have to make some preparation beginning with the recall of the notion of Malliavin derivative. We restrict here to the strict necessary, and so we will work with the Malliavin derivative namely on the space of smooth Wiener functionals. We define the space $\mathcal{S}$ of smooth Wiener functionals by
\begin{equation*}
\begin{split}
\mathcal{S}:=&\big\{ \varphi\big(B(\Delta_1),...,B(\Delta_n)\big),\ n\geq1\,\big|\, 0=t_0<t_1<...<t_n=1,\\
 &\hskip 5.3cm \Delta_i=(t_{i-1},t_i],\, B(\Delta_i)=B_{t_i}-B_{t_{i-1}},\, \varphi\in C_b^\infty(\mathbb{R}^n)\big\},
\end{split}
\end{equation*}
which is a dense subset of $L^2(\Omega,\mathcal{F},Q)$. For $\xi\in\mathcal{S}$, the Malliavin derivative $D=(D_s)_{s\in[0,1]}$ is defined as follows:
$$ D_s\xi=\sum_{i=1}^n(\partial_{x_i}\varphi)\big(B(\Delta_1),..., B(\Delta_n)\big)I_{\Delta_i}(s),\ s\in[0,1]. $$
For $\xi\in\mathcal{S}$, $D\xi=(D_s\xi)_{s\in[0,1]}\in L^2(dsdQ)(:= L^2([0,1]\times\Omega,dsdQ))$.
\begin{remark}
$D:\mathcal{S}\rightarrow L^2(dsdQ)$ is a closable map, i.e., for all $F\in L^2(\Omega,\mathcal{F},Q)$ it holds that, if there exists a sequence $(F_j)_{j\geq1}\subset\mathcal{S}$, such that \\
\indent	${\rm i)}$ $F_j\rightarrow F$ in $L^2(Q)$, as $j\rightarrow+\infty$, and\\
\indent	${\rm ii)}$ $(DF_j)_{j\geq1}$ is a Cauchy sequence in $L^2(dsdQ)$, \\
	then the $L^2(dsdQ)$-limit of $(DF_j)_{j\ge 1}$ depends on $F$ but not on the special choice of $(F_j)_{j\geq1}$.
	
Consequently, for $F\in L^2(\Omega,\mathcal{F},Q)$ for which there exists such a Cauchy sequence $(F_j)_{j\ge 1}\subset\mathcal{S}$ satisfying {\rm i)} and {\rm ii)} (We write: $F\in\mathbb{D}^{1,2}(Q)$) one can define $DF:=L^2(dsdQ)$-$\lim\limits_{j\rightarrow\infty}DF_j$. This extends $D$ from $\mathcal{S}$ to $\mathbb{D}^{1,2}(Q)$. We remark that $\mathbb{D}^{1,2}(Q)\subsetneqq L^2(\Omega,\mathcal{F},Q)$.
\end{remark}
\begin{lemma}\label{Lemma3.5bis}
Under Assumption 3 we have that the mapping $s\rightarrow \bar{\xi}\big(T^{\lambda',\lambda}_s(B^\lambda)\big)$ is continuous on $[0,1]$ and differentiable in all $s\in(t_{i-1},t_i),\, 1\le i\le n$:
$$ \partial_s\big[\bar{\xi}\big(T^{\lambda',\lambda}_s(B^\lambda)\big)\big]
=(D_s\bar{\xi})\big(T_s^{\lambda,\lambda'}(B^\lambda)\big)\cdot (\bar{\gamma}^{\lambda'}_s-\bar{\gamma}^\lambda_s)\big(T_s^{\lambda, \lambda'}(B^\lambda)\big).$$
\end{lemma}
\begin{proof} (of Lemma \ref{Lemma3.5bis}).
Recall that, due to Assumption 3, for some $\psi\in C_b^\infty(\mathbb{R}^n)$,
$\xi=\psi\big(B(\Delta_1),\dots,$ $B(\Delta_n)\big)$. On the other hand, from (\ref{841}) we see that
$$ B(T_s^\lambda)(\Delta_i)=B(\Delta_i)+\int_{[0,s]\cap\Delta_i}\varphi_{i}^ {\lambda}\big(B(T_r^\lambda)(\Delta_1),\dots, B(T_r^\lambda)(\Delta_{i-1})\big)dr,\, s\in[0,1],\, 1\le i\le n, $$
from where we get that, for some $\Phi\in C^\infty_{\ell,b}(\mathbb{R}^n;\mathbb{R}^n)$ (the space of smooth functions from $\mathbb{R}^n$ to $\mathbb{R}^n$ with at most linear growth and bounded derivatives of all order greater or equal to the first one),
$$\big(B(T_t^\lambda)(\Delta_1),\dots,B(T_t^\lambda)(\Delta_n)\big)=\Phi\big(B(\Delta_1),\dots,B(\Delta_n)\big).$$
Thus,
\begin{equation}\label{psibar}
\bar{\xi}=\xi(T_t^\lambda)=\psi\big(B(T_t^\lambda)(\Delta_1),\dots,B(T_t^\lambda)(\Delta_n)\big)=\widehat{\psi}\big(B(\Delta_1),\dots,B(\Delta_n)\big)\in\mathcal{S},
\end{equation}
for $\widehat{\psi}=\psi\circ\Phi\in C_b^\infty(\mathbb{R}^n)$. Now, with a similar argument, we deduce from (\ref{equ3.20bis})
\begin{equation*}
\begin{split}
&\bar{\xi}\big(T_s^{\lambda,\lambda'}(B^\lambda)\big)=\bar{\xi}(T_s^{\lambda,\lambda'}\circ A^\lambda_t)\\
&=\widehat{\psi}\big(B(T_s^{\lambda,\lambda'})(\Delta_1),\dots,B(T_s^{\lambda,\lambda'})(\Delta_n)\big)\circ A_t^\lambda\\
&=\widehat{\psi}\bigg(B(\Delta_1)+\int_{[0,s]\cap\Delta_1}(\bar{\gamma}^{\lambda'}_r-\bar{\gamma}_r^\lambda)(T^{\lambda',\lambda}_r)dr,\dots, B(\Delta_n)+\int_{[0,s]\cap\Delta_n}(\bar{\gamma}^{\lambda'}_r-\bar{\gamma}_r^\lambda)(T^{\lambda',\lambda}_r)dr \bigg)\circ A_t^\lambda.
\end{split}
\end{equation*}
We observe that the latter expression is differentiable w.r.t. $s\in (t_{i-1},t_i),\, 1\leq i\leq n$, and that, using the definition of the Malliavin derivative for smooth Wiener functionals we get
\begin{equation*}
\begin{split}
&\partial_s\big[\bar{\xi}\big(T^{\lambda',\lambda}_s(B^\lambda)\big)\big]\\
&=\Big\{\sum_{i=1}^n(\partial_{x_i}\widehat{\psi})
\Big(B(\Delta_1)+\int_{[0,s]\cap\Delta_1}(\bar{\gamma}^{\lambda'}_r-\bar{\gamma}_r^\lambda)(T^{\lambda',\lambda}_r)dr,\dots,\\
&\hskip 3cm B(\Delta_n)+\int_{[0,s]\cap\Delta_n}(\bar{\gamma}^{\lambda'}_r-\bar{\gamma}_r^\lambda)(T^{\lambda',\lambda}_r)dr \Big)
\cdot(\bar{\gamma}^{\lambda'}_s-\bar{\gamma}^\lambda_s)(T_s^{\lambda,\lambda'})I_{\Delta_i}(s)\Big\}\circ A^\lambda_t\\
&=\big\{(D_s\bar{\xi})(T_s^{\lambda,\lambda'})\cdot(\bar{\gamma}^{\lambda'}_s-\bar{\gamma}^\lambda_s)(T_s^{\lambda,\lambda'})\big\}\circ A_t^\lambda\\
&=(D_s\bar{\xi})\big(T_s^{\lambda,\lambda'}(B^\lambda)\big)\cdot(\bar{\gamma}^{\lambda'}_s-\bar{\gamma}^\lambda_s)\big(T_s^{\lambda,\lambda'}(B^\lambda)\big).
\end{split}
\end{equation*}
\end{proof}
Let us come now to the proof of Proposition \ref{lem1}.
\begin{proof} (of Proposition \ref{lem1}).
Let  $f:\mathcal{P}_2(\mathbb{R})\rightarrow\mathbb{R}$ satisfy Assumption 2. Then, applying Proposition \ref{Proposition3.2}, we get
\begin{equation}\label{equ3.23}
\begin{split}
&f\big((\mathcal{E}_t^{\lambda'}Q)_{\xi}\big)-f\big((\mathcal{E}_t^{\lambda}Q)_{\xi}\big)=f\big((Q_t^\lambda)_{\bar{\xi}\big(T^{\lambda,\lambda'}_t(B^\lambda)\big)}\big)-f\big((Q_t^\lambda)_{\bar{\xi}(B^\lambda)}\big)
=\int_0^t\partial_s\big[f\big((Q_t^\lambda)_{\bar{\xi}\big(T^{\lambda,\lambda'}_s(B^\lambda)\big)}\big)\big]ds\\
=&\int_0^t E^{Q_t^\lambda}\Big[(\partial_\mu f)\Big((Q_t^\lambda)_{\bar{\xi}\big(T^{\lambda,\lambda'}_s(B^\lambda)\big)},\bar{\xi}\big(T^{\lambda,\lambda'}_s(B^\lambda)\big)\Big)\partial_s\big[\bar{\xi}\big(T^{\lambda,\lambda'}_s(B^\lambda)\big)\big]\Big]ds.
\end{split}
\end{equation}
Indeed, the latter equality follows from the fact that the function $f$ lifted to $L^2(\Omega,\mathcal{F},Q^\lambda_t)$, $\widetilde{f}:L^2(\Omega,\mathcal{F},Q^\lambda_t)\rightarrow \mathbb{R}$ is Fr\'{e}chet differentiable. This Fr\'echet differentiability is a consequence of the definition of the differentiability of $f:\mathcal{P}_2(\mathbb{R})\rightarrow\mathbb{R}$ combined with Lemma \ref{LemA.1} in the Appendix. Thus,
$$ f\big((Q_t^\lambda)_{\bar{\xi}\big(T^{\lambda,\lambda'}_t(B^\lambda)\big)}\big)-f\big((Q_t^\lambda)_{\bar{\xi}(B^\lambda)}\big)
=\widetilde{f}\big(\bar{\xi}\big(T^{\lambda,\lambda'}_t(B^\lambda)\big) \big)
-\widetilde{f}\big(\bar{\xi}(B^\lambda)\big),$$
and as $s\rightarrow \bar{\xi}\big(T^{\lambda,\lambda'}_s(B^\lambda)\big)$ is differentiable w.r.t. $s\in(t_{i-1},t_i),\, 1\le i\le n$, with $
\partial_s\big[\bar{\xi}\big(T^{\lambda',\lambda}_s(B^\lambda)\big)\big]$ $
=(D_s\bar{\xi})\big(T_s^{\lambda,\lambda'}(B^\lambda)\big)\cdot (\bar{\gamma}^{\lambda'}_s-\bar{\gamma}^\lambda_s)\big(T_s^{\lambda, \lambda'}(B^\lambda)\big)$ (Recall Lemma \ref{Lemma3.5bis}) and $\bar{\xi}\in\mathcal{S}$, $s\rightarrow \bar{\xi}\big(T^{\lambda,\lambda'}_s(B^\lambda)\big)$ is also
$L^2(Q_t^\lambda)$-differentiable. Hence, $s\rightarrow \widetilde{f}\big(\bar{\xi}\big(T^{\lambda,\lambda'}_s(B^\lambda)\big)\big)\in L^2(Q_t^\lambda)$ is differentiable, and
\begin{equation*}
\begin{split}
\partial_s\widetilde{f}\big(\bar{\xi}\big(T^{\lambda,\lambda'}_s(B^\lambda)\big)\big)&
=D\widetilde{f}\big(\bar{\xi}\big(T^{\lambda,\lambda'}_s (B^\lambda)\big)\big)\big(\partial_s\big[\bar{\xi}\big(T^{\lambda,\lambda'}_s(B^\lambda)\big)\big]\big)\\
&=E^{Q_t^\lambda}\Big[(\partial_\mu f)\Big((Q_t^\lambda)_{\bar{\xi}\big(T^{\lambda,\lambda'}_s(B^\lambda)\big)},\bar{\xi}\big(T^{\lambda,\lambda'}_s(B^\lambda)\big)\Big)\partial_s\big[\bar{\xi}\big(T^{\lambda,\lambda'}_s(B^\lambda)\big)\big] \Big].
\end{split}
\end{equation*}
Hence, combining the above relations we obtain (\ref{equ3.23}). Moreover, using Lemma \ref{Lemma3.5bis} we see that

\begin{equation*}
\begin{split}
&f\big((\mathcal{E}_t^{\lambda'}Q)_{\xi}\big)-f\big((\mathcal{E}_t^{\lambda}Q)_{\xi}\big)\\
&=\int_0^t E^{Q_t^\lambda}\Big[(\partial_\mu f)\Big((Q_t^\lambda)_{\bar{\xi}\big(T^{\lambda,\lambda'}_s(B^\lambda)\big)},\bar{\xi}\big(T^{\lambda,\lambda'}_s(B^\lambda)\big)\Big)\partial_s\big[\bar{\xi}\big(T^{\lambda,\lambda'}_s(B^\lambda)\big)\big]\Big]ds\\
&=\int_0^t E^{Q_t^\lambda}\Big[(\partial_\mu f)\Big((Q_t^\lambda)_{\bar{\xi}\big(T^{\lambda,\lambda'}_s(B^\lambda)\big)},\bar{\xi}\big(T^{\lambda,\lambda'}_s(B^\lambda)\big)\Big)(D_s\bar{\xi})\big(T_s^{\lambda,\lambda'}(B^\lambda)\big)(\bar{\gamma}_s^{\lambda'}-\bar{\gamma}_s^\lambda)\big(T_s^{\lambda,\lambda'}(B^\lambda)\big)\Big]ds.\\
\end{split}
\end{equation*}
We recall that $Q\circ[T^{\lambda,\lambda'}_s]^{-1} =\bar{\mathcal{E}}_s^{\lambda,\lambda'}Q$, and we observe that, for all $\vartheta\in L_{+}^0(\Omega,\mathcal{F},Q)$,
$$ E^{Q_t^\lambda}\big[\bar{\vartheta}\big(T^{\lambda,\lambda'}_s(B^\lambda)\big)\big]=E^{Q_t^\lambda}\big[\bar{\vartheta}(T^{\lambda,\lambda'}_s \circ A_t^\lambda)\big]=E^{Q}\big[\bar{\vartheta}(T^{\lambda, \lambda'}_s)\big]=E^Q[\bar{\mathcal{E}}_s^{\lambda,\lambda'}\bar{\vartheta}]. $$
Then we get from the above relation
\begin{equation*}
\begin{split}
&f\big((\mathcal{E}_t^{\lambda'}Q)_{\xi}\big)-f\big((\mathcal{E}_t^{\lambda}Q)_{\xi}\big)\\
&=\int_0^t E^{Q_t^\lambda}\Big[(\partial_\mu f)\Big((Q_t^\lambda)_{\bar{\xi}\big(T^{\lambda,\lambda'}_s(B^\lambda)\big)},\bar{\xi}\big(T^{\lambda,\lambda'}_s(B^\lambda)\big)\Big)(D_s\bar{\xi})\big(T_s^{\lambda,\lambda'}(B^\lambda)\big)(\bar{\gamma}_s^{\lambda'}-\bar{\gamma}_s^\lambda)\big(T_s^{\lambda,\lambda'}(B^\lambda)\big)\Big]ds\\
&=\int_0^t E^{Q}\Big[(\partial_\mu f)\Big((Q_t^\lambda)_{\bar{\xi}\big(T^{\lambda,\lambda'}_s(B^\lambda)\big)},\bar{\xi}\Big)D_s[\bar{\xi}](\bar{\gamma}_s^{\lambda'}-\bar{\gamma}_s^\lambda)\cdot\bar{\mathcal{E}}_s^{\lambda,\lambda'}\Big]ds\\
&=I_{\lambda,\lambda'}+R_{\lambda,\lambda'},
\end{split}
\end{equation*}
where we have put $\displaystyle I_{\lambda,\lambda'}=E^{Q}\Big[\int_0^t(\partial_\mu f)\Big((Q_t^\lambda)_{\xi},\bar{\xi}\Big)D_s[\bar{\xi}](\bar{\gamma}_s^{\lambda'}-\bar{\gamma}_s^\lambda)\bar{\mathcal{E}}_s^{\lambda,\lambda'}ds\Big], $\emph{} and
$$ R_{\lambda,\lambda'}=E^{Q}\Big[\int_0^t\Big\{(\partial_\mu f)\Big((Q_t^\lambda)_{\bar{\xi}\big(T^{\lambda,\lambda'}_s(B^\lambda)\big)},\bar{\xi}\Big)-(\partial_\mu f)\big((Q_t^\lambda)_{\xi},\bar{\xi}\big)\Big\}D_s[\bar{\xi}](\bar{\gamma}_s^{\lambda'}-\bar{\gamma}_s^\lambda)\bar{\mathcal{E}}_s^{\lambda,\lambda'}ds\Big]. $$
\underline{Computation for $I_{\lambda,\lambda'}$}. Let $D^\lambda$ denote the Malliavin derivative w.r.t. the $Q_t^\lambda$-Brownian motion $B^\lambda$ instead of $B$; it has the same definition as the Malliavin derivative $D$, only that the Brownian motion $B$ under $Q$ is replaced by the Brownian motion $B^\lambda$ under $Q^\lambda_t$. Recall from (\ref{psibar}) that for our $\xi=\psi\big(B(\Delta_1),\dots,B(\Delta_n)\big)\in\mathcal{S}$ ($\psi\in C^\infty_b(\mathbb{R}^n)$) there is some $\widehat{\psi}\in C_b^\infty(\mathbb{R}^n)$ such that $\bar{\xi}:=\xi(T_t^\lambda)=\widehat{\psi}\big(B(\Delta_1),\dots,B(\Delta_n)\big)$, and thus,
\begin{equation}\label{*8}
\begin{split}
&D_s^\lambda\xi=D_s^\lambda[\bar{\xi}(B^\lambda)]=D_s^\lambda\big[\widehat{\psi}\big(B^\lambda(\Delta_1),\dots,B^\lambda(\Delta_n)\big)\big]\\
=&\sum_{i=1}^n(\partial_{x_i}\widehat{\psi})\big(B^\lambda(\Delta_1),\dots,B^\lambda(\Delta_n)\big)I_{\Delta_i}(s)=(D_s\bar{\xi})(B^\lambda),\ s\in[0,1].
\end{split}
\end{equation}
Then, using this relation after having applied Girsanov transformation we obtain
\begin{equation}\label{equ**}
\begin{split}
I_{\lambda,\lambda'}&=E^{Q}\Big[\int_0^t(\partial_\mu f)\Big((Q_t^\lambda)_{\xi},\bar{\xi}\Big)D_s[\bar{\xi}](\bar{\gamma}_s^{\lambda'}-\bar{\gamma}_s^\lambda)\bar{\mathcal{E}}_s^{\lambda,\lambda'}ds\Big]\\
&=E^{Q_t^\lambda}\Big[\int_0^t(\partial_\mu f)\Big((Q_t^\lambda)_{\xi},\bar{\xi}(B^\lambda)\Big)D_s[\bar{\xi}](B^\lambda)(\bar{\gamma}_s^{\lambda'}-\bar{\gamma}_s^\lambda)(B^\lambda)\bar{\mathcal{E}}_s^{\lambda,\lambda'}(B^\lambda)ds\Big]\\
&=E^{Q_t^\lambda}\Big[\int_0^t(\partial_\mu f)\big((Q_t^\lambda)_{\xi},\xi\big)D_s^\lambda[\xi](\gamma_s^{\lambda'}-\gamma_s^\lambda)\mathcal{E}_s^{\lambda,\lambda'}ds\Big].\\
\end{split}
\end{equation}
(Recall that $\bar{\mathcal{E}}_s^{\lambda,\lambda'}(B^\lambda)=\mathcal{E}_s^{\lambda,\lambda'}$). Now, observing that
$$ D_s^\lambda\Big\{\int_0^\xi(\partial_\mu f)\big((Q_t^\lambda)_{\xi},y\big)dy\Big\}=(\partial_\mu f)\big((Q_t^\lambda)_{\xi},\xi\big)D_s^\lambda[\xi],\ s\in[0,1],$$
and using the duality between the Malliavin derivative $D^\lambda$ and the It\^{o} integral w.r.t. the $Q^\lambda_t$-Brownian motion $B^\lambda$, it follows that
\begin{equation*}
\begin{split}
I_{\lambda,\lambda'}&=E^{Q_t^\lambda}\Big[\int_0^t(\partial_\mu f)\big((Q_t^\lambda)_{\xi},\xi\big)D_s^\lambda[\xi](\gamma_s^{\lambda'}-\gamma_s^\lambda)\mathcal{E}_s^{\lambda,\lambda'}ds\Big]\\
&=E^{Q_t^\lambda}\Big[\int_0^t D_s^\lambda\Big\{\int_0^\xi(\partial_\mu f)\big((Q_t^\lambda)_{\xi},y\big)dy\Big\}(\gamma_s^{\lambda'}-\gamma_s^\lambda)\mathcal{E}_s^{\lambda,\lambda'}ds\Big]\\
&=E^{Q_t^\lambda}\Big[\Big\{\int_0^\xi(\partial_\mu f)\big((Q_t^\lambda)_{\xi},y\big)dy\Big\}\int_0^t (\gamma_s^{\lambda'}-\gamma_s^\lambda)\mathcal{E}_s^{\lambda,\lambda'}dB_s^\lambda\Big].
\end{split}
\end{equation*}
Then, from (\ref{equ3.19}) and another Girsanov transformation it follows that
\begin{equation*}
\begin{split}
I_{\lambda,\lambda'}&=E^{Q_t^\lambda}\Big[\Big\{\int_0^\xi(\partial_\mu f)\big((Q_t^\lambda)_{\xi},y\big)dy\Big\}\int_0^t (\gamma_s^{\lambda'}-\gamma_s^\lambda)\mathcal{E}_s^{\lambda,\lambda'}dB_s^\lambda\Big]\\
&=E^{Q_t^\lambda}\Big[\Big\{\int_0^\xi(\partial_\mu f)\big((Q_t^\lambda)_{\xi},y\big)dy\Big\}(\mathcal{E}_t^{\lambda,\lambda'}-1)\Big]\\
&=E^{Q}\Big[\Big\{\int_0^\xi(\partial_\mu f)\big((Q_t^\lambda)_{\xi},y\big)dy\Big\}(\mathcal{E}_t^{\lambda'}-\mathcal{E}_t^{\lambda})\Big].
\end{split}
\end{equation*}
Here we need the continuous $L^1(Q)$-differentiability of $\Lambda\ni\lambda\rightarrow\mathcal{E}_t^{\lambda}$; its proof will be given later in Lemma \ref{lem3.6}. Thanks to this $L^1(Q)$-differentiability it holds that
$$ \big|\mathcal{E}_t^{\lambda'}-\mathcal{E}_t^{\lambda}-(\lambda'-\lambda)\partial_\lambda\mathcal{E}_t^{\lambda}\big|_{L^1(Q)}=o(|\lambda'-\lambda|),\ \lambda'\rightarrow\lambda. $$
This allows to deduce that
$$ I_{\lambda,\lambda'}=(\lambda'-\lambda)E^{Q}\Big[\Big\{\int_0^\xi\partial_\mu f\big((Q_t^\lambda)_{\xi},y\big)dy\Big\}\partial_\lambda\mathcal{E}_t^{\lambda}\Big]+o(|\lambda'-\lambda|). $$
Now we turn to the estimate for $R_{\lambda,\lambda'}$.

\smallskip

\noindent\underline{Estimate of $R_{\lambda,\lambda'}$}. Let $h\in\mbox{Lip}_1(\mathbb{R})$ with $h(0)=0$. Then, by iterating a Girsanov transformation argument (in particular (\ref{equ3.20})) and using that $\xi\in\mathcal{S}$ is bounded, we get
\begin{equation}\label{**}
\begin{split}
&\ \Big|\int_{\mathbb{R}} h d(Q_t^\lambda)_{\bar{\xi}\big(T^{\lambda,\lambda'}_s(B^\lambda)\big)}-\int_{\mathbb{R}} h d(Q_t^\lambda)_{\xi}\Big|
=\big|E^{Q_t^\lambda}\big[h\big(\bar{\xi}\big(T^{\lambda,\lambda'}_s(B^\lambda)\big)\big)\big]-E^{Q_t^\lambda}\big[h(\xi)\big]\big|\\
=&\ \big|E^{Q_t^\lambda}\big[\mathcal{E}_s^{\lambda,\lambda'}h(\xi)\big]-E^{Q_t^\lambda}\big[h(\xi)\big]\big|=\big|E^{Q}\big[\mathcal{E}_t^\lambda(\mathcal{E}_s^{\lambda,\lambda'}-1)h(\xi)\big]\big|\\
\leq&\ |\xi|_{L^\infty}E^Q\big[\mathcal{E}_t^\lambda\big|\mathcal{E}_s^{\lambda,\lambda'}-1\big|\big]=|\xi|_{L^\infty}E^Q\big[\big|\mathcal{E}_s^{\lambda'}-\mathcal{E}_s^{\lambda}\big|\big].
\end{split}
\end{equation}
Again from the $L^1(Q)$-differentiability of $\lambda'\rightarrow\mathcal{E}_t^{\lambda'}$ we have that, for $\delta_1>0$ small enough,
$$ \frac{1}{|\lambda'-\lambda|}\big|\mathcal{E}_t^{\lambda'}-\mathcal{E}_t^{\lambda}-(\lambda'-\lambda)\partial_\lambda\mathcal{E}_t^{\lambda}\big|_{L^1(Q)}\leq1,\, \mbox{ for all } \lambda'\in\Lambda\, \mbox { with }\, 0<|\lambda'-\lambda|\leq \delta_1,$$
and thus,  $\big|\mathcal{E}_t^{\lambda'}-\mathcal{E}_t^{\lambda}\big|_{L^1(Q)}\leq\big(1+|\partial_\lambda\mathcal{E}_t^{\lambda}|_{L^1(Q)}\big)|\lambda'-\lambda|$, for all $\lambda'\in\Lambda$ with  $|\lambda'-\lambda|\leq \delta_1$.
Consequently, from (\ref{**}),
\begin{equation}\label{*****}
\begin{split}
\Big|\int_{\mathbb{R}} h d(Q_t^\lambda)_{\bar{\xi}\big(T^{\lambda,\lambda'}_s(B^\lambda)\big)}-\int_{\mathbb{R}} h d(Q_t^\lambda)_{\xi}\Big|\leq|\xi|_{L^\infty}E^Q\big[\big|\mathcal{E}_t^{\lambda'}-\mathcal{E}_t^{\lambda}\big|\big]\leq C|\lambda'-\lambda|.
\end{split}
\end{equation}
Furthermore, with the help of the Kantorovich-Rubinstein Theorem we conclude that for all $\lambda'\in\Lambda$ with $|\lambda'-\lambda|\leq \delta_1$,
$$ W_1\big((Q_t^\lambda)_{\bar{\xi}\big(T^{\lambda,\lambda'}_s(B^\lambda)\big)},(Q_t^\lambda)_\xi\big)\leq C|\lambda'-\lambda|. $$
Then, as with $\xi$ also $\bar{\xi}$ is in $\mathcal{S}$, Girsanov transformation arguments and H\"{o}lder inequality yield
\begin{equation*}
\begin{split}
|R_{\lambda,\lambda'}|&=\Big|E^{Q}\Big[\int_0^t\Big\{(\partial_\mu f)\Big((Q_t^\lambda)_{\bar{\xi}\big(T^{\lambda,\lambda'}_s(B^\lambda)\big)},\bar{\xi}\Big)-(\partial_\mu f)\big((Q_t^\lambda)_{\xi},\bar{\xi}\big)\Big\}D_s[\bar{\xi}](\bar{\gamma}_s^{\lambda'}-\bar{\gamma}_s^\lambda)\bar{\mathcal{E}}_s^{\lambda,\lambda'}ds\Big]\Big|\\
&\leq C E^Q\Big[\int_0^t W_1\big((Q_t^\lambda)_{\bar{\xi}\big(T^{\lambda,\lambda'}_s(B^\lambda)\big)},(Q_t^\lambda)_\xi\big)\big|\bar{\gamma}_s^{\lambda'}-\bar{\gamma}_s^\lambda\big|\bar{\mathcal{E}}_s^{\lambda,\lambda'}ds\Big]\\
&\leq C|\lambda'-\lambda|E^Q\Big[\int_0^t\big|\bar{\gamma}_s^{\lambda'}-\bar{\gamma}_s^\lambda\big|\bar{\mathcal{E}}_s^{\lambda,\lambda'}ds\Big]\leq C|\lambda'-\lambda|E^{Q_t^\lambda}\Big[\int_0^t\big| \gamma_s^{\lambda'}-\gamma_s^\lambda\big|\mathcal{E}_s^{\lambda,\lambda'}ds\Big]\\
&= C|\lambda'-\lambda|E^{Q}\Big[\int_0^t\big|\gamma_s^{\lambda'}-\gamma_s^\lambda\big|\mathcal{E}_s^{\lambda,\lambda'}\mathcal{E}_t^{\lambda}ds\Big]\leq C|\lambda'-\lambda|\Big(E^{Q}\Big[\int_0^t\big|\gamma_s^{\lambda'}-\gamma_s^\lambda\big|^2 ds\Big]\Big)^{\frac{1}{2}}.
\end{split}
\end{equation*}
But as $\lambda'\rightarrow\gamma^{\lambda'}\in L^2(dsdQ)$ is differentiable at $\lambda$,
\begin{equation}\label{equ***}
|R_{\lambda,\lambda'}|\leq C|\lambda'-\lambda|\cdot|\gamma^{\lambda'}-\gamma^\lambda|_{L^2(dsdQ)}\leq C|\lambda'-\lambda|^2=o(|\lambda'-\lambda|),\ \lambda'\rightarrow\lambda.
\end{equation}
Consequently,
\begin{equation*}
\begin{split}
&f\big((\mathcal{E}_t^{\lambda'}Q)_{\xi}\big)-f\big((\mathcal{E}_t^{\lambda}Q)_{\xi}\big)=I_{\lambda,\lambda'}+R_{\lambda,\lambda'}\\
=&(\lambda'-\lambda)E^{Q}\Big[\Big\{\int_0^\xi(\partial_\mu f)\big((Q_t^\lambda)_{\xi},y\big)dy\Big\}\partial_\lambda\mathcal{E}_t^{\lambda}\Big]+o(|\lambda'-\lambda|),\, \mbox{ for } \, \lambda'\in\Lambda,\, |\lambda'-\lambda|\rightarrow 0.
\end{split}
\end{equation*}
This proves that the mapping $\Lambda\ni\lambda\rightarrow f\big((Q_t^\lambda)_\xi\big)$ is differentiable, and
\begin{equation*}
\begin{split}
\partial_\lambda f\big((Q_t^\lambda)_\xi\big)&=E^Q\Big[\Big(\int_0^\xi(\partial_\mu f)\big((Q_t^\lambda)_\xi,y\big)dy\Big)\partial_\lambda[\mathcal{E}_t^\lambda]\Big]\\
&=E^{Q_t^\lambda}\Big[\Big(\int_0^\xi(\partial_\mu f)\big((Q_t^\lambda)_\xi,y\big)dy\Big)\partial_\lambda[\ln \mathcal{E}_t^\lambda]\Big],\ \lambda\in\Lambda.
\end{split}
\end{equation*}
The proof of Proposition \ref{lem1} will be completed by the following statement which we have used.
\end{proof}
\begin{lemma}\label{lem3.6}
Under Assumption 3 the mapping $\Lambda\ni\lambda\rightarrow\mathcal{E}_t^\lambda$ is continuously $L^1(Q)$-differentiable.
\end{lemma}
\begin{proof} From the definition of $\mathcal{E}^{\lambda'}_t$ as Girsanov density we have
$$ \ln(\mathcal{E}_t^{\lambda'})=\int_0^t \gamma_s^{\lambda'}dB_s-\frac{1}{2}\int_0^t|\gamma_s^{\lambda'}|^2ds,\ \lambda'\in\Lambda. $$
From Assumption 3 we recall that $|\gamma_s^{\lambda'}|\leq C,\  s\in[0,1],\, \lambda'\in\Lambda,$ and $\lambda'\rightarrow\gamma^{\lambda'}$ is continuously $L^2(dsdQ)$-differentiable. Therefore,
$$ \frac{\ln(\mathcal{E}_t^{\lambda'})-\ln(\mathcal{E}_t^{\lambda})}{\lambda'-\lambda}\rightarrow\int_0^t \partial_\lambda\gamma_s^{\lambda}dB_s-\int_0^t\gamma_s^{\lambda}\cdot\partial_\lambda\gamma_s^{\lambda}ds, \mbox{ in } L^2(Q),\ \lambda'\rightarrow \lambda,$$
and, hence, also in probability $Q$. Thus, in probability $Q$,
$$ \lim_{\lambda'\rightarrow\lambda}\frac{\mathcal{E}_t^{\lambda'}-\mathcal{E}_t^{\lambda}}{\lambda'-\lambda}=\mathcal{E}_t^\lambda\cdot\lim_{\lambda'\rightarrow\lambda}\frac{\ln(\mathcal{E}_t^{\lambda'})-\ln(\mathcal{E}_t^{\lambda})}{\lambda'-\lambda}= \mathcal{E}_t^{\lambda}\Big\{\int_0^t\partial_\lambda\gamma_s^{\lambda}dB_s-\int_0^t\gamma_s^{\lambda}\cdot\partial_\lambda\gamma_s^{\lambda}ds\Big\}, $$
and in order to complete the proof, it suffices to show that $\Big\{\Big|\frac{\mathcal{E}_t^{\lambda'}-\mathcal{E}_t^{\lambda}}{\lambda'-\lambda}\Big|,\, \lambda'\in\Lambda,\, |\lambda'-\lambda|\leq \delta_1\Big\}$, for some $\delta_1>0$, is uniformly integrable. For this it is sufficient to check that $\Big\{\Big|\frac{\mathcal{E}_t^{\lambda'}-\mathcal{E}_t^{\lambda}}{\lambda'-\lambda}\Big|,\, \lambda'\in\Lambda,\, |\lambda'-\lambda|\leq \delta_1\Big\}$ is bounded in $L^p(Q)$, for some $1<p<2$. Indeed, choosing arbitrarily $1<p<p'<2$, we have
\begin{equation*}
\begin{split}
&\Big(E^Q\Big[\Big|\frac{\mathcal{E}_t^{\lambda'}-\mathcal{E}_t^{\lambda}}{\lambda'-\lambda}\Big|^p\Big]\Big)^{\frac{1}{p}}=\Big(E^{Q_t^\lambda}\Big[\Big|\frac{\mathcal{E}_t^{\lambda,\lambda'}-1}{\lambda'-\lambda}\Big|^p(\mathcal{E}_t^\lambda)^p\frac{1}{\mathcal{E}_t^\lambda}\Big]\Big)^{\frac{1}{p}}\\
&\leq  C_{p,p'}\Big(E^{Q_t^\lambda}\Big[\Big|\frac{\mathcal{E}_t^{\lambda,\lambda'}-1}{\lambda'-\lambda}\Big|^{p'}\Big]\Big)^{\frac{1}{p'}}.
\end{split}
\end{equation*}
Here we have used that $\gamma^{\lambda'}$ is bounded, uniformly w.r.t. $\lambda'\in\Lambda$. Moreover, as $\displaystyle \mathcal{E}_t^{\lambda,\lambda'}=1+\int_0^t(\gamma_s^{\lambda'}-\gamma_s^\lambda)\mathcal{E}_s^{\lambda,\lambda'}dB_s^{\lambda},$ standard estimates give that, for some $\delta_1>0$,
\begin{equation*}
\begin{split}
&\Big(E^Q\Big[\Big|\frac{\mathcal{E}_t^{\lambda'}-\mathcal{E}_t^{\lambda}}{\lambda'-\lambda}\Big|^p\Big]\Big)^{\frac{1}{p}}\leq C_{p,p'}\Big(E^{Q_t^\lambda}\Big[\Big|\frac{\mathcal{E}_t^{\lambda,\lambda'}-1}{\lambda'-\lambda}\Big|^{p'}\Big]\Big)^{\frac{1}{p'}}= \frac{C_{p,p'}}{|\lambda'-\lambda|}\Big(E^{Q_t^\lambda}\Big[\Big|\int_0^t(\gamma_s^{\lambda'}-\gamma_s^{\lambda})\mathcal{E}_s^{\lambda,\lambda'}dB_s^\lambda\Big|^{p'}\Big]\Big)^{\frac{1}{p'}}\\
&\leq \frac{C_{p,p'}}{|\lambda'-\lambda|}\Big(E^{Q_t^\lambda}\Big[\Big(\int_0^t\big(|\gamma_s^{\lambda'}-\gamma_s^{\lambda}|\mathcal{E}_s^{\lambda,\lambda'}\big)^2ds\Big)^{\frac{p'}{2}}\Big]\Big)^{\frac{1}{p'}}\\
&\leq \frac{C_{p,p'}}{|\lambda'-\lambda|}\Big(E^{Q}\Big[\Big(\int_0^t|\gamma_s^{\lambda'}-\gamma_s^{\lambda}|^2ds\Big)^{\frac{p'}{2}}\sup_{s\in[0,t]}(\mathcal{E}_s^{\lambda,\lambda'})^{p'}\mathcal{E}_t^\lambda\Big]\Big)^{\frac{1}{p'}}
\leq \frac{C_{p,p'}}{|\lambda'-\lambda|}\Big(E^{Q}\Big[\int_0^t|\gamma_s^{\lambda'}-\gamma_s^{\lambda}|^2ds\Big]\Big)^{\frac{1}{2}}\\
&=\frac{C_{p,p'}}{|\lambda'-\lambda|}|\gamma^{\lambda'}-\gamma^{\lambda}|_{L^2(dsdQ)}\leq C_{p,p'},\, \, \mbox{for all } \lambda'\in \Lambda \mbox{ with } |\lambda'-\lambda|\leq \delta_1.
\end{split}
\end{equation*}
The latter estimate follows from the $L^2(dsdQ)$-differentiability of $\lambda'\rightarrow\gamma^{\lambda'}$.
The boundness of $\Big\{\Big|\frac{\mathcal{E}_t^{\lambda'}-\mathcal{E}_t^{\lambda}}{\lambda'-\lambda}\Big|,\, \lambda'\in\Lambda,\, |\lambda'-\lambda|\leq \delta_1\Big\}$ in $L^p(Q)$ $(1<p<2)$ has been deduced and the proof ends.
\end{proof}

\smallskip

From Proposition \ref{lem1} we get an immediate consequence as follows.
\begin{corollary}\label{coro1}
Under the Assumptions 2 and 3, for $\lambda,\lambda'\in\Lambda$ such that $[\lambda,\lambda']=\big\{\lambda(s)=s\lambda'+(1-s)\lambda,\ s\in[0,1]\big\}\subset\Lambda$,
\begin{equation*}
\begin{split}
&f\big((\mathcal{E}_t^{\lambda'}Q)_{\xi}\big)-f\big((\mathcal{E}_t^{\lambda}Q)_{\xi}\big)\Big(=\int_0^1 \partial_s\big[f\big((\mathcal{E}_t^{\lambda(s)}Q)_\xi\big)\big]ds\Big)\\
&=(\lambda'-\lambda)\int_0^1 E^Q\Big[\Big(\int_0^\xi\partial_\mu f\big((Q_t^{\lambda(s)})_\xi,y\big)dy\Big)\partial_\lambda\mathcal{E}_t^{\lambda(s)}\Big]ds.
\end{split}
\end{equation*}
\end{corollary}
In addition to Proposition \ref{lem1} the proof of the main statement of this subsection, Theorem \ref{Theorem3.1}, also needs the following approximation result.
\begin{proposition}\label{lem2}
Let $\Lambda\ni\lambda\rightarrow L^\lambda\in \mathcal{L}^Q\cap L^2(\Omega,\mathcal{F},Q)$ be a continuously $L^2(Q)$-differentiable mapping. Then there exists a sequence of bounded smooth Wiener step processes $\gamma^{\lambda,n},\ n\geq1$,
\begin{equation}\label{eqlem}
\gamma_t^{\lambda,n}=\sum_{i=1}^{N_n}\varphi_i^{\lambda,n}\big(B(\Delta^n_1),...,B(\Delta^n_{i-1})\big)I_{\Delta^n_i}(t),\ t\in[0,1],
\end{equation}
with $0=t^n_0<t^n_1<...<t^n_{N_n}=1$, $\Delta^n_i=(t^n_{i-1},t^n_i]$, such that\\
\indent {\rm i)} $\varphi^{\cdot,n}_i:\Lambda\times\mathbb{R}^{i-1}\rightarrow\mathbb{R}$ is a bounded Borel function,\\
\indent {\rm ii)} $\varphi^{\lambda,n}_{i}:\mathbb{R}^{i-1}\rightarrow\mathbb{R}$ is of class $C^\infty$ and the derivatives of all order are bounded over $\Lambda\times\mathbb{R}^{i-1}$,\\
\indent \ \ \ \ \ $1\leq i\leq n$,\\
\indent {\rm iii)} $\Lambda\ni\lambda\rightarrow\gamma^{\lambda,n}\in L^2_{\mathbb{F}}([0,1]\times\Omega,dsdQ)$ is continuously $L^2(dsdQ)$-differentiable on $\Lambda$, such that,\\
\indent \ \ \ \ \ for $\displaystyle \mathcal{E}_t^{\lambda,n}:=\exp\Big\{\int_0^t\gamma^{\lambda,n}_sdB_s-\frac{1}{2}\int_0^t|\gamma^{\lambda,n}_s|^2ds\Big\},\ t\in[0,1],\ \lambda\in\Lambda$, we have
\begin{equation*}
\begin{split}
&{\rm a) }\ (\mathcal{E}_1^{\lambda,n},\partial_\lambda\mathcal{E}_1^{\lambda,n})\xrightarrow[n\rightarrow\infty]{L^2(Q)}(L^\lambda,\partial_\lambda L^\lambda),\, \mbox{ for all }\lambda\in\Lambda,\\
&{\rm b) } \mbox{ for all } \lambda,\lambda'\in\Lambda \mbox{ with } [\lambda,\lambda']\subset\Lambda,\ \lambda(s):=s\lambda'+(1-s)\lambda,\ s\in[0,1],\\
&\ \ \ \ \ \ \ \ \ \ \ (\mathcal{E}_1^{\lambda(\cdot),n},\partial_\lambda\mathcal{E}_1^{\lambda(\cdot),n})\xrightarrow[n\rightarrow\infty]{L^2(dsdQ)}(L^{\lambda(\cdot)},\partial_\lambda L^{\lambda(\cdot)}).
\end{split}
\end{equation*}
\end{proposition}
For a better readability of the work, its proof is postponed to the Appendix. Let us now apply this proposition, in order to give the proof of Theorem \ref{Theorem3.1}.

\begin{proof} (of Theorem \ref{Theorem3.1}). Let the mapping $\Lambda\ni\lambda\rightarrow L^\lambda\in \mathcal{L}^Q\cap L^2(\Omega,\mathcal{F},Q)$ be continuously $L^2(Q)$-differentiable,  and $\xi\in L^4(\Omega,\mathcal{F},Q)$.
Then, due to Proposition \ref{lem2}, there exists a sequence of smooth Wiener step processes $\gamma^{\lambda,n},\ n\geq1$, of the form \eqref{eqlem}, such that\\
i) $\Lambda\ni\lambda\rightarrow\gamma^{\lambda,n}\in L^2_{\mathbb{F}}([0,1]\times\Omega,dsdQ)$ is continuously $L^2(dsdQ)$-differentiable on $\Lambda$;\\
ii) For the Dol\'ean-Dade exponential $\displaystyle\mathcal{E}_1^{\lambda,n}:=\exp\Big\{\int_0^1\gamma^{\lambda,n}_sdB_s-\frac{1}{2}\int_0^1|\gamma^{\lambda,n}_s|^2ds\Big\},\ \lambda\in\Lambda,\, n\ge 1$, we have
\begin{equation*}
\begin{split}
&\mbox{a) } \mbox{ For all } \lambda\in\Lambda,\ (\mathcal{E}_1^{\lambda,n},\partial_\lambda\mathcal{E}_1^{\lambda,n})\xrightarrow[n\rightarrow\infty]{L^2(Q)}(L^\lambda,\partial_\lambda L^\lambda);\\
&\mbox{b) } \mbox{ For all } \lambda,\lambda'\in\Lambda \mbox{ with } [\lambda,\lambda']\subset\Lambda,\, \mbox{ and } \lambda(s)=s\lambda'+(1-s)\lambda,\ s\in[0,1],\\
&\ \ \ \ \ \ \ \ \ \ \ \big(\mathcal{E}_1^{\lambda(\cdot),n},\partial_\lambda\mathcal{E}_1^{\lambda(\cdot),n}\big)\xrightarrow[n\rightarrow\infty]{L^2(dsdQ)}\big(L^{\lambda(\cdot)},\partial_\lambda L^{\lambda(\cdot)}\big).
\end{split}
\end{equation*}
Moreover, by now standard arguments we have the existence of a sequence of smooth Wiener functionals $\{\xi^n\}_{n\geq1}$ such that $\xi^n\rightarrow\xi$, in $L^4(Q)$, $n\rightarrow\infty$. From Corollary \ref{coro1} we have
\begin{equation}\label{proof1}
\begin{split}
&f\big((\mathcal{E}^{\lambda',n}Q)_{\xi^n}\big)-f\big((\mathcal{E}^{\lambda,n}Q)_{\xi^n}\big)\\	
&=(\lambda'-\lambda)\int_0^1 E^Q\Big[\Big(\int_0^{\xi^n}\partial_\mu f\big((\mathcal{E}^{\lambda(s),n}Q)_{\xi^n},y\big)dy\Big)\partial_\lambda\mathcal{E}^{\lambda(s),n}\Big]ds,\ n\geq1.
\end{split}
\end{equation}
Let us show that taking the limit in this equality, as $n\rightarrow\infty$, yields the result stated in the theorem.\\
\textbf{1)} Let us begin our discussion with the term of the left-hand side of (\ref{proof1}). For this we note that, for all $h\in \mbox{Lip}_1(\mathbb{R})$ with $h(0)=0$,
\begin{equation*}
\begin{split}
&\Big|\int_{\mathbb{R}} h d(\mathcal{E}_1^{\lambda,n}Q)_{\xi^n}-\int_{\mathbb{R}} h d(L^\lambda Q)_{\xi}\Big|=\Big|E^Q\big[\mathcal{E}_1^{\lambda,n} h(\xi^n)\big]-E^Q\big[L^\lambda h(\xi)\big]\Big|\\
&\leq E^Q\big[\big|\xi^n\big|\cdot\big|\mathcal{E}_1^{\lambda,n}-L^\lambda\big|\big]+E^Q\big[L^\lambda\cdot\big|\xi^n-\xi\big|\big]\\
&\leq \big|\xi^n\big|_{L^2(Q)}\big|\mathcal{E}_1^{\lambda,n}-L^\lambda\big|_{L^2(Q)}+\big|L^\lambda\big|_{L^2(Q)}\big|\xi^n-\xi\big|_{L^2(Q)}=:a_n(\lambda).
\end{split}
\end{equation*}
Consequently, due to the Kanorovich-Rubinstein Theorem,
$$W_1\big((\mathcal{E}^{\lambda,n}Q)_{\xi^n}, (L^\lambda Q)_{\xi}\big)\leq a_n(\lambda),$$
and as, due to Assumption 2, $\big|f\big((\mathcal{E}^{\lambda,n}Q)_{\xi^n}\big)-f\big((L^\lambda Q)_{\xi}\big)\big|\leq CW_1\big((\mathcal{E}^{\lambda,n}Q)_{\xi^n},(L^\lambda Q)_{\xi}\big),$ it follows that, for all $\lambda\in\Lambda$, $\big|f\big((\mathcal{E}^{\lambda,n}Q)_{\xi^n}\big)-f\big((L^\lambda Q)_{\xi}\big)\big|\leq C a_n(\lambda)\rightarrow 0,\, \mbox{ as } n\rightarrow\infty.$ Consequently, for all $\lambda,\lambda'\in\Lambda$,
$$ f\big((\mathcal{E}^{\lambda',n}Q)_{\xi^n}\big)-f\big((\mathcal{E}^{\lambda,n}Q)_{\xi^n}\big)\xrightarrow[n\rightarrow\infty]{} f\big((L^{\lambda'} Q)_{\xi}\big)-f\big((L^\lambda Q)_{\xi}\big). $$
\textbf{2)} Let us now show the convergence for the right-hand side of (\ref{proof1}). For this we put $\displaystyle\vartheta_n(s):=\int_0^{\xi^n}\partial_\mu f\big((\mathcal{E}^{\lambda(s),n}Q)_{\xi^n},y\big)dy$ and $\displaystyle\vartheta(s):=\int_0^{\xi}\partial_\mu f\big((L^{\lambda(s)}Q)_{\xi},y\big)dy$, $s\in[0,1]$, and from Assumption 2 we obtain
\begin{equation*}
\begin{split}
\big|\vartheta_n(s)-\vartheta(s)\big|&\leq C|\xi-\xi^n|+C\int_0^{|\xi|}W_1\big((\mathcal{E}^{\lambda(s),n}Q)_{\xi^n},(L^{\lambda(s)}Q)_{\xi}\big)dy\\
&\leq C|\xi-\xi^n|+C|\xi|a_n\big(\lambda(s)\big),
\end{split}
\end{equation*}
and, hence,
$$ \int_0^1 E^Q\big[\big|\vartheta_n(s)-\vartheta(s)\big|^2\big]ds\leq CE^Q[|\xi-\xi^n|^2]+CE^Q[|\xi|^2]\int_0^1\Big(a_n\big(\lambda(s)\big)\Big)^2ds\xrightarrow[n\rightarrow\infty]{}0. $$
Consequently, for the right-hand side of (\ref{proof1}) we get
\begin{equation*}
\begin{split}
&\Big|\int_0^1E^Q\big[\vartheta_n(s)\partial_\lambda\mathcal{E}^{\lambda(s),n}\big]ds-\int_0^1E^Q\big[\vartheta(s)\partial_\lambda L^{\lambda(s)}\big]ds\Big|\\
\leq&\Big(\int_0^1 E^Q\big[\big|\vartheta_n(s)-\vartheta(s)\big|^2\big]ds\Big)^{\frac{1}{2}}\Big(\int_0^1 E^Q\big[\big|\partial_\lambda\mathcal{E}^{\lambda(s),n}\big|^2\big]ds\Big)^{\frac{1}{2}}\\
&+\Big(\int_0^1 E^Q\big[\big|\partial_\lambda\mathcal{E}^{\lambda(s),n}-\partial_\lambda L^{\lambda(s)}\big|^2\big]ds\Big)^{\frac{1}{2}}\Big(\int_0^1 E^Q\big[\big|\vartheta(s) \big|^2\big]ds\Big)^{\frac{1}{2}}\rightarrow0,\ n\rightarrow\infty.
\end{split}
\end{equation*}
Combined with our result from 1) this yields
\begin{equation}\label{equ3.30}
\begin{split}
f\big((L^{\lambda'}Q)_{\xi}\big)-f\big((L^{\lambda}Q)_{\xi}\big)=(\lambda'-\lambda)\int_0^1 E^Q\Big[\Big(\int_0^{\xi}\partial_\mu f\big((L^{\lambda(s)}Q)_{\xi},y\big)dy\Big)\partial_\lambda L^{\lambda(s)}\Big]ds.
\end{split}
\end{equation}
Finally, observe that, as
\begin{equation*}
\begin{split}
& E^Q\Big[\Big|\int_0^{\xi}\Big(\partial_\mu f\big((L^{\lambda(s)}Q)_{\xi},y\big)-\partial_\mu f\big((L^{\lambda}Q)_{\xi},y\big)\Big)dy \Big|\Big]\\
&\leq CE^Q\big[|\xi|\cdot W_1\big((L^{\lambda(s)}Q)_\xi, (L^{\lambda}Q)_\xi\big)\big]\\
&\leq CE^Q\big[|\xi|\big]E^Q\big[|\xi||L^{\lambda(s)}-L^\lambda|\big],
\end{split}
\end{equation*}
we deduce from (\ref{equ3.30}) with the help of the $L^2(Q)$-continuity of the mapping $\lambda\rightarrow (L^\lambda,\partial_\lambda L^\lambda)$ that
\begin{equation}
\begin{split}
&f\big((L^{\lambda'}Q)_{\xi}\big)-f\big((L^{\lambda}Q)_{\xi}\big)\\
&=(\lambda'-\lambda)\int_0^1 E^Q\Big[\Big(\int_0^{\xi}\partial_\mu f\big((L^{\lambda(s)}Q)_{\xi},y\big)dy\Big)\partial_\lambda L^{\lambda(s)}\Big]ds\\
&=(\lambda'-\lambda) E^Q\Big[\Big(\int_0^{\xi}\partial_\mu f\big((L^{\lambda}Q)_{\xi},y\big)dy\Big)\partial_\lambda L^{\lambda}\Big]+\widetilde{R}_{\lambda,\lambda'}(\lambda'-\lambda),\\
\end{split}
\end{equation}
where
$$ \widetilde{R}_{\lambda,\lambda'}=\int_0^1 E^Q\Big[\Big(\int_0^{\xi}\partial_\mu f\big((L^{\lambda(s)}Q)_{\xi},y\big)dy\Big)\partial_\lambda L^{\lambda(s)}\Big]ds-E^Q\Big[\Big(\int_0^{\xi}\partial_\mu f\big((L^{\lambda}Q)_{\xi},y\big)dy\Big)\partial_\lambda L^{\lambda}\Big]. $$
Standard arguments show that $\widetilde{R}_{\lambda,\lambda'}\rightarrow0$, as $\lambda'\rightarrow\lambda$.
This proves that $\lambda\rightarrow f\big((L^\lambda Q)_\xi\big)$ is continuously differentiable, and
\begin{equation}
\partial_\lambda f\big((L^\lambda Q)_\xi\big)=E^Q\Big[\Big(\int_0^{\xi}\partial_\mu f\big((L^{\lambda}Q)_{\xi},y\big)dy\Big)\partial_\lambda L^{\lambda}\Big].
\end{equation}
\end{proof}

After having studied the case of 1-dimensional random variables $\xi$, let us discuss now briefly the multi-dimensional case. Our objective is to show that under suitable assumptions we have also for multi-dimensional $\xi$ the relation (\ref{3.19bis}) stated in Theorem \ref{Theorem3.2} for 1-dimensional $\xi$.

However, while $\xi\in L^4(\Omega,\mathcal{F},Q;\mathbb{R}^d)$ can be now $d$-dimensional, $d\ge 1$, we continue to suppose for simplicity that the coordinate process $B$ in the classical Wiener space $(\Omega,\mathcal{F},Q)$ is 1-dimensional. Our results can be extended in an obvious way to the case of a multi-dimensional coordinate process $B$.
Let us use the notations which have been introduced in this subsection. But, since we will have to consider perturbations of $\xi$, we emphasise the dependence of $F_{Q_L}(L')=f\big((L'Q_L)_\xi\big)$ on $\xi$ by writing
$$F_{Q_L}^\xi(L')=f\big((L'Q_L)_\xi\big),\, L\in\mathcal{L}^Q,\ L'\in\mathcal{L}^{Q_L} \mbox{ and } \xi\in L^4(\Omega,\mathcal{F},Q_L;\mathbb{R}^d).$$

Working first with Assumption 2 and assumptions of the same type as those in Assumption 3, the fact that the Brownian motion $B$ is 1-dimensional, allows to use also for our  $d$-dimensional $\xi$ ($d\ge 1$) all the computations made before in the proof of Proposition \ref{lem1} until equation (\ref{equ**}) including. Assume that $\Lambda$ is a bounded subinterval of $\mathbb{R}$ and let us make the following specification for our \textbf{Assumption 3'}: We let

\noindent \textbf{i)} $\xi$ be a $d$-dimensional smooth Wiener functional:
$$\xi=\varphi\big(B(\Delta_1),...,B(\Delta_n)\big),\ \varphi\in C_b^\infty(\mathbb{R}^n;\mathbb{R}^d).$$
For the notations we refer to Assumption 3; also here we put for simplicity $T=1$.

\noindent \textbf{ii)} Let $L\in\mathcal{L}^Q$ be a smooth Wiener functional (of the above form, for another $\varphi\in C_b^\infty(\mathbb{R}^n)$) such that, for some $0<c<C<+\infty$, it holds $c\le L\le C$.

Putting
$$\gamma_s:=\frac{E^Q[D_sL\,|\,\mathcal{F}_s]}{E^Q[L\,|\,\mathcal{F}_s]},\, s\in[0,1],$$
we get a process which is smooth in the sense of the Malliavin derivative $D=(D_s)_{s\in[0,1]}$, and this process as well as its Malliavin derivatives of any order are bounded over $[0,1]\times\Omega$. Moreover, due to the Clark-Ocone formula, $L_s:=E^Q[L\,|\,\mathcal{F}_s]$ has the form
$$L_s=\exp\Big\{\int_0^s\gamma_rdB_r-\frac12\int_0^s|\gamma_r|^2dr\Big\},\, s\in[0,1].$$

\noindent \textbf{iii)} Let $\theta$ be a smooth Wiener process of the form
$$\begin{array}{lll}
\theta_s=\displaystyle\sum_{i=1}^n\varphi_i\big(B(\Delta_1),...,B(\Delta_{i-1})\big)I_{\Delta_i}(s),\, s\in[0,1],\,
\varphi_i\in C_b^\infty(\mathbb{R}^{i-1}),\, 1\le i\le n.
\end{array}$$

\noindent \textbf{iv)} Let $\lambda\in\Lambda$ be arbitrarily fixed. For any $\lambda'\in\Lambda$, we put $\displaystyle \gamma^{\lambda'}=\gamma+(\lambda'-\lambda)\theta$ (We remark that, in particular, $\gamma^\lambda=\gamma$), and we also observe that $\gamma^{\lambda'},\lambda'\in\Lambda$, is a bounded smooth (in Malliavin's sense) random field over $\Lambda\times[0,1]\times \Omega$ with bounded derivatives of all order w.r.t. $(\lambda,\omega)$.
We also recall the notation
$$ \mathcal{E}_t^{\lambda'}:=\exp\Big\{\int_0^t\gamma^{\lambda'}_sdB_s-\frac12\int_0^t|\gamma_s^{\lambda'}|^2ds\Big\},\, \lambda'\in\Lambda.$$
With it we have, in particular, $\mathcal{E}_t^{\lambda}=L_t,\, \mathcal{E}_1^{\lambda}=L_1=L.$
\smallskip

\noindent \textbf{v)} For all $s\in[0,1]$, we introduce $\displaystyle\mathcal{E}^{\lambda,\lambda'}_t(s):=(1-s)\mathcal{E}^\lambda_t+s\mathcal{E}^{\lambda'}_t\big(=(1-s)L_t+s\mathcal{E}^{\lambda'}_t\big)$.

\noindent Let us point out that because of iv), $\displaystyle \mathcal{E}_s^{\lambda,\lambda'}:=\frac{\mathcal{E}_s^{\lambda'}}{\mathcal{E}_s^{\lambda}}\rightarrow 1,$ as $\lambda'\rightarrow\lambda.$

The above setting allows to work as under Assumption 3.  Thus, revisiting the proof of Proposition \ref{lem1}, we see that, for getting (\ref{equ3.23}), (\ref{equ**}) and (\ref{equ***}), the dimension 1 of $\xi$ was not used; the computations are exactly the same for a higher dimension of $\xi$. So we have, for $t\in[0,1]$,
\begin{equation}\label{d111a}
f\big((\mathcal{E}_t^{\lambda'}Q)_\xi\big)-f\big((\mathcal{E}_t^{\lambda}Q)_\xi\big)=I_{\lambda,\lambda'}+R_{\lambda,\lambda'}, \mbox{ where } R_{\lambda,\lambda'}=o(|\lambda'-\lambda|), \mbox{ as } \lambda'\rightarrow\lambda,
\end{equation}
and
\begin{equation}\label{daaa3}
\begin{split}
I_{\lambda,\lambda'}&=E^{Q}\Big[\int_0^t(\partial_\mu f)\Big((Q_t^\lambda)_{\xi},\bar{\xi}\Big)D_s[\bar{\xi}](\bar{\gamma}_s^{\lambda'}-\bar{\gamma}_s^\lambda)\bar{\mathcal{E}}_s^{\lambda,\lambda'}ds\Big]\\
&=E^{Q_t^\lambda}\Big[\int_0^t(\partial_\mu f)\Big((Q_t^\lambda)_{\xi},\bar{\xi}(B^\lambda)\Big)D_s[\bar{\xi}](B^\lambda)(\bar{\gamma}_s^{\lambda'}-\bar{\gamma}_s^\lambda)(B^\lambda)\bar{\mathcal{E}}_s^{\lambda,\lambda'}(B^\lambda)ds\Big]\\
&=E^{Q_t^\lambda}\Big[\int_0^t(\partial_\mu f)\big((Q_t^\lambda)_{\xi},\xi\big)D_s^\lambda[\xi](\gamma_s^{\lambda'}-\gamma_s^\lambda)\mathcal{E}_s^{\lambda,\lambda'}ds\Big]\\
&=(\lambda'-\lambda)E^{Q_t^\lambda}\Big[\int_0^t(\partial_\mu f)\big((Q_t^\lambda)_{\xi},\xi\big)D_s^\lambda[\xi]\theta_sds\Big]
+o(|\lambda'-\lambda|),\, \mbox{as }\lambda'\rightarrow \lambda.
\end{split}
\end{equation}
(Recall that $D^\lambda=(D^\lambda_s)_{s\in[0,1]}$ denotes the Malliavin derivative w.r.t. the Brownian motion $B^\lambda$). In order to get the latter equality, recall that, due to Assumption iv), $\gamma^{\lambda'}=\gamma^{\lambda}+(\lambda'-\lambda)\theta$ and
\begin{equation}\label{340-0}
\begin{split}
\mathcal{E}_s^{\lambda,\lambda'}=1+\int_0^s(\gamma_r^{\lambda'}-\gamma^{\lambda}_r)\mathcal{E}_r^{\lambda,\lambda'}dB_r^\lambda=1+(\lambda'-\lambda)\int_0^s\theta_r\mathcal{E}_r^{\lambda,\lambda'}dB_r^\lambda,\, s\in[0,1].
\end{split}
\end{equation}
Taking into account that $\mathbb{F}$-adaptedness of the process $\gamma^{\lambda'}-\gamma^{\lambda}=(\lambda'-\lambda)\theta$, we get from \eqref{daaa3} that
\begin{equation}\label{340-1}
\begin{split}
I_{\lambda,\lambda'}&=E^{Q_t^\lambda}\Big[\int_0^t E^{Q_t^\lambda}\big[(\partial_\mu f)\big((Q_t^\lambda)_{\xi},\xi\big)D_s^\lambda[\xi]\,\big|\,\mathcal{F}_s\big](\gamma_s^{\lambda'}-\gamma_s^{\lambda})\mathcal{E}_s^{\lambda,\lambda'}ds\Big]\\
&=E^{Q_t^\lambda}\Big[\int_0^t E^{Q_t^\lambda}\big[(\partial_\mu f)\big((Q_t^\lambda)_{\xi},\xi\big)D_s^\lambda[\xi]\,\big|\,\mathcal{F}_s\big]dB_s^\lambda\int_0^t(\gamma_s^{\lambda'}-\gamma_s^{\lambda})\mathcal{E}_s^{\lambda,\lambda'}dB_s^\lambda\Big]\\
&=E^{Q_t^\lambda}\Big[\int_0^t E^{Q_t^\lambda}\big[(\partial_\mu f)\big((Q_t^\lambda)_{\xi},\xi\big)D_s^\lambda[\xi]\,\big|\,\mathcal{F}_s\big]dB_s^\lambda(\mathcal{E}_t^{\lambda,\lambda'}-1)\Big].
\end{split}
\end{equation}
Hence,
$$ I_{\lambda,\lambda'}=E^{Q}\Big[\int_0^t E^{Q_t^\lambda}\big[(\partial_\mu f)\big((Q_t^\lambda)_{\xi},\xi\big)D_s^\lambda[\xi]\,\big|\,\mathcal{F}_s\big]dB_s^\lambda(\mathcal{E}_t^{\lambda'}-\mathcal{E}_t^{\lambda})\Big]. $$
On the other hand, as $\mathcal{E}^{\lambda'}-\mathcal{E}^{\lambda}$ is an $(\mathbb{F},Q)$-martingale, from Doob's martingale inequality, Burkholder-Davis-Gundy inequality and standard estimates we see that, for some generic constant $C>0$,
\begin{equation*}
\begin{split}
|\lambda'-\lambda|&\leq C|\lambda'-\lambda|E^{Q_t^\lambda}\Big[\Big(\int_0^t|\theta_s|^2(\mathcal{E}_s^{\lambda,\lambda'})^2ds\Big)^{\frac{1}{2}}\Big]\\
&=C E^{Q_t^\lambda}\Big[\Big(\int_0^t\big|(\gamma_s^{\lambda'}-\gamma_s^{\lambda})\mathcal{E}_s^{\lambda,\lambda'}\big|^2ds\Big)^{\frac{1}{2}}\Big]\\
&\leq C E^{Q_t^\lambda}\big[\sup_{s\in[0,t]}|\mathcal{E}_s^{\lambda,\lambda'}-1|\big]=C E^{Q}\Big[\sup_{s\in[0,t]}\Big(\frac{\mathcal{E}_t^{\lambda}}{\mathcal{E}_s^{\lambda}}|\mathcal{E}_s^{\lambda'}-\mathcal{E}_s^{\lambda}|\Big)\Big]\\
&\leq C \Big(E^{Q}\big[\sup_{s\in[0,t]}|\mathcal{E}_s^{\lambda'}-\mathcal{E}_s^{\lambda}|^2\big]\Big)^{\frac{1}{2}}\leq C \Big(E^{Q}\big[|\mathcal{E}_t^{\lambda'}-\mathcal{E}_t^{\lambda}|^2\big]\Big)^{\frac{1}{2}}.
\end{split}
\end{equation*}
Thus,\ $\displaystyle \frac{|\lambda'-\lambda|}{\Big(E^{Q}\big[|\mathcal{E}_t^{\lambda'}-\mathcal{E}_t^{\lambda}|^2\big]\Big)^{\frac{1}{2}}}\leq C, $\
and for $R_{\lambda,\lambda'}$ in \eqref{d111a} we have
$$ \frac{|R_{\lambda,\lambda'}|}{\Big(E^{Q}\big[|\mathcal{E}_t^{\lambda'}-\mathcal{E}_t^{\lambda}|^2\big]\Big)^{\frac{1}{2}}}\leq C\frac{|R_{\lambda,\lambda'}|}{|\lambda'-\lambda|}\rightarrow0,\ \mbox{ as } \lambda'\rightarrow\lambda. $$
Consequently, from \eqref{340-1} and \eqref{d111a},
\begin{equation}\label{340-2}
\begin{split}
&f\big((\mathcal{E}_t^{\lambda'}Q)_\xi\big)-f\big((\mathcal{E}_t^{\lambda}Q)_\xi\big)\\
&=E^{Q}\Big[\int_0^t E^{Q_t^\lambda}\big[(\partial_\mu f)\big((Q_t^\lambda)_{\xi},\xi\big)D_s^\lambda[\xi]\,\big|\,\mathcal{F}_s\big]dB_s^\lambda(\mathcal{E}_t^{\lambda'}-\mathcal{E}_t^{\lambda})\Big]+o(|\mathcal{E}_t^{\lambda'}-\mathcal{E}_t^{\lambda}|_{L^2(Q)}).
\end{split}
\end{equation}
Recall that $\mathcal{E}_t^{\lambda}=L$ and $F_Q^\xi(L'):=f\big((L'Q)_\xi\big),\, L'\in\mathcal{L}^Q$. The equality \eqref{340-2} proves the $L^2(Q)$-differentiability of $L'\rightarrow F_Q^\xi(L')$ at $L$ along special curves $\lambda'\rightarrow\mathcal{E}_t^{\lambda'}$, and the derivative is
\begin{equation}\label{340-3}
DF_Q^\xi(L)=DF_Q^\xi(\mathcal{E}_t^{\lambda})=\int_0^t E^{Q_t^\lambda}\big[(\partial_\mu f)\big((Q_t^\lambda)_{\xi},\xi\big)D_s^\lambda[\xi]\,\big|\,\mathcal{F}_s\big]dB_s^\lambda.
\end{equation}
Combining approximation techniques from Proposition \ref{lem2} with the argument of the proof of Theorem \ref{P3} and Remark \ref{Re3.2}, one can show that there is a Borel measurable function $g:\mathbb{R}^d\rightarrow\mathbb{R}$ such that $DF_Q^\xi(L)=g(\xi),\ Q$-a.s., and $g$ depends on $(Q, L,\xi)$ only through $\big(Q_\xi,(LQ)_\xi\big)$. This allows to write
\begin{equation}\label{340-4}
\partial_L F\big(Q_\xi,(LQ)_\xi,x\big):=g(x),\ x\in\mathbb{R}^d.
\end{equation}
Note that for $d=1$, using the Clark-Ocone formula, we have for the right-hand side of \eqref{340-3} just
\begin{equation*}
\begin{split}
&\int_0^t E^{Q_t^\lambda}\big[(\partial_\mu f)\big((Q_t^\lambda)_{\xi},\xi\big)D_s^\lambda[\xi]\,\big|\,\mathcal{F}_s\big]dB_s^\lambda=\int_0^t E^{Q_t^\lambda}\Big[D_s^\lambda\big[\int_0^\xi(\partial_\mu f)\big((Q_t^\lambda)_{\xi},y\big)dy\big]\,\Big|\,\mathcal{F}_s\Big]dB_s^\lambda\\
=&\int_0^\xi (\partial_\mu f)\big((Q_t^\lambda)_{\xi},y\big)dy-E^{Q_t^\lambda}\Big[\int_0^\xi(\partial_\mu f)\big((Q_t^\lambda)_{\xi},y\big)dy\Big],\ Q\mbox{-a.s.},
\end{split}
\end{equation*}
i.e., we rediscover here our result for $d=1$. For $d\geq2$, let us characterize the relation between $DF_Q^\xi$ and $\partial_\mu f$ in a manner which is more explicit than \eqref{340-3}.

In order to shorten our argument, let us suppose now the $L^2(Q)$-differentiability of $L'\rightarrow F_{Q}^\xi(L'):=f\big((L'Q)_\xi\big)$ at all $L'\in\mathcal{L}^{Q}\cap L^2(\Omega,\mathcal{F},Q)$, and identifying its derivative $(DF_Q^\xi)(L')\in L(L^2_0(\Omega,\mathcal{F},Q);\mathbb{R})$ with $DF_Q^\xi(L')\in L^2_0(\Omega,\mathcal{F},Q)$ such that $DF_Q^\xi(L')(\eta)=E^Q[DF_Q^\xi(L')\eta]$, for all $\eta\in L^2_0(\Omega,\mathcal{F},Q)$, we obtain
\begin{equation}\label{d111}
\begin{split}
&f\big((\mathcal{E}_t^{\lambda'}Q)_\xi\big)-f\big((\mathcal{E}_t^{\lambda}Q)_\xi\big)=\displaystyle F^\xi_{Q}(\mathcal{E}_t^{\lambda'})- F^\xi_{Q}(\mathcal{E}_t^{\lambda})=\int_0^1\partial_s [F^\xi_{Q}\big(\mathcal{E}_t^{\lambda,\lambda'}(s)\big)]ds\\
=&E^{Q}\big[DF^\xi_{Q}(\mathcal{E}_t^{\lambda}) (\mathcal{E}^{\lambda'}_t-\mathcal{E}^{\lambda}_t)\big]+
\int_0^1E^{Q}\Big[\Big(DF^\xi_{Q}\big(\mathcal{E}_t^{\lambda,\lambda'}(s)\big)-DF^\xi_{Q}(\mathcal{E}_t^{\lambda}) \Big)(\mathcal{E}^{\lambda'}_t- \mathcal{E}^{\lambda}_t)\Big]ds\\
=&I^\xi(\lambda,\lambda')+R^\xi(\lambda,\lambda'),
\end{split}
\end{equation}
where  $\displaystyle I^\xi(\lambda,\lambda'):=E^{Q}\big[DF^\xi_{Q}(\mathcal{E}_t^{\lambda}) (\mathcal{E}^{\lambda'}_t-\mathcal{E}^{\lambda}_t)\big],$\   and $$\displaystyle R^\xi(\lambda,\lambda'):=\int_0^1E^{Q}\Big[\Big(DF^\xi_{Q}\big(\mathcal{E}_t^{\lambda,\lambda'}(s)\big)-DF^\xi_{Q}(\mathcal{E}_t^{\lambda}) \Big)(\mathcal{E}^{\lambda'}_t- \mathcal{E}^{\lambda}_t)\Big]ds.$$

Before making the computation for $I^\xi(\lambda,\lambda')$, we remark that $0<c\le L_t\big(=\mathcal{E}^\lambda_t\big)\le C\in\mathbb{R}_+$ implies that with the $L^2(Q)$-differentiability of $F_Q^\xi$ at $L_t$ we also have the $L^2(Q_t^\lambda)$-differentiability of $F_{Q_t^\lambda}^\xi:\mathcal{L}^{Q^\lambda_t}\cap L^2(\Omega,\mathcal{F},Q^\lambda_t)\rightarrow\mathbb{R}$ at $1$
(Recall that $Q_t^\lambda=\mathcal{E}_t^\lambda Q(=L_tQ)$), and
\begin{equation}\label{tri}
\begin{split}
DF^\xi_Q(\mathcal{E}_t^\lambda)&=\displaystyle DF^\xi_{Q^\lambda_t}(1)-E^Q\big[DF^\xi_{Q^\lambda_t}(1)\big],\\
DF^\xi_{Q^\lambda_t}(1)&=\displaystyle DF^\xi_Q(\mathcal{E}_t^\lambda)-E^{Q^\lambda_t}\big[DF^\xi_Q(\mathcal{E}_t^\lambda)\big].
\end{split}
\end{equation}
Moreover, we recall from Theorem \ref{P3} and its proof (see also Lemma \ref{Lemjing}) that, for $L\in\mathcal{L}^Q\cap L^2(\Omega,\mathcal{F},Q)$, the differentiability of $F_{Q_t^\lambda}^\xi$ at $1$ implies that there is a Borel function $g:\mathbb{R}^d\rightarrow \mathbb{R}$ such that $DF_{Q_t^\lambda}^\xi(1)=g(\xi)$, $Q_L$-a.s. and $g$ depends on $(Q,\mathcal{E}_t^\lambda,\xi)$ only through the law $(Q_t^\lambda)_\xi$, which is expressed by the notation $\partial_1 F\big((Q_t^\lambda)_\xi,x\big):=g(x),\, x\in \mathbb{R}^d$. So we have $\displaystyle DF_{Q_t^\lambda}^\xi(1)=\partial_1 F \big((Q_t^\lambda)_\xi,\xi\big),\, Q_t^\lambda(\sim Q)\mbox{-a.s.}$ Note that, for $L_t=\mathcal{E}_t^\lambda$, $\partial_L F$ introduced in \eqref{340-4} and $\partial_1 F$ thanks to \eqref{tri} are related by
$$ \partial_{L_t}F\big(Q_\xi,(L_t Q)_\xi,x\big)=\partial_1 F\big((Q_{L_t})_\xi,x\big)-E^Q\big[\partial_1 F\big((Q_{L_t})_\xi,\xi\big)\big],\ x\in\mathbb{R}^d. $$
Consequently, from \eqref{tri} and the definition $\mathcal{E}_t^{\lambda,\lambda'}=\mathcal{E}_t^{\lambda'}/\mathcal{E}_t^{\lambda}$,
\begin{equation*}
\begin{split}
I^\xi(\lambda,\lambda')&=E^{Q}[DF^\xi_{Q}(\mathcal{E}_t^{\lambda}) (\mathcal{E}^{\lambda'}_t- \mathcal{E}^{\lambda}_t)]=E^{Q}[DF^\xi_{Q^\lambda_t}(1)(\mathcal{E}^{\lambda'}_t- \mathcal{E}^{\lambda}_t)]\\
&=E^{Q^\lambda_t}[DF^\xi_{Q^\lambda_t}(1) (\mathcal{E}^{\lambda,\lambda'}_t- 1)]= E^{Q^\lambda_t}\big[\partial_1 F\big((Q_t^\lambda)_\xi,\xi\big)(\mathcal{E}^{\lambda, \lambda'}_t- 1)\big],
\end{split}
\end{equation*}
and as $\displaystyle\mathcal{E}^{\lambda,\lambda'}_t=1+\int_0^t(\gamma_s^{\lambda'}-\gamma_s^\lambda)\mathcal{E}^{\lambda,\lambda'}_sdB^\lambda_s$ (See  \eqref{equ3.19}), the above relation yields
\begin{equation}\label{d111b}\begin{split}
I^\xi(\lambda,\lambda')
&=\displaystyle E^{Q^\lambda_t}\Big[(\partial_1 F)\big((Q_t^\lambda)_\xi,\xi\big)\int_0^t (\gamma_s^{\lambda'}-\gamma_s^\lambda)\mathcal{E}^{\lambda, \lambda'}_sdB^\lambda_s\Big]\\
&=\displaystyle (\lambda'-\lambda)E^{Q^\lambda_t}\Big[(\partial_1 F)\big((Q_t^\lambda)_\xi,\xi\big)\int_0^t \theta_s\mathcal{E}^{\lambda,\lambda'}_sdB^\lambda_s\Big].\\
\end{split}\end{equation}
Hence, supposing that $\partial_1 F\big((Q_t^\lambda)_\xi,\cdot\big):\mathbb{R}^d\rightarrow\mathbb{R}$ is continuously differentiable and the derivative is bounded, we get from the integration by parts formula for the Malliavin derivative w.r.t. the $Q_t^\lambda$-Brownian motion $B^\lambda$
\begin{equation}\label{d111c}\begin{split}
I^\xi(\lambda,\lambda')&=\displaystyle (\lambda'-\lambda)E^{Q^\lambda_t}\Big[(\partial_1 F)\big((Q_t^\lambda)_\xi,\xi\big)\int_0^t \theta_s\mathcal{E}^{\lambda,\lambda'}_sdB^\lambda_s\Big]\\
&=\displaystyle (\lambda'-\lambda)E^{Q^\lambda_t}\Big[\partial_x(\partial_1 F)\big((Q_t^\lambda)_\xi,\xi\big)\int_0^t D^\lambda_s[\xi]\theta_s\mathcal{E}^{\lambda,\lambda'}_sds\Big]\\
&=\displaystyle (\lambda'-\lambda)E^{Q^\lambda_t}\Big[\partial_x(\partial_1 F)\big((Q_t^\lambda)_\xi,\xi\big)\int_0^t D^\lambda_s[\xi]\theta_sds\Big]+r(\lambda,\lambda'),
\end{split}\end{equation}
with $\displaystyle r(\lambda,\lambda')=(\lambda'-\lambda)^2 E^{Q^\lambda_t}\Big[\partial_x(\partial_1 F)\big((Q_t^\lambda)_\xi,\xi\big)\int_0^t\Big(D^\lambda_s[\xi]\theta_s\int_0^s\theta_r\mathcal{E}_r^{\lambda,\lambda'}dB_r^\lambda\Big)ds\Big]=o(|\lambda'-\lambda|)$, as $\lambda'\rightarrow\lambda.$

In order to estimate
\begin{equation*}
R^\xi(\lambda,\lambda')=\displaystyle\int_0^1E^{Q} \Big[\Big( DF^\xi_{Q}\big(\mathcal{E}_t^{\lambda,\lambda'}(s)\big) -DF^\xi_{Q}(\mathcal{E}_t^{\lambda}) \Big)( \mathcal{E}^{\lambda'}_t- \mathcal{E}^{\lambda}_t)\Big]ds,
\end{equation*}
we assume that the derivative w.r.t. the density $DF^\xi_{Q}(\cdot):\mathcal{L}^{Q}\cap L^2(\Omega,\mathcal{F},Q)\rightarrow \mathbb{R}$ is $L^2(Q)$-Lipschitz.
\begin{remark}\label{Reshort}
We note that
$$ DF^\xi_Q(L)=\partial_L F\big(Q_\xi,(LQ)_\xi,\xi\big),\, L\in\mathcal{L}^{Q}\cap L^2(\Omega,\mathcal{F},Q), $$
is $L^2(Q)$-Lipschitz w.r.t. $L$, if there is a constant $C\in\mathbb{R}_+$ such that
$$ \big|\partial_L F\big(Q_\xi,\mu,\xi\big)-\partial_L F\big(Q_\xi,\mu',\xi\big)\big|\leq CW_1(\mu,\mu'),\ \mu,\mu'\in\mathcal{P}_2(\mathbb{R}^d).  $$
Indeed, using the Kantorovich-Rubinstein duality, one checks easily that
$$ W_1\big((LQ)_\xi,(L'Q)_\xi\big)\leq E^Q[|L-L'|\cdot|\xi|]\leq \big(E^Q[|\xi|^2]\big)^{\frac{1}{2}}\big(E^Q[|L-L'|^2]\big)^{\frac{1}{2}}, $$
for all $L,\,L'\in\mathcal{L}^{Q}\cap L^2(\Omega,\mathcal{F},Q)$.
\end{remark}

Then, recalling that $\mathcal{E}^\lambda_t=L_t=E^Q[L\,|\,\mathcal{F}_t]$ is bounded,
we have
\begin{equation*}
\begin{split}
|R^\xi(\lambda,\lambda')|&\le CE^Q\big[|\mathcal{E}^{\lambda'}_t -\mathcal{E}^\lambda_t|^2\big]\le CE^{Q^\lambda_t} \big[|\mathcal{E}^{\lambda,\lambda'}_t -1|^2\big]=\displaystyle CE^{Q^\lambda_t} \Big[\Big|\int_0^t (\gamma_s^{\lambda'}-\gamma_s^\lambda)\mathcal{E}^{\lambda, \lambda'}_sdB^\lambda_s \Big|^2\Big]\\
&=\displaystyle C|\lambda'-\lambda|^2 E^{Q^\lambda_t} \Big[\Big|\int_0^t \theta_s\mathcal{E}^{\lambda,\lambda'}_sdB^\lambda_s \Big|^2\Big]=\displaystyle C|\lambda'-\lambda|^2 E^{Q^\lambda_t} \Big[\int_0^t |\theta_s\mathcal{E}^{\lambda,\lambda'}_s|^2ds\Big].
\end{split}
\end{equation*}
Recalling that the smooth Wiener step process $\theta$ is bounded
and $\mathcal{E}^{\lambda, \lambda'}_s,\, s\in[0,1],\, \lambda'\in\Lambda$, is bounded in $L^{Q^\lambda_t}$, we obtain
$$\displaystyle |R^\xi(\lambda,\lambda')|\le C|\lambda'-\lambda|^2,\, \lambda'\in\Lambda.$$
Together with (\ref{d111}) and (\ref{d111c}) this yields
\begin{equation}\label{d111d}
f\big((\mathcal{E}_t^{\lambda'}Q)_\xi\big)-f\big((\mathcal{E}_t^{\lambda}Q)_\xi\big)=\displaystyle (\lambda'-\lambda)E^{Q^\lambda_t}\Big[\partial_x(\partial_1 F)\big((Q_t^\lambda)_\xi,\xi\big)\int_0^t D^\lambda_s[\xi]\theta_sds\Big]+o(|\lambda'-\lambda|),
\end{equation}
as $\lambda'\rightarrow\lambda.$ Comparing this relation with (\ref{d111a}) combined with (\ref{daaa3}), we get
\begin{equation*}
E^{Q^\lambda_t}\Big[\partial_x(\partial_1 F)\big((Q_t^\lambda)_\xi,\xi\big)\int_0^t D^\lambda_s[\xi]\theta_sds\Big]=E^{Q_t^\lambda}\Big[(\partial_\mu f)\big((Q_t^\lambda)_{\xi},\xi\big)\int_0^tD_s^\lambda[\xi]\theta_sds\Big].
\end{equation*}
Hence, for $\Gamma_\cdot:L^2(\Omega,\mathcal{F},Q^\lambda_t;\mathbb{R}^d)\rightarrow L^2(\Omega,\mathcal{F},Q^\lambda_t;\mathbb{R}^d)$ defined by
$$\displaystyle \Gamma_\xi:=\partial_x(\partial_1 F)\big((Q_t^\lambda)_\xi,\xi\big)-(\partial_\mu f)\big((Q_t^\lambda)_{\xi},\xi\big),$$
it holds
\begin{equation}\label{daab1}
E^{Q^\lambda_t}\Big[\Gamma_\xi\int_0^tD^\lambda_s[\xi]\theta_sds\Big]=0,
\end{equation}
for all smooth step process $\theta$ and for any smooth Wiener functional $\xi$.

Let us assume that the functions
$\partial_x(\partial_1 F)(\cdot,\cdot),\, \partial_\mu f(\cdot,\cdot):\mathcal{P}_2(\mathbb{R}^d)\times\mathbb{R}^d\rightarrow \mathbb{R}^d$ are bounded and continuous with a continuity modulus $\rho:\mathbb{R}_+\rightarrow\mathbb{R}_+,\, \rho(0+)=0$, which implies that $\Gamma_\vartheta$ is bounded and
\begin{equation}\label{daaac1}
|\Gamma_{\vartheta}-\Gamma_{\zeta}|\le 2\rho\Big(W_2\big((Q^\lambda_t)_\vartheta,(Q^\lambda_t)_\zeta\big)+|\vartheta-\zeta|\Big),\ \vartheta,\ \zeta\in L^2(\Omega,\mathcal{F},Q^\lambda_t).
\end{equation}

We denote by $e_i$ ($1\le i\le d$) the $i$-th unit vector in $\mathbb{R}^d$: $e_{i,j}:=1$ for $j=i$, $e_{i,j}:=0$ for $j\not=i$. Let $t=1$. For $\varepsilon>0$ we put $\xi^{i}_\varepsilon:=E^{Q}[\xi\,|\,\mathcal{F}_{t-\varepsilon}]+\varepsilon(B^\lambda_t-B^\lambda_{t-\varepsilon})e_i$. Then $D^\lambda_s[\xi^{i}_\varepsilon]=\varepsilon e_i,\, s\in[t-\varepsilon,t]$, and choosing the smooth Wiener process $\theta^\varepsilon$ such that $\theta_s^\varepsilon=0,\, s\in[0,t-\varepsilon]$, we get from (\ref{daab1})
\begin{equation}\label{daab2}
E^{Q^\lambda_t}\Big[\Gamma_{\xi^{i}_\varepsilon}e_i\int_{t-\varepsilon}^t\theta_s^\varepsilon ds\Big]=0,\, 1\le i\le d,
\end{equation}
for all smooth Wiener process $\theta^\varepsilon$. On the other side it is easy to check that
$\displaystyle\big\{\int_{t-\varepsilon}^t\theta_s^\varepsilon ds\, :\, \theta^\varepsilon\, \mbox{ smooth}$ $\displaystyle\mbox{Wiener functonal}\,\big\}$ is dense in $L^2(\Omega,\mathcal{F},Q)$, and thus also in $L^2(\Omega,\mathcal{F},Q^\lambda_t).$ Hence, it follows from (\ref{daab2}) that $\Gamma_{\xi^{i}_\varepsilon}e_i=0$, $1\le i\le d$, $Q^\lambda_t$-a.s. Finally, taking into account (\ref{daaac1})
and that $\xi^{i}_\varepsilon\rightarrow \xi$, $Q$-a.s. and in $L^2(\Omega,\mathcal{F},Q^\lambda_t)$, as $\varepsilon\rightarrow 0$, $1\le i\le d$, we conclude $\Gamma_\xi=0$, $Q^\lambda_t$-a.s., i.e.,
$$\displaystyle \partial_x(\partial_1 F)\big((Q_t^\lambda)_\xi,x\big)=\partial_\mu f\big((Q_t^\lambda)_{\xi},x\big),$$
which holds a priori $(Q_t^\lambda)_\xi$-a.s., but using the continuity of $\partial_x(\partial_1 F)(\cdot,\cdot),\, \partial_\mu f(\cdot,\cdot):\mathcal{P}_2(\mathbb{R}^d)\times\mathbb{R}^d\rightarrow \mathbb{R}^d$, we obtain that the equality holds for all $x\in\mathbb{R}^d$ and for all $\xi\in L^2(\Omega,\mathcal{F},Q^\lambda_t)$. Finally, for $t=1$ we have $Q^\lambda_t=LQ=Q_L$, i.e.,
\begin{equation}\label{d2a}
\partial_x(\partial_1 F)\big((Q_L)_\xi,x\big)=\partial_\mu f\big((Q_L)_{\xi},x\big),
\end{equation}
for all $x\in\mathbb{R^d},\, \xi\in L^2(\Omega,\mathcal{F},Q_L)$ and $L\in\mathcal{L}^Q$ smooth Wiener functional bounded from below by a strictly positive constant.

Finally, Proposition \ref{lem2} and the arguments developed in its proof show how to extend this result to general bounded densities $L\in\mathcal{L}^Q$ which are bounded from below by a strictly positive constant. This allows to state the following:
\begin{theorem}
Let $L\in\mathcal{L}^Q$ be such that, for some real constants $C,c>0$, $c\le L\le C$, and suppose that, for all $\xi\in L^4(\Omega,\mathcal{F},Q;\mathbb{R}^d)$ and for $f:\mathcal{P}_2(\mathbb{R}^d)\rightarrow \mathbb{R}$ satisfying Assumption 1 we have that the function $F_Q^\xi(L'):=f\big((L'Q)_\xi\big),\, L'\in \mathcal{L}^Q\cap L^2(\Omega,\mathcal{F},Q),$ is $L^2(Q)$-differentiable, $DF_Q^\xi(\cdot)$ is $L^2(Q)$-Lipschitz, $\partial_1 F\big((Q_L)_\xi,\cdot\big)$ is continuously differentiable, and $\partial_x(\partial_1 F)(\cdot,\cdot),\,\partial_\mu f(\cdot,\cdot):\mathcal{P}_2(\mathbb{R}^d)\times\mathbb{R}^d\rightarrow \mathbb{R}$ are bounded and continuous (with continuity modulus). Then,
$$\partial_x(\partial_1 F)\big((Q_L)_\xi,x\big)=\partial_\mu f\big((Q_L)_{\xi},x\big),\, x\in\mathbb{R}^d.$$
\end{theorem}

\smallskip
At the end of this section let us briefly discuss how the derivative w.r.t. the density $L\in\mathcal{L}^Q$ of a function $f\big((LQ)_\xi\big)$ presents, when $f\big((LQ)_\xi\big)=\Phi(f^{LQ}_\xi)$ is defined as a differentiable function $\Phi$ of the density $f^{LQ}_\xi$ of the law of a random variable $\xi\in L^2(\Omega,\mathcal{F},LQ)$. Such an approach basing on densities of random variables has been chosen in recent works by Bensoussan et al.; see, for instance, \cite{BFY15} and the papers cited therein. However, we adapt this setting slightly, in order to have it in coherence with the theory we have developed in the present work.

Let us denote by $\mathcal{L}^{\lambda_1}$ the set of all probability densities over $(\mathbb{R},\mathcal{B}(\mathbb{R}),dx)$. Given a complete probability space $(\Omega,\mathcal{F},Q)$ and a function $\Phi: \mathcal{L}^{\lambda_1}\cap L^2(\mathbb{R},\mathcal{B} (\mathbb{R}),dx)\rightarrow\mathbb{R}$, we put $f(Q_\xi):=\Phi(f^Q_\xi)$, for $\xi\in L^2(\Omega,\mathcal{F},Q)$ with density $f_\xi^Q\in L^2(\mathbb{R},\mathcal{B} (\mathbb{R}),dx)\rightarrow\mathbb{R}$.
We suppose $\Phi:\mathcal{L}^{\lambda_1}\cap L^2(\mathbb{R},\mathcal{B} (\mathbb{R}),dx)\rightarrow\mathbb{R}$ is $L^2(dx)$-differentiable in the sense of Definition \ref{def3.3}, i.e., for any $h\in \mathcal{L}^{\lambda_1}\cap L^2(\mathbb{R},\mathcal{B} (\mathbb{R}),dx)$ there exists $(D\Phi)(h)\in L(L^2_0(\mathbb{R},\mathcal{B} (\mathbb{R}),dx),\mathbb{R})$ such that, identified by the Riesz Representation Theorem with some $D\Phi(h,\cdot)\in L^2_0(\mathbb{R},\mathcal{B} (\mathbb{R}),dx)$, we have
\begin{equation*}
\begin{split}
\Phi(h')-\Phi(h)&=(D\Phi)(h)(h'-h)+o(|h'-h|_{L^2(\mathbb{R},dx)})\\
&=\int_{\mathbb{R}}D\Phi(h,x)\big(h'(x)-h(x)\big)dx+o(|h'-h|_{L^2(\mathbb{R},dx)}).
\end{split}
\end{equation*}
We suppose that $D\Phi(h,\cdot)$ is differentiable, and the derivative belongs to $L^2(\mathbb{R},\mathcal{B} (\mathbb{R}),dx).$

Let $\xi\in L^2(\Omega,\mathcal{F},Q)$ with density $f^Q_\xi\in  L^2(\mathbb{R},\mathcal{B} (\mathbb{R}),dx)$ and $\eta\in L^2(\Omega,\mathcal{F},Q)$ be such that the couple $(\xi,\eta)$ has a density $f^Q_{(\xi,\eta)}\in L^2(\mathbb{R}^2,\mathcal{B} (\mathbb{R}^2),dz)$ which is continuously differentiable in its first variable. Then, for all $\varepsilon\in\mathbb{R}$, $\xi+\varepsilon\eta$ has the density
$$ f^Q_{\xi+\varepsilon\eta}(x)=\int_{\mathbb{R}}f^Q_{(\xi,\eta)}(x-\varepsilon y,y)dy,\, x\in\mathbb{R},$$
and, under suitable additional integrability assumptions for $\partial_x f^Q_\xi(\cdot,\cdot)$,
$$ \partial_\varepsilon f^Q_{\xi+\varepsilon\eta}(x)_{|\varepsilon=0}=-\int_{\mathbb{R}}y\partial_xf^Q_{(\xi,\eta)}(x,y)dy,\,  x\in\mathbb{R}.$$
Then, using the integration by parts formula,
\begin{equation*}
\begin{split}
&\partial_\varepsilon f(Q_{\xi+\varepsilon\eta})_ {|\varepsilon=0}= (D\Phi)(f^Q_\xi) \Big(\!\!-\!\!\int_{\mathbb{R}}y\partial_xf^Q_{(\xi,\eta)}(\cdot,y)dy\Big)=\int_{\mathbb{R}}\Big(\!\!-\!\!\int_{\mathbb{R}} D\Phi(f^Q_\xi,x)y\partial_x f^Q_{(\xi,\eta)}(x,y)dx\Big)dy\\
=&\int_{\mathbb{R}}\Big(\int_{\mathbb{R}} \partial_xD\Phi(f^Q_\xi,x)yf^Q_{(\xi,\eta)}(x,y)dx\Big)dy= E^Q\big[\partial_xD\Phi(f^Q_\xi,\xi)\eta\big],
\end{split}
\end{equation*}

\noindent which shows that, if $f:\mathcal{P}_2(\mathbb{R}) \rightarrow\mathbb{R}$ defined above is differentiable, then, because of the density of the set of  random variables $\eta$ in $L^2(\Omega,\mathcal{F},Q)$ satisfying the above assumptions on the density $f_{(\xi,\eta)}$, we have $\partial_\mu f(Q_\xi,\xi)=\partial_xD\Phi(f^Q_\xi,\xi),\, Q$-a.s., i.e.,
\begin{equation}\label{BFY1}
\partial_\mu f(Q_\xi,x)=\partial_xD\Phi(f^Q_\xi,x),\ Q_\xi(dx)\mbox{-a.s.}
\end{equation}

Let now $L\in \mathcal{L}^Q$ for which we suppose, for simplicity, $0<c\le L\le C\in\mathbb{R}_+.$ We remark that, for $\xi\in L^2(\Omega,\mathcal{F},Q)$ with density $f^Q_\xi\in L^2(\mathbb{R},\mathcal{B} (\mathbb{R}),dx)$, the density under $Q_L:=LQ$ is, obviously, given by $f^{LQ}_\xi(x)=E^Q[L\,|\, \xi=x]f^Q_\xi(x),\, x\in\mathbb{R}$, and it belongs to $L^2(\Omega,\mathcal{F},Q_L)(=L^2(\Omega,\mathcal{F},Q)).$
Given another $L'\in\mathcal{L}^Q$ for which we suppose, for simplicity, $0<c\le L'\le C\in\mathbb{R}_+,$ we put $L^\varepsilon:=(1-\varepsilon)L+\varepsilon L'\in\mathcal{L}^Q,\, \varepsilon\in[0,1].$ Then $\displaystyle \partial_\varepsilon f_\xi^{L^\varepsilon Q}(x)=E^Q[L'-L\,|\, \xi=x]f_\xi^Q(x),\, x\in\mathbb{R},$\ and
\begin{equation*}
\begin{split}
&\partial_\varepsilon\Phi(f^{L^\varepsilon Q}_\xi)_{|\varepsilon=0}=(D\Phi)(f^{L Q}_\xi)\big(E^Q[L'-L\,|\,\xi=x]f^Q_\xi(x)\big)\\
=&\int_{\mathbb{R}} D\Phi(f^{L Q}_\xi,x)E^Q[L'-L\,|\,\xi=x]f^Q_\xi(x)dx=E^Q\big[D\Phi(f^{L Q}_\xi,\xi)(L'-L)\big].
\end{split}
\end{equation*}

\noindent Consequently, if $F^\xi_Q(L):=f\big((LQ)_\xi\big),\, L\in\mathcal{L}^Q\cap L^2(\Omega,\mathcal{F},Q)$, is differentiable (in the sense of Definition \ref{def3.3}), then
$$ DF^\xi_Q(L)=D\Phi(f^{LQ}_\xi,\xi)- E^Q\big[D\Phi(f^{LQ}_\xi, \xi)\big],\, Q\mbox{-a.s.}$$
On the other hand, recall that
\begin{equation*}
\begin{split}
 DF^\xi_{Q_L}(1)&=DF^\xi_Q(L)-E^{Q_L}\big[DF^\xi_Q(L)\big]=D\Phi(f^{LQ}_\xi,\xi)-E^{Q_L}\big[D\Phi(f^{LQ}_\xi,\xi)\big],
\, Q\mbox{-a.s.},
\end{split}
\end{equation*}

\noindent from where we deduce that
\begin{equation}\label{BFY2}
\partial_1 F\big((Q_L)_\xi,x\big)=D\Phi(f^{LQ}_\xi,x)-E^{Q_L}\big[D\Phi(f^{LQ}_\xi,\xi)\big],
\, (Q_L)_\xi\mbox{-a.s.}
\end{equation}
Finally, combining \eqref{BFY1} and \eqref{BFY2}, we see that also here it holds that
$$\partial_x(\partial_1 F)\big((Q_L)_\xi,x\big)=\partial_x D\Phi(f^{LQ}_\xi,x)=\partial_\mu f\big((Q_L)_\xi,x\big),\, \, (Q_L)_\xi(dx)\mbox{-a.s.}$$

\section{Appendix}

\subsection{Fr\'{e}chet differentiability of the lifted function}

\begin{lemma}\label{LemA.1}
Let $f:\mathcal{P}_2(\mathbb{R}^d)\rightarrow\mathbb{R}$ be differentiable, and let $Q'$ be any probability measure on the Radon space $(\Omega,\mathcal{F})$. Then $f$ lifted w.r.t. the probability measure $Q'$, $\widetilde{f}'(\eta'):=f(Q'_{\eta'})$, $\eta'\in L^2(\Omega,\mathcal{F},Q';\mathbb{R}^d)$, is Fr\'{e}chet differentiable over $L^2(\Omega,\mathcal{F},Q';\mathbb{R}^d)$, and $D\widetilde{f}'(\eta')=(\partial_\mu f)(Q'_{\eta'},\eta')$, $Q'$-a.s.
\end{lemma}
\begin{proof}
We suppose that the differentiability of $f:\mathcal{P}_2(\mathbb{R}^d)\rightarrow\mathbb{R}$ is defined through the lifted function $\widetilde{f}(\eta):=f(Q_{\eta})$, $\eta\in L^2(\Omega,\mathcal{F},Q;\mathbb{R}^d)$, for the probability $Q$ fixed in Section 3. Now, for an arbitrarily given $\eta'_0\in L^2(\Omega,\mathcal{F},Q';\mathbb{R}^d)$, let $\eta_0\in L^2(\Omega,\mathcal{F},Q;\mathbb{R}^d)$ be such that $Q_{\eta_0}=Q'_{\eta'_0}$ and $\widetilde{f}$ is Fr\'echet differentiable at $\eta_0$, and for $\eta'\in L^2(\Omega,\mathcal{F},Q';\mathbb{R}^d)$ with $|\eta'-\eta'_0|_{L^2(Q)}\rightarrow0$, let $\eta\in L^2(\Omega,\mathcal{F},Q;\mathbb{R}^d)$ be such that $Q_{(\eta,\eta_0)}=Q'_{(\eta',\eta'_0)}$ (see Proposition 3.1 in \cite{BCL}). Then $|\eta-\eta_0|_{L^2(Q)}=|\eta'-\eta'_0|_{L^2(Q')}\rightarrow0$, and
\begin{equation*}
\begin{split}
&\widetilde{f}'(\eta')-\widetilde{f}'(\eta'_0)=f(Q'_{\eta'})-f(Q'_{\eta'_0})=f(Q_{\eta})-f(Q_{\eta_0})=\widetilde{f}(\eta)-\widetilde{f}(\mu_0)\\
=&D\widetilde{f}(\eta_0)(\eta-\eta_0)+o(|\eta-\eta_0|_{L^2(Q)})=E^Q\big[(\partial_\mu f)(Q_{\eta_0},\eta_0)(\eta-\eta_0)\big]+o(|\eta-\eta_0|_{L^2(Q)})\\
=&E^{Q'}\big[(\partial_\mu f)(Q'_{\eta'_0},\eta'_0)(\eta'-\eta'_0)\big]+o(|\eta'-\eta'_0|_{L^2(Q')}).
\end{split}
\end{equation*}
It follows that $\widetilde{f}':L^2(\Omega,\mathcal{F},Q';\mathbb{R}^d)\rightarrow\mathbb{R}$ is Fr\'{e}chet differentiable, and $D\widetilde{f}'(\eta'_0)=(\partial_\mu f)(Q'_{\eta'_0},\eta'_0)$, $Q'$-a.s.

\end{proof}

\subsection{Proof of Proposition \ref{lem2}}
This part is devoted to the proof of Proposition \ref{lem2}.
\begin{proof} (of Proposition \ref{lem2}).
The proof is split into seven steps. For simplicity we choose $T=1$ and $\Lambda=\mathbb{R}$. Recall that the mapping $\lambda\rightarrow L^\lambda\in\mathcal{L}^Q\cap L^2(\Omega,\mathcal{F},Q)$ is supposed to be continuously $L^2(Q)$-differentiable.

\noindent\textbf{Step 1.} We introduce the partition $t_i^n:=i\cdot 2^{-n}$ ($0\leq i\leq N_n=2^n$) for the interval $[0,1]$, and we put $\Delta_i^n=(t_{i-1}^n,t_i^n]$. Let us consider the increasing sequence of sub $\sigma$-fields $\mathcal{G}_n:=\sigma\big\{B(\Delta_i^n)(:=B_{t_i^n}-B_{t_{i-1}^n}),\ 1\leq i\leq N_n\big\}\vee\mathcal{N}_Q,\ n\geq1,$ with limit $\bigvee\limits_{n\geq1}\mathcal{G}_n$ $\big(:=\sigma\big\{\bigcup_{n\ge 1}\mathcal{G}_n\big\}\big)=\mathcal{F}$.

We define now $L_n^\lambda:=E^Q[L^\lambda\,|\,\mathcal{G}_n],\ \lambda\in\Lambda,\ n\geq1$. Then we have the following properties.\\
\textbf{a)} $L_n^\lambda\in\mathcal{L}^Q\cap L^2(\Omega,\mathcal{F},Q),\, \lambda\in\Lambda,\, n\ge 1$.\\
\textbf{b)} The mapping $\lambda\rightarrow L_n^\lambda$ is continuously $L^2(Q)$-differentiable, and $\partial_\lambda L_n^\lambda=E^Q[\partial_\lambda L^\lambda\,|\,\mathcal{G}_n],\ Q$-a.s., $\lambda\in\Lambda,\ n\geq1$. Indeed,
\begin{equation*}
\begin{split}
&\big|L_n^{\lambda'}-L_n^{\lambda}-(\lambda'-\lambda)E^Q[\partial_\lambda L^\lambda\,|\,\mathcal{G}_n]\big|_{L^2(Q)}=\big|E^Q[L^{\lambda'}-L^{\lambda}-(\lambda'-\lambda)\partial_\lambda L^\lambda\,|\,\mathcal{G}_n]\big|_{L^2(Q)}\\
\leq&\big|L^{\lambda'}-L^{\lambda}-(\lambda'-\lambda)\partial_\lambda L^\lambda\big|_{L^2(Q)}=o(|\lambda'-\lambda|),\ \lambda'\rightarrow\lambda.
\end{split}
\end{equation*}
Moreover, for all $\lambda\in\Lambda$,
$$ \big|\partial_\lambda L_n^{\lambda'}-\partial_\lambda L_n^{\lambda}\big|_{L^2(Q)}=\big|E^Q[\partial_\lambda L^{\lambda'}-\partial_\lambda L^{\lambda}\,|\,\mathcal{G}_n]\big|_{L^2(Q)}\leq\big|\partial_\lambda L^{\lambda'}-\partial_\lambda L^{\lambda}\big|_{L^2(Q)}\rightarrow0,\ \lambda'\rightarrow\lambda. $$
\textbf{c)} From the martingale convergence theorem we have\\
\indent i) $L_n^\lambda=E^Q[L^\lambda\,|\,\mathcal{G}_n] \xrightarrow[n\rightarrow\infty]{L^2(Q)}E^Q[L^\lambda\,|\,\mathcal{F}]=L^\lambda,\ \lambda\in\Lambda$, and\\
\indent ii) $\partial_\lambda L_n^\lambda=E^Q[\partial_\lambda L^\lambda\,|\,\mathcal{G}_n]\xrightarrow[n\rightarrow\infty]{L^2(Q)}E^Q[\partial_\lambda L^\lambda\,|\,\mathcal{F}]=\partial_\lambda L^\lambda,\ \lambda\in\Lambda$.\\
\textbf{d)} As $\lambda\rightarrow L^\lambda$ and $\lambda\rightarrow \partial_\lambda L^\lambda$ are $L^2(Q)$-continuous mappings, we have for all $\lambda,\,\lambda'\in\Lambda$ with $[\lambda,\lambda']:=\big\{\lambda(s):=s\lambda'+(1-s)\lambda,\, s\in[0,1]\big\}\subset\Lambda$,
$$E^Q\big[|(L_n^{\lambda(s)},\partial_\lambda L_n^{\lambda(s)})|^2\big]\leq E^Q\big[|(L^{\lambda(s)},\partial_\lambda L^{\lambda(s)})|^2\big]\leq \max\limits_{s\in[0,1]}E^Q\big[|(L^{\lambda(s)},\partial_\lambda L^{\lambda(s)})|^2\big]<+\infty,$$
$s\in[0,1],\ n\geq1$. This combined with c) shows that
$$(L_n^{\lambda(\cdot)},\partial_\lambda L_n^{\lambda(\cdot)})\xrightarrow[n\rightarrow\infty]{L^2(dsdQ)}(L^{\lambda(\cdot)},\partial_\lambda L^{\lambda(\cdot)}).$$

\noindent\textbf{Step 2}. Let us fix now an arbitrary $n\ge 1$, and put $N=2^n,\ \Delta_i=\Delta_i^n,\ t_i=t_i^n$. In Step 1 we have seen that it suffices to prove the approximation result for mappings $\lambda\rightarrow L^\lambda\in\mathcal{L}^Q\cap L^2(\Omega,\mathcal{G}_n,Q)$ which are continuously $L^2(Q)$-differentiable. As $(L,\partial L):\Lambda\times\Omega\rightarrow \mathbb{R}^2$ is $\mathcal{B}(\Lambda)\otimes\mathcal{G}_n$-measurable, there exists a Borel function $(h,g):\Lambda\times\mathbb{R}^N\rightarrow\mathbb{R}^2$ such that
$$ (L^\lambda,\partial_\lambda L^\lambda)=(h^\lambda,g^\lambda)\big(B(\Delta_1),...,B(\Delta_N)\big),\ \lambda\in\Lambda.$$
Observe that the probability law
$\mu_N:=Q_{\theta_N}$ of $\theta_N:=\big(B(\Delta_1),...,B(\Delta_N)\big)$ is just the Gaussian law $\mathcal{N}(0,\frac{1}{N}I_N)$ with mean zero and covariance matrix $\frac{1}{N}I_N$ ($I_N$ is the unit matrix over $\mathbb{R}^N$), defined over $(\mathbb{R}^N,\mathcal{B}(\mathbb{R}^N))$. We observe that  $\lambda\rightarrow h^\lambda\in L^2(\mathbb{R}^N,\mu_N)$
is continuously differentiable, and $\partial_\lambda h^\lambda=g^\lambda,\, \mu_N$-a.s., $\lambda\in\Lambda$.
Indeed, $|h^{\lambda'}-h^\lambda-(\lambda'-\lambda)g^\lambda|_{L^2(\mu_N)}= |L^{\lambda'}-L^\lambda-(\lambda'-\lambda)\partial_\lambda L^\lambda|_{L^2(Q)}=o(|\lambda'-\lambda|),\, \lambda'\in\Lambda$ with $\lambda'\rightarrow \lambda.$ The $L^2(\mu_N)$-continuity of $\lambda\rightarrow g^\lambda(=\partial_\lambda h^\lambda)$ follows similarly from the $L^2(Q)$-continuity of $\lambda\rightarrow \partial_\lambda L^\lambda.$

\noindent\textbf{Step 3.} For $\ell\ge 1$, let $\psi_\ell\in C_{\ell,b}^\infty(\mathbb{R})$ such that $-\ell\le \psi_\ell\le \ell,$ $\psi_\ell(r)=r$ for $|r|\le \ell-2,$ $\psi_\ell(r)=\ell$ for $r\ge\ell$, $\psi_\ell(r)=-\ell$ for $r\le -\ell,$ and $0\le \partial_r\psi_\ell\le 1.$ Moreover, let  $\varphi_\ell\in C^\infty_b(\mathbb{R}^N)$ be such that $0\le \varphi_\ell\le 1$, $\varphi_\ell(x)=1$ for $|x|\le \ell-2$, $\varphi_\ell(x)=0$ for $|x|\ge \ell$, and $|\partial_x\varphi_\ell|\le 1.$ We need also a function $\widetilde{\varphi}\in C^\infty_b(\mathbb{R})$ with the same properties as $\varphi_\ell$, but defined on $\mathbb{R}$ instead of $\mathbb{R}^N$. Given these functions, we define for $h^\lambda$ introduced in the preceding Step 2,
$$\displaystyle h_\ell(\lambda,x):=h^\lambda_\ell(x):=\psi_\ell\big(h^{\psi_\ell(\lambda)}(x)\big)\varphi_\ell(x)\widetilde{\varphi}_\ell(\lambda),\, (\lambda,x)\in \mathbb{R}\times\mathbb{R}^N.$$
Obviously, $|h_\ell|\leq \ell$ on $\mathbb{R}\times\mathbb{R}^N$, $h_\ell(\lambda,x)=0$ if $|\lambda|\ge
\ell$ or $|x|\ge\ell$, and from the dominated convergence theorem and the continuous $L^2(\mu_N)$-differentiability of $\lambda\rightarrow h^\lambda$ we obtain $(h^\lambda_\ell,\partial_\lambda h^\lambda_\ell)\rightarrow (h^\lambda,\partial_\lambda h^\lambda)$ in $L^2{(\mathbb{R}^N,\mu_N)}$ as well as $(h^{\lambda(\cdot)}_\ell,\partial_\lambda h^{\lambda(\cdot)}_\ell)\rightarrow (h^{\lambda(\cdot)},\partial_\lambda h^{\lambda(\cdot)})$ in $L^2{([0,1]\times\mathbb{R}^N,ds\mu_N(dx))}$, as $\ell\rightarrow +\infty.$ In particular, it holds that $\displaystyle \int_{\mathbb{R}^N}h^\lambda_\ell(x)\mu_N(dx)\rightarrow \int_{\mathbb{R}^N}h^\lambda(x)\mu_N(dx)=E^Q[L^\lambda]=1$, as $\ell\rightarrow+\infty$.

\noindent\textbf{Step 4.} Let us now approximate $h^\lambda_\ell,\, \ell\ge 1,$  with the help of Sobolev's mollification procedure. For this we consider $\chi^2\in C^\infty(\mathbb{R}^N;\mathbb{R}_+)$ with supp$(\chi^2)\subset\overline{B_1(0)}\big(:=\big\{x\in\mathbb{R}^N\,\big|\, |x|\le 1\big\}\big)$ and  $\int_{\mathbb{R}^N}\chi^2(y)dy=1$. In the same way we introduce $\chi^1\in C^\infty(\mathbb{R}^N;\mathbb{R}_+)$, but now for $N=1$, and, given $\varepsilon>0$, we put $\chi^1_\varepsilon(\lambda)=\frac{1}{\varepsilon}\chi^1(\frac{1}{\varepsilon}\lambda),\, \chi^2_\varepsilon(x)=\frac{1}{\varepsilon^N}\chi^1(\frac{1}{\varepsilon}x)$ and $\chi_\varepsilon(\lambda,x)=\chi^1_\varepsilon(\lambda)\chi^2_\varepsilon(x),\, (\lambda,x)\in\mathbb{R}\times\mathbb{R}^N$. Given $h_\ell(\lambda,x)=h^\lambda_\ell(x)$ from the preceding step, we define $h_{\ell,\varepsilon}$ by the convolution
$$ h^\lambda_{\ell,\varepsilon}(x):=h_{\ell,\varepsilon}(\lambda,x):=(h_\ell*\chi_\varepsilon)(\lambda,x)=\int_{\mathbb{R}\times\mathbb{R}^N}h_\ell(\lambda-\lambda',x-x')\chi_\varepsilon(\lambda',x')d\lambda'dx',\ (\lambda,x)\in\mathbb{R}\times\mathbb{R}^N.$$
We remark that, for all $\varepsilon\in(0,1)$, $h_{\ell,\varepsilon}\in C^\infty_b(\mathbb{R}\times\mathbb{R}^N)$ with supp$(h_{\ell,\varepsilon})\subset(-\ell-1,\ell+1)\times B_{\ell+1}(0)$. Hence, for all $\lambda\in\Lambda=\mathbb{R}$,
\begin{equation*}
\begin{split}
&|\partial_\lambda h^\lambda_{\ell,\varepsilon}-\partial_\lambda h^\lambda_{\ell}|^2_{L^2(\mu_N)}\le C_N\int_{\mathbb{R}^N}|\partial_\lambda h^\lambda_{\ell,\varepsilon}(x)-\partial_\lambda h^\lambda_{\ell}(x)|^2dx\\
\le& C_N\int_{\mathbb{R}}\int_{\mathbb{R}^N}|(\partial_\lambda h^{\lambda-\lambda'}_{\ell}*\chi^2_\varepsilon)(x)-\partial_\lambda h^\lambda_{\ell}(x)|^2dx\cdot \chi^1_\varepsilon(\lambda')d\lambda'\\
\le& C_N\int_{\mathbb{R}}\int_{\mathbb{R}^N}|(\partial_\lambda h^{\lambda-\lambda'}_{\ell}*\chi^2_\varepsilon)(x)-( \partial_\lambda h^\lambda_{\ell}*\chi^2_\varepsilon)(x)|^2dx \cdot \chi^1_\varepsilon(\lambda')d\lambda'\\
& \qquad+C_N\int_{\mathbb{R}^N}|( \partial_\lambda h^{\lambda}_{\ell}*\chi^2_\varepsilon)(x)- \partial_\lambda h^\lambda_{\ell}(x)|^2dx\\
\le& C_N\int_{\mathbb{R}}\int_{\mathbb{R}^N}|\partial_\lambda h^{\lambda-\lambda'}_{\ell}(x)- \partial_\lambda h^\lambda_{\ell}(x)|^2dx \cdot \chi^1_\varepsilon(\lambda')d\lambda'+C_N\int_{\mathbb{R}^N}|( \partial_\lambda h^{\lambda}_{\ell}*\chi^2_\varepsilon)(x)- \partial_\lambda h^\lambda_{\ell}(x)|^2dx,
\end{split}
\end{equation*}
where from Step 3 and the $L^2(\mu_N)$-continuity of $\lambda\rightarrow \partial_\lambda h^\lambda$ we have
$$\int_{\mathbb{R}}\int_{\mathbb{R}^N}|\partial_\lambda h^{\lambda-\lambda'}_{\ell}(x)- \partial_\lambda h^\lambda_{\ell}(x)|^2dx \cdot \chi^1_\varepsilon(\lambda')d\lambda' \le C_{N,l}\sup_{|\lambda'|\le\varepsilon}|\partial_\lambda h^{\lambda-\lambda'}_{\ell}- \partial_\lambda h^\lambda_{\ell}|^2_{L^2(\mu_N)}\rightarrow 0,\mbox{ as }\varepsilon\rightarrow 0, $$
and, on the other hand, also
$$\int_{\mathbb{R}^N}|( \partial_\lambda h^{\lambda}_{\ell}*\chi^2_\varepsilon)(x)- \partial_\lambda h^\lambda_{\ell}(x)|^2dx\rightarrow 0,\, \mbox{ as } \varepsilon\rightarrow 0. $$
This latter convergence is a well-known property of the Sobolev mollification. Consequently, for all $\lambda\in\Lambda,$
$$\displaystyle |\partial_\lambda h^\lambda_{\ell,\varepsilon}-\partial_\lambda h^\lambda_{\ell}|^2_{L^2(\mu_N)}\rightarrow 0,\mbox{ and analogously, } | h^\lambda_{\ell,\varepsilon}- h^\lambda_{\ell}|^2_{L^2(\mu_N)}\rightarrow 0,\mbox{ as }\varepsilon\rightarrow 0.$$

\noindent In the same way, using the continuous $L^2(\mu_N)$-differentiability of $\lambda\rightarrow  h^\lambda_\ell$ and, thus, the boundedness of $\lambda\rightarrow (h^\lambda,\partial_\lambda h^\lambda)$ in $L^2(\mathbb{R}^N,\mu_N)$, we get that also
$$(h^{\lambda(\cdot)}_{\ell,\varepsilon}, \partial_\lambda h^{\lambda(\cdot)}_{\ell,\varepsilon})\rightarrow (h^{\lambda(\cdot)}_{\ell}, \partial_\lambda h^{\lambda(\cdot)}_{\ell}) \mbox{ in } L^2([0,1]\times\mathbb{R}^N,ds\mu_N(dx)), \mbox{ as } \varepsilon\rightarrow 0.$$

\noindent\textbf{Step 5.} Let  $h^\lambda_\varepsilon:=h^\lambda_{\ell,\varepsilon_\ell}$ from Step 4  be such that $(h^\lambda_{\ell,\varepsilon_\ell},\partial_\lambda h^\lambda_{\ell,\varepsilon_\ell})\rightarrow (h^\lambda,\partial_\lambda h^\lambda)$ in $L^2(\mu_N)$ and $(h^{\lambda(\cdot)}_{\ell,\varepsilon_\ell},\partial_\lambda h^{\lambda(\cdot)}_{\ell,\varepsilon_\ell})\rightarrow (h^{\lambda(\cdot)},\partial_\lambda h^{\lambda(\cdot)})$ in $L^2([0,1]\times \mathbb{R}^N,ds\mu_N(dx))$, as $\ell\rightarrow+\infty$, where $\varepsilon_\ell\rightarrow0$ is suitably chosen. We define $F^\lambda_\varepsilon:=h^\lambda_\varepsilon \big(B(\Delta_1),...,B(\Delta_N)\big)\in\mathcal{S},\ \lambda\in\Lambda=\mathbb{R},\ \varepsilon>0$. From Step 4 we know that $\lambda\rightarrow F^\lambda_\varepsilon$ is in $C^\infty_b(\mathbb{R})$, and $\partial_\lambda F^\lambda_\varepsilon=(\partial_\lambda h^\lambda_\varepsilon)\big(B(\Delta_1),...,B(\Delta_N)\big)$. Moreover, due to our choice,
\begin{equation*}
\begin{split}
(F^\lambda_\varepsilon,\partial_\lambda F^\lambda_\varepsilon)&\xrightarrow[\varepsilon\rightarrow0]{L^2(Q)}(L^\lambda,\partial_\lambda L^\lambda),\\
(F^{\lambda(\cdot)}_\varepsilon,\partial_\lambda F^{\lambda(\cdot)}_\varepsilon)&\xrightarrow[\varepsilon\rightarrow0]{L^2(dsdQ)}(L^{\lambda(\cdot)},\partial_\lambda L^{\lambda(\cdot)}).
\end{split}
\end{equation*}
Also note that $0\leq F^\lambda_\varepsilon\leq C_\varepsilon(=\ell),\ \lambda\in\Lambda,\ \varepsilon>0$. Let us define $L_\varepsilon^\lambda:=\frac{\varepsilon+F_\varepsilon^\lambda}{\varepsilon+E^Q[F_\varepsilon^\lambda]}$. Obviously, $L^\lambda_\varepsilon$ is a smooth Wiener functional in $\mathcal{S}$ and belongs to $\mathcal{L}^Q$. Moreover, $\frac{\varepsilon}{\varepsilon+C_\varepsilon} \le L_\varepsilon^\lambda\leq \frac{\varepsilon+C_\varepsilon}{\varepsilon}$, and standard estimates allow to show that
$\lambda\rightarrow L_\varepsilon^\lambda$ belongs to $C^\infty_b(\mathbb{R})$, and
$$\displaystyle \partial_\lambda L^\lambda_\varepsilon=\frac{\partial_\lambda F^\lambda_\varepsilon}{\varepsilon+E^Q[F^\lambda_\varepsilon]}-\frac{E^Q[\partial_\lambda F_\varepsilon^\lambda](\varepsilon+F^\lambda_\varepsilon)}{(\varepsilon+E^Q[F_\varepsilon^\lambda])^2.}$$
Moreover, as $(F^\lambda_\varepsilon,\partial_\lambda F^\lambda_\varepsilon)\xrightarrow[\varepsilon\rightarrow0]{L^2(Q)} (L^\lambda,\partial_\lambda L^\lambda)$, we have,
$E^Q[F^\lambda_\varepsilon]\rightarrow E^Q[L^\lambda]=1$ and $E^Q[\partial_\lambda F^\lambda_\varepsilon]\rightarrow E^Q[\partial_\lambda L^\lambda]=0$, and from the formula for $(L^\lambda_\varepsilon,\partial_\lambda L^\lambda_\varepsilon)$ we conclude that
\begin{equation*}
(L^\lambda_\varepsilon,\partial_\lambda L^\lambda_\varepsilon)\xrightarrow[\varepsilon\rightarrow0]{L^2(Q)}(L^\lambda,\partial_\lambda L^\lambda), \mbox{ for all }\lambda\in\Lambda.
\end{equation*}
Finally, using a similar argument, we obtain also
\begin{equation*}
(L^{\lambda(\cdot)}_\varepsilon,\partial_\lambda L^{\lambda(\cdot)}_\varepsilon)\xrightarrow[\varepsilon\rightarrow0]{L^2(dsdQ)}(L^{\lambda(\cdot)},\partial_\lambda L^{\lambda(\cdot)}).
\end{equation*}

\noindent\textbf{Step 6.} We now use $L_\varepsilon^\lambda$ from Step 5 and define $M_s^{\lambda,\varepsilon}:=E^Q[L_\varepsilon^\lambda\,|\,\mathcal{F}_s],\ s\in[0,1]$. Obviously $M^{\lambda,\varepsilon}$ is an $(\mathbb{F},Q)$-martingale, and $\frac{\varepsilon}{\varepsilon+C_\varepsilon}\leq M^{\lambda,\varepsilon}\leq\frac{\varepsilon+C_\varepsilon}{\varepsilon}$. Thanks to the Clark-Ocone formula, as $L^\lambda_\varepsilon\in\mathcal{S}$, we have for the bounded adapted process $Z_s^{\lambda,\varepsilon}=E^Q[D_s L_\varepsilon^\lambda\,|\,\mathcal{F}_s],\ s\in[0,1]$, that
\begin{equation}\label{Aaa}
 L_\varepsilon^\lambda=1+\int_0^1 Z_s^{\lambda,\varepsilon}dB_s,\ Q\mbox{-a.s.},\ \lambda\in\Lambda,
\end{equation}
and
$$ M_t^{\lambda,\varepsilon}=1+\int_0^t Z_s^{\lambda,\varepsilon}dB_s=1+\int_0^t \frac{Z_s^{\lambda,\varepsilon}}{M_s^{\lambda,\varepsilon}}M_s^{\lambda,\varepsilon}dB_s,\ t\in[0,T]. $$
We introduce the process $\gamma_s^{\lambda,\varepsilon}:=\frac{Z_s^{\lambda,\varepsilon}}{M_s^{\lambda,\varepsilon}},\ s\in[0,1]$. Obviously, $\gamma^{\lambda,\varepsilon}\in L^2_{\mathbb{F}}([0,1]\times\Omega,dsdQ)$ and
$$ L^{\lambda}_{\varepsilon}=M_1^{\lambda,\varepsilon}=\exp\Big\{\int_0^1 \gamma_s^{\lambda,\varepsilon}dB_s-\frac{1}{2}\int_0^1 |\gamma_s^{\lambda,\varepsilon}|^2ds\Big\}. $$
In order to study the properties of the process $\gamma^{\lambda,\varepsilon}$ we recall that $L_\varepsilon^\lambda=\widetilde{h}^\lambda_\varepsilon\big(B(\Delta_1),\dots,B(\Delta_N)\big)$, where $\widetilde{h}^\lambda_\varepsilon(x)=\frac{\varepsilon+h^\lambda_\varepsilon(x)}{\varepsilon+E^Q[F^\lambda_\varepsilon]}$ belongs to $C^\infty_b(\mathbb{R}\times\mathbb{R}^N)$. Let us make the simplifying assumption that, for any fixed $\varepsilon>0$, $L^\lambda_\varepsilon=\widetilde{h}^\lambda(B_1),\ \widetilde{h}\in C_b^\infty(\mathbb{R}\times\mathbb{R})$. Without any difficulty the computation can be extended to the case of a general $L_\varepsilon^\lambda=\widetilde{h}^\lambda_\varepsilon\big(B(\Delta_1),\dots,B(\Delta_N)\big)$ with $\widetilde{h}_\varepsilon\in C^\infty_b(\mathbb{R}\times\mathbb{R}^N)$.

We remark that, taking the Malliavin derivative of $L^\lambda_\varepsilon=\widetilde{h}^\lambda(B_1)$, we get
$D_s L^\lambda_\varepsilon=(\widetilde{h}^\lambda)'(B_1),\, s\in[0,1],$ and hence
$$ E^Q[D_s L^\lambda_\varepsilon\,|\,\mathcal{F}_s]=E^Q[(\widetilde{h}^\lambda)'(B_1)\,|\,\mathcal{F}_s]=\frac{1}{\sqrt{2\pi}}\int_{\mathbb{R}}(\widetilde{h}^\lambda)'(B_s+\sqrt{1-s}\cdot x)e^{-\frac{1}{2}x^2}dx. $$
We put
$$ g_1^\lambda(s,y):=\frac{1}{\sqrt{2\pi}}\int_{\mathbb{R}}(\widetilde{h}^\lambda)'(y+\sqrt{1-s}\cdot x)e^{-\frac{1}{2}x^2}dx,\ (s,y)\in[0,1]\times\mathbb{R}. $$
Then $g_1^\lambda:[0,1]\times\mathbb{R}\rightarrow\mathbb{R}$ is continuous,  $g_1^\lambda(s,\cdot)\in C^\infty$ ($s\in[0,1]$), and the derivatives of all order w.r.t. $y$ of $g_1^\lambda$ are bounded over $[0,1]\times\mathbb{R}$. Similarly we define
$$ g_2^\lambda(s,y):=\frac{1}{\sqrt{2\pi}}\int_{\mathbb{R}}\widetilde{h}^\lambda(y+\sqrt{1-s}\cdot x)e^{-\frac{1}{2}x^2}dx,\ (s,y)\in[0,1]\times\mathbb{R}, $$
and, obviously, $E^Q[L^{\lambda}_{\varepsilon}\,|\,\mathcal{F}_s]=g_2^\lambda(s,B_s)$. Note that $g_2\ge\frac{\varepsilon}{\varepsilon+C_\varepsilon}$ (See Step 5).  We conclude that $\gamma_s^{\lambda,\varepsilon}= g^\lambda(s,B_s):= \frac{g_1^\lambda(s,B_s)}{g_2^\lambda(s,B_s)}$, $s\in[0,1]$, where $g:[0,1]\times\mathbb{R}\times\mathbb{R}\rightarrow\mathbb{R}$ is continuous, and $g(s,\cdot,\cdot)\in C^\infty(\mathbb{R}\times\mathbb{R})$ with derivatives of all order bounded over $[0,1]\times\mathbb{R}\times\mathbb{R}.$

In order to complete the proof of the proposition, we have still to approximate in a suitable way the process $\gamma^{\lambda,\varepsilon}$ by a smooth Wiener step process.

\noindent\textbf{Step 7}. For arbitrarily fixed $\varepsilon>0$
 we approximate $\gamma_s^{\lambda}:=\gamma_s^{\lambda,\varepsilon}=g^\lambda(s,B_s),\, s\in[0,1]$, introduced in the preceding step, by
$$\tilde{\gamma}^{\lambda,k}_s:=\sum_{i=1}^kg^\lambda(s_{i-1}^k,B_{s_{i-1}^k})I_{\Delta_{i,k}}(s)\Big(=\sum_{i=1}^kg^\lambda\big(s_{i-1}^k,B(\Delta_{1,k})+\cdots+B(\Delta_{i-1,k})\big)I_{\Delta_{i,k}}(s)\Big),$$
where, for $k\ge 1$, $s^k_i:=i/k,\, 0\le i\le k,$ and
$\Delta_{i,k}:=(s_{i-1}^k,s_i^k],\, 1\le i\le k.$ From the regularity properties of the function $g(s,\lambda,x)=g^\lambda(s,x),\, (s,\lambda,x)\in[0,1]\times\mathbb{R}\times\mathbb{R}$, stated in Step 6 we see that $\tilde{\gamma}^{\lambda,k}$ is a smooth Wiener step process. Moreover, as $g$ and its derivatives w.r.t. $(\lambda,x)$ are bounded over $[0,1]\times\mathbb{R}\times\mathbb{R}$, we check easily that for the Dol\'ean-Dade exponential
$$\mathcal{E}_1^{\lambda,k}=\exp\Big\{ \int_0^1\tilde{\gamma}^{\lambda,k}_sdB_s-\frac12 \int_0^1|\tilde{\gamma}^{\lambda,k}_s|^2ds\Big\},$$
it holds $(\mathcal{E}_1^{\lambda,k},\partial_\lambda\mathcal{E}_1^{\lambda,k})\rightarrow(L_\varepsilon^\lambda,\partial_\lambda L_\varepsilon^\lambda)$ in $L^2(Q)$ and $(\mathcal{E}_1^{{\lambda(\cdot)},k},\partial_\lambda\mathcal{E}_1^{ {\lambda(\cdot)},k})\rightarrow(L_\varepsilon^{\lambda(\cdot)},\partial_\lambda L_\varepsilon^{\lambda(\cdot)})$ in $L^2([0,1]\times\Omega,dsdQ)$, as $k\rightarrow+\infty.$

Finally, combining the Steps 1-7 we finish the proof of the proposition.
\end{proof}

\end{document}